\documentclass[11pt]{amsart}
\usepackage{amssymb}
\usepackage{graphicx}
\usepackage{xcolor} 
\usepackage{tensor}
\usepackage{fullpage} 
\usepackage{tikz}

\usepackage{amsmath}
\usepackage{amsthm}
\usepackage{verbatim}
\usepackage{hyperref}
\usepackage{array} 
\usepackage{enumitem}

\setlist[enumerate]{leftmargin=1.8em}
\setlist[itemize]{leftmargin=1.8em}
\usepackage{seqsplit,mathtools}

\definecolor{green}{rgb}{0,0.8,0} 

\newtheorem{theorem}{Theorem}[section]

\newtheorem{lemma}[theorem]{Lemma}

\newtheorem{proposition}[theorem]{Proposition}

\theoremstyle{definition}

\theoremstyle{remark}
\newtheorem{remark}[theorem]{Remark}

\numberwithin{equation}{section}
\newcommand{\nrm}[1]{\Vert#1\Vert}

\newcommand{\tld}[1]{\widetilde{#1}}

\newcommand{\nnrm}[1]{{\vert\kern-0.25ex\vert\kern-0.25ex\vert #1 
		\vert\kern-0.25ex\vert\kern-0.25ex\vert}}

\newcommand{\supp}{{\mathrm{supp}}\,}

\newcommand{\lap}{\Delta}

\newcommand{\rd}{\partial}
\newcommand{\nb}{\nabla}

\newcommand{\alp}{\alpha}
\newcommand{\bt}{\beta}

\newcommand{\dlt}{\delta}

\newcommand{\eps}{\varep}

\newcommand{\kpp}{\kappa}
\newcommand{\lmb}{\lambda}

\newcommand{\omg}{\omega}

\newcommand{\varep}{\varepsilon}

\newcommand{\bfe}{{\bf e}}

\newcommand{\bfG}{{\bf G}}

\newcommand{\bfK}{{\bf K}}

\newcommand{\bbR}{\mathbb R}

\newcommand{\calA}{\mathcal A}

\newcommand{\ohp}{\mathbb{R}^{2}_{+}}

\definecolor{purple}{rgb}{0.65, 0, 1}
\definecolor{orange}{rgb}{1,.5,0}

\begin{document}

	\title{Lamb dipole stability without sign condition}

	\author{Ken Abe}
	\address{Department of Mathematics, Graduate School of Science, Osaka Metropolitan University, 3-3-138 Sugimoto, Sumiyoshi-ku Osaka, 558-8585, Japan.}
	\email{kabe@omu.ac.jp}
	
	\author{Kyudong Choi}
	\address{Department of Mathematical Sciences, Ulsan National Institute of Science and Technology, 50 UNIST-gil, Eonyang-eup, Ulju-gun, Ulsan 44919, Republic of Korea.}
	\email{kchoi@unist.ac.kr}
	
	\author{In-Jee Jeong}
	\address{School of Mathematics, Korea Institute for Advanced Study, Seoul 02445, Republic of Korea.}
	\email{ijeong@kias.re.kr}
	
	\author{Guolin Qin}
	\address{State Key Laboratory of Mathematical Sciences, Academy of Mathematics and Systems Science, Chinese Academy of Sciences, 100190 Beijing, P.R. China}
	\email{qinguolin18@mails.ucas.ac.cn}

	\date{\today}

	\renewcommand{\thefootnote}{\fnsymbol{footnote}}
	\renewcommand{\thefootnote}{\arabic{footnote}}

	\begin{abstract} 
		For the 2D incompressible Euler equations on the half-plane, we establish Lyapunov stability of the Lamb dipole without sign condition on the vorticity. The main challenge arises from growth of the $x_{2}$ weighted norm for mixed sign vorticities, whereas this norm is conserved and responsible for dipole stability in the single signed vorticity case. This difficulty is handled by a Lagrangian bootstrap argument, based on carefully dividing the solution into several pieces and analyzing interactions among them, with the most delicate interaction involving vorticity drifting to infinity in the vertical direction. We expect that the strategy developed here will provide a starting point for the stability analysis of more general vortex configurations, including interactions among multiple vortex dipoles and among multiple vortex rings in settings involving mixed-sign vorticity.
	\end{abstract}
	
	\maketitle  
	
	\date{\today}

	\section{Introduction}\label{sec:intro}

	We study the vorticity evolution for incompressible and inviscid fluid in the upper half-plane $\bbR^2_+ = \left\{ x  : x_{2} \ge 0 \right\}$: the vorticity $\omg$ satisfies \begin{equation}\label{eq:2D-Euler}
		\left\{
		\begin{aligned}
			\rd_t \omg + u \cdot \nb \omg = 0, & \\
			u = -\nb^{\perp} (-\lap)^{-1} \omg ,&
		\end{aligned}
		\right.
	\end{equation}  where $\nb^\perp = (-\rd_{x_2}, \rd_{x_1})$ and $(-\lap)^{-1}$ is the inverse negative Laplacian  in $\bbR^2_+$. The half-plane problem \eqref{eq:2D-Euler} is equivalent with studying vorticity defined in the whole plane $\bbR^2$ with odd symmetry in $x_{2}$. For this reason, the half-plane problem provides a simple setup for studying \textit{vortex dipoles}, a pair of counter-rotating vortex propagating along the $x_{1}$ axis. 
	
	In the mathematical literature, vortex dipoles are usually studied by showing existence of \textit{traveling waves} to \eqref{eq:2D-Euler} and then their dynamical stability. However, all the existing nonlinear stability results for traveling waves requires the \textbf{sign condition} $\omg \ge 0$ to hold in $\bbR^2_+$, which is not satisfactory from a physical point of view. This restriction can be attributed to the following disparity: while the key quantity responsible for vortex dipole stability is the weighted norm \begin{equation}\label{eq:weight}
		\begin{split}
			\nrm{\omg}_{L^1_*(\bbR^2_+)} := \int_{\bbR^2_+} x_{2}|\omg(x)| dx, 
		\end{split}
	\end{equation} available conserved quantity is the \textit{vorticity impulse}, which is \eqref{eq:weight} without the absolute value: 
	\begin{equation}\label{eq:impulse}
		\begin{split}
			\mu[\omg] := \int_{\bbR^2_+} x_{2}\omg(x)\, dx. 
		\end{split}
	\end{equation} (One may observe that \eqref{eq:impulse} equals the first component of the linear momentum, $\int_{\bbR^2_+} u_{1}(x) dx$, so that $\mu[\omg] > 0$ shows tendency of fluid particles to drift to the right.) Under the sign assumption $\omg \ge 0$, these two quantities coincide and control the amount of vorticity far from $\partial\bbR^2_+$. On the other hand, as soon as the initial vorticity $\omg_{0}$ is allowed to have a small negative part, it is expected that the weighted norm \eqref{eq:weight} grows linearly in time for the solution of \eqref{eq:2D-Euler}, which ruins the existing stability arguments. These technical issues will be discussed in detail below, \S \ref{subsec:discuss}. 
	
	In the current work, we obtain nonlinear stability of the classical Lamb dipole, without the sign condition on the initial data. The idea is to view this stability issue as an \textit{interaction problem} between various parts of the solution, where the most challenging interaction comes from negative signed vorticity traveling to $x_{2}\to\infty$. This is nontrivial not just because the (initially small) negative part grows in the norm \eqref{eq:weight}, but also because the interaction kernel has singularity and decays rather slowly. 
	
	\subsection{The Lamb dipole and main result} The Lamb (or Chaplygin--Lamb) dipole is one of the most well-known traveling wave of 2D Euler, discovered independently by Lamb and Chaplygin (\cite{Lamb,Chap1903}) more than a century ago. We consider the following \textit{normalized Lamb dipole} $\omg_{Lamb}$ on $\mathbb{R}^{2}_+$: 
	\begin{equation}\label{eq:Lamb-def}
		\begin{split}
			\omg_{Lamb}(x) := \left(\frac{2c_{L}}{- J_{0}(c_{L}) }\right) \frac{  x_{2} J_{1}(c_{L}|x|)}{|x| }  \, \mathbf{1}_{ \{ |x| \le 1\} } (x),\quad x\in\mathbb{R}^2_+.
		\end{split}
	\end{equation} Here, $J_m(r)$ is the $m$-th order Bessel function of the first kind and $c_{L}\simeq 3.8317$ is the first positive zero of $J_1(r)$. It is normalized in the sense that $\omg_{Lamb}$ is supported on the unit half disk $B^+_1:=\{x\in\mathbb{R}^2_+\,:\, |x| \le 1 \},$ and defines a traveling wave of \eqref{eq:2D-Euler} in $\bbR^2_+$ with unit speed in the $x_{1}$-direction; $\omg_{Lamb}(x - t\mathbf{e}_{1} )$ provides a solution to  \eqref{eq:2D-Euler}. More detailed properties of $\omg_{Lamb}$ are given in \S \ref{subsec:Lamb} below.
	
	\begin{figure}
		\centering
		\includegraphics[width=0.4\linewidth]{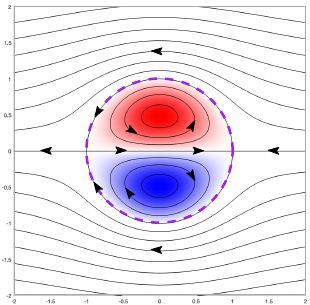}
		\caption{The Lamb dipole and its velocity in the moving frame (Figure from \cite{AJY})}
		\label{fig:lamb}
	\end{figure}
	
	The nonlinear stability of Lamb dipole was proved in \cite{AC2019} under the {sign condition} $\omg_{0} \ge 0$ in $\bbR^2_+$. Our main result remove this assumption, at the cost of assuming some decay \eqref{assump_moment} for the data. For simplicity, let us write $\nrm{\rho}_{X \cap Y} = \nrm{\rho}_{X} + \nrm{\rho}_{Y}$ for two norms $X$ and $Y$.   \begin{theorem}  \label{thm:main}
		Let $A > 0$. 
		For any  $\varep > 0$, there exist constants $C=C(A) > 1,$ $ \dlt = \dlt(A,  \epsilon)>0$ such that the following stability holds: for any  $\omg_{0} \in (L^{1} \cap L^{\infty}) (\bbR^2_+)$ satisfying \begin{equation}\label{eq:mainthm-assumptions}
			\begin{split}
				\nrm{\omg_{0} -  {\omg}_{Lamb}}_{(L^{2} \cap L^{1}_{*})(\ohp)} < \dlt,\quad 
				\nrm{\omg_{0} }_{ L^{\infty}(\ohp)} \le A ,\quad\mbox{and}
			\end{split}
		\end{equation} 
		\begin{equation}\label{assump_moment}
			\int_{\{x\in\mathbb{R}^2_+\,:\,|x|>1\}}|x||\omg_0(x)|dx<\dlt,
		\end{equation}
		the unique solution $\omg(t,\cdot)$ to \eqref{eq:2D-Euler} for the initial data $\omega_0$ satisfies \begin{equation}\label{eq:mainthm-conclusion}
			\begin{split}
				\nrm{\omg(t,\cdot +p_{1}(t) \bfe_{1} ) -  {\omg}_{Lamb}( \cdot )}_{(L^{2}\cap L^1)(\ohp) \cap L^{1}_{*}{( \{ |x_{1} | \le C
						\varep^{-1} \}  )}} < \varep,\quad t\in\mathbb{R},
			\end{split}
		\end{equation} for some Lipschitz function  $p_{1}:\mathbb{R}\to\mathbb{R}$  satisfying $|p_{1}(t) - t| \le \varep^{1/2}(1+|t|)$ for all $t \in \bbR$.  
	\end{theorem}

	Theorem \ref{thm:main} seems to be the first stability result of traveling waves to the incompressible Euler equations in $\bbR^2_+$ which does not require the sign condition on the vorticity. The strategy developed in this paper is robust, and we expect that the result can be extended to a more general family of vortex dipoles and rings. Furthermore, based on our approach one may study the interaction problem of multiple vortex dipoles with \textit{mixed} signs; see \S \ref{subsec:pa} for details. 
	
	Compared to the Lamb dipole stability \textit{with} sign condition $\omg_{0} \ge 0$, Theorem \ref{thm:main} requires several additional assumptions, whose necessity is discussed in \S \ref{subsec:discuss} after precisely stating the stability result with sign condition from \cite{AC2019,ACJ}.
	
	The proof is based on combining a variational principle (the Lamb dipole is the kinetic energy maximizer under appropriate constraints on vorticity) with \textit{Lagrangian bootstrapping schemes}, where we study the fluid particle trajectories to quantify smallness of interactions among various parts of the solution. This framework was developed in the recent work \cite{AJY}, which obtained the stability of linear superpositions of sufficiently separated Lamb dipoles under the sign condition. In the present work, we face all the same difficulties in \cite{AJY} but several more, due to absence of the sign assumption. These technical challenges are discussed in detail in \S \ref{subsec:ideas}. To demonstrate that our framework is not limited to the study of Lamb dipoles, several applications will be presented in \S \ref{subsec:pa}. We then discuss existing stability results for various traveling waves and other equilibria for the two-dimensional vorticity equation in \S \ref{subsec:refs}. Furthermore, we present several remaining open problems. 
	
	
	
	We close this section with a simple remark. 
	
	\begin{remark}[Stability for rescaled Lamb dipoles]
		While Theorem \ref{thm:main} is stated for the normalized Lamb dipole for simplicity, a similar statement holds (by a suitable rescaling of $C$ and $\dlt$) for the whole two-parameter family of Lamb dipoles, denoted by $\omg_{Lamb}^{\kpp,\mu}$: 
		\begin{equation}\label{eq:lamb-rescale-def}
			\omega_{Lamb}^{\kpp,\mu}(x) := \left(\frac{\kappa^{3}}{ \sqrt{\pi} c_{L}^{3} \mu}\right)^{1/2} \, {\omega}_{Lamb}\left( \left( \frac{c_{L}\mu}{\sqrt{\pi}\kpp} \right)^{-1/2} x\right) .
		\end{equation}
		The rescaling is defined so that we have\footnote{With this notation,  $\omg_{Lamb} = \omg_{Lamb}^{c_{L}\sqrt{\pi}, \pi}$.} $\nrm{\omega_{Lamb}^{\kpp,\mu}}_{L^2(\ohp)} = \kpp$ and $\nrm{\omega_{Lamb}^{\kpp,\mu}}_{L^1_*(\ohp)} = \mu$.
	\end{remark}

	\subsection{Comparison with Lamb dipole stability under sign condition }\label{subsec:discuss} In this section, we first recall the Lamb dipole stability statement under sign condition (\cite{AC2019,ACJ}). 
	Given a function $\rho$ on $\bbR^2_+$, we write $\kpp[\rho] := \nrm{\rho}_{L^2(\bbR^2_+)}$ and recall that the kinetic energy is given by \begin{equation*}
		\begin{split}
			E[\rho] := \frac1{8\pi} \iint_{\bbR^2_+ \times \bbR^2_+} \log\left( 1 + \frac{4x_2y_2}{|x-y|^2} \right) \rho(x)\rho(y) dxdy = \frac12 \int_{\bbR^2_+} \rho(x) (-\lap_{\bbR^2_+})^{-1}\rho (x) dx. 
		\end{split}
	\end{equation*} All the $L^p$ norms and $E[\rho]$ are preserved in time for Yudovich solutions of \eqref{eq:2D-Euler}. 
	
	The fundamental property of the Lamb dipole is that it is the unique energy maximizer under constraints on ${L^1_*}$ and $L^{2}$: 
	\begin{proposition}[{{\cite{Burton05b}, \cite[Corollary 2.5]{AJY}, and \cite[Theorem 1.4]{ACJ}}}]\label{prop:en_eq}
		For any  $\rho\in (L^2\cap L^1_{*})(\mathbb{R}^2_+)$, we have the sharp energy inequality
		\begin{align} \label{eq: SEI}
			E[\rho]
			\leq C_L \nrm{\rho}_{L^{2}(\mathbb{R}^{2}_{+})}\nrm{\rho}_{L^{1}_{*}(\mathbb{R}^{2}_{+})}, \quad \mbox{where} \quad C_L:=(c_L\sqrt{\pi})^{-1} \simeq 0.1502.
		\end{align}    
		If $\rho$ is nonzero, the equality is achieved only when either $\rho$ or $-\rho$ is equal to some $x_{1}$-translate of the two-parameter family of Lamb dipoles defined in \eqref{eq:lamb-rescale-def}.
		
		Under $\omg \ge 0$, we have $\nrm{\omg}_{L^1_*} = \mu[\omg]$, and the above implies that in the admissible class \begin{equation*}
			\begin{split}
				\calA_{+} := \left\{  \omg \ge 0 \, : \, \kpp[\omg] \le  \kpp[{\omg}_{Lamb}], \, \mu[\omg] \le  \mu[{\omg}_{Lamb}] \right\},
			\end{split}
		\end{equation*} the supremum of $E$ is achieved precisely for $x_{1}$-translations of $\omg_{Lamb}$.  
	\end{proposition}
	
	Together with a compactness statement which gives convergence of any energy maximizing sequences in the admissible class $\calA_{+}$, one obtains nonlinear orbital stability  \cite{AC2019,ACJ}.	 
	\begin{theorem}[\cite{ACJ}] \label{thm:Lamb-sign} For any $\varep>0$, there exists $\dlt>0$ such that the following holds: \begin{equation}\label{cond_sign}
			\mbox{for any}\quad \omg_{0} \ge 0\quad\mbox{in}\quad\ohp
		\end{equation} satisfying \begin{equation*}
			\begin{split}
				\nrm{\omg_{0} -  {\omg}_{Lamb}}_{L^{2} } + \left|\int_{\bbR^2_+} x_{2}\left( \omg_{0} - \omg_{Lamb} \right) dx \right| < \dlt,
			\end{split}
		\end{equation*} there exists a solution $\omg \in L^{\infty}(\bbR; (L^{2} \cap L^{1}_{*})(\ohp))$ corresponding to the initial data $\omg_{0}$ such that 
		\begin{equation}\label{eq:mainthm-conclusion-sign}
			\begin{split}
				\nrm{\omg(t,\cdot + \tau(t) \bfe_{1} ) -  {\omg}_{Lamb}( \cdot )}_{(L^{2} \cap L^{1}_{*})(\ohp)} < \varep \qquad \mbox{for all} \qquad t \in \mathbb{R} 
			\end{split}
		\end{equation} holds 
		for some shift function $\tau:\mathbb{R}\to\mathbb{R}$.
	\end{theorem}

	We are now in a position to explain the differences between the stability statements in Theorem \ref{thm:main} and Theorem \ref{thm:Lamb-sign}: 
	
	\smallskip 
	
	\noindent \textbf{Closeness in $O(\varep^{-1})$ strip}. Our conclusion  \eqref{eq:mainthm-conclusion} gives $L^2\cap L^1$ closeness of the solution (after a shift) to the Lamb dipole globally in space, but the $L^1_*$ closeness is given only in a strip of width $O(\varep^{-1})$, unlike \eqref{eq:mainthm-conclusion-sign}. We believe that \eqref{eq:mainthm-conclusion-sign} is simply \textbf{false} in our setting, without the sign condition \eqref{cond_sign}: Imagine a perturbation of $\omg_{Lamb}$ which is given by a very small Lamb dipole traveling in the \textit{vertical} direction (cf. Figure \ref{fig:difficulty}). Assuming that the interaction between this perturbation and the main Lamb dipole is negligible, the $L^{1}_*$ norm of this part diverges linearly in time, not contradicting any of the conservation laws. This growth of $L^{1}_*$ norm is the most serious difficulty in this work. By the same reasoning, we do not expect a global-in-time control of 
	$$ \int_{\{x\in\mathbb{R}^2_+\,:\,|x|>1\}}|x||\omg(t,\cdot+p_1(t)\bfe_{1})|dx$$ even though this quantity is small at the initial time \eqref{assump_moment}. 
	
	\smallskip 
	
	\noindent \textbf{Moment smallness condition}. Regarding the assumptions on the data, the main difference is that while the sign condition \eqref{cond_sign} in Theorem \ref{thm:Lamb-sign} is removed, we introduce a moment smallness condition \eqref{assump_moment}. The necessity of \eqref{assump_moment} comes from dealing with interaction of the ``Lamb dipole part'' with \textit{opposite-signed} dipoles coming from infinity, see Figure \ref{fig:difficulty}. Even if such dipoles are altogether small in $L^{p}$ for all $p$, depending on their initial locations, they can interact adversely for a long time  with the Lamb dipole part: a perturbation initially distance $D$ away from the Lamb dipole can interact with the Lamb dipole for a timescale of order $O(D)$. In this regard, \eqref{assump_moment} ensures that its size is bounded by $o(D^{-1})$, balancing the interaction timescale. 
	
	\smallskip 
	
	\noindent \textbf{Dependence of $\dlt$ in the $L^{\infty}$ of vorticity}. Unlike Theorem \ref{thm:Lamb-sign}, in our result $\delta$ depends not just on $\varepsilon$ but also on $\nrm{\omg_{0}}_{L^\infty(\ohp)}$. This is because our proof of stability is not just ``kinematic'' (consequence of frozen-in-time estimates) but depends crucially on dynamics of fluid particles; we need a bound on the maximum velocity, which comes from having a bound on $\nrm{\omg_{0}}_{L^\infty(\ohp)}$.\footnote{To this end, any $L^{p}$ norm with $p>2$ can do the same job.}

	\begin{figure}
		\centering
		\includegraphics[width=0.8\linewidth]{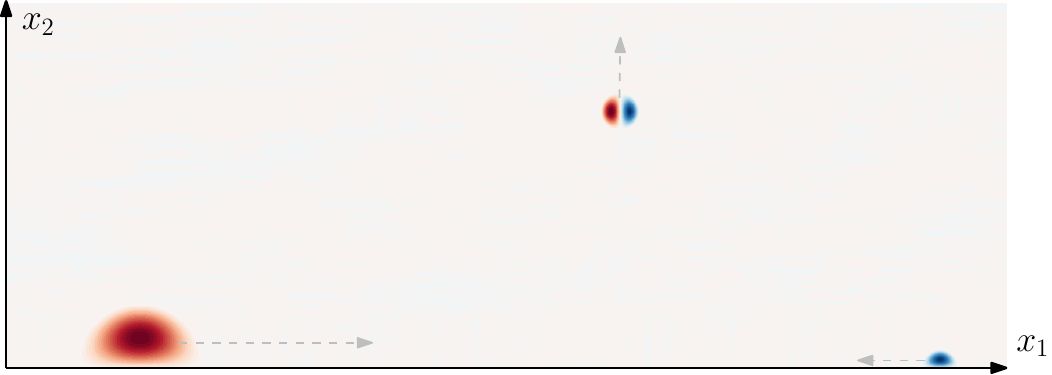}
		\caption{Challenges in proving Lamb dipole stability without sign condition. The positive and negative vorticity regions are represented by red and blue colors, respectively. The propagation velocities (assuming small interaction) are depicted in gray dashed lines.}
		\label{fig:difficulty}
	\end{figure}

	\subsection{Proof strategy}\label{subsec:ideas}
	
	At a high level, the idea in the proof of Theorem \ref{thm:main} is to apply the variational principle not to the whole solution $\omg(t)$ but rather to some \textit{part} of the solution, namely $\omg(t)$ restricted to some subset of $\bbR^2_+$. The immediate difficulty one faces is that all the key quantities (impulse, enstrophy, energy) change in time, due to interactions with the rest of the solution. 
	
	A framework introduced in \cite{AJY} allows one to quantify smallness such interactions, by setting up a number of bootstrap hypotheses which are closed by tracking the space-time location of fluid particles. To explain this framework, and also to expand on \textit{additional} challenges arising in this work, let us recall the statement of multiple Lamb dipole stability, in the simplest case of two Lamb dipoles \cite{AJY} (with modified notation for consistency). 
	
	\begin{theorem}\label{thm:Lamb-dipoles}
		Fix $A>0$ and ${\kpp_{r}, \kpp_{\ell}},{\mu_{r}},\mu_{\ell} > 0$, with $ \kpp_{r} >  \kpp_{\ell}.$  Then, for any $\varepsilon>0$ , there exist $\dlt > 0, D_{0} \ge \dlt^{-1}$ depending on $\varepsilon, A, {\kpp_{r}, \kpp_{\ell}},{\mu_{r}},\mu_{\ell}$ such that as long as $\bar{p}_{r}  > \bar{p}_{\ell}  + D_{0}$,  we have the following stability result. For any nonnegative initial data $\omg_{0} \ge 0$ on $\bbR^{2}_{+}$ with \begin{equation}\label{eq:ini-assumption}
			\begin{split}
				\left\Vert \omg_{0} -  \left( \omega_{Lamb}^{\kpp_{r},\mu_{r}}(\cdot - \bar{p}_{r} \bfe_{1} ) + \omega_{Lamb}^{\kpp_{\ell},\mu_{\ell}}(\cdot - \bar{p}_{\ell} \bfe_{1} ) \right) \right\Vert_{L^{2} \cap L^1_*} \le \dlt, \quad \nrm{\omg_{0}}_{L^{1} \cap L^{\infty}},\, |\supp(\omega_0)|\le A, 
			\end{split}
		\end{equation} the corresponding solution $\omg(t,\cdot)$ to $\omg_{0}$ satisfies \begin{equation}\label{eq:varep-close}
			\begin{split}
				\left\Vert \omg(t,\cdot) -  \left( \omega_{Lamb}^{\kpp_{r},\mu_{r}}(\cdot - {p}_{r}(t) \bfe_{1} ) + \omega_{Lamb}^{\kpp_{\ell},\mu_{\ell}}(\cdot -  {p}_{\ell}(t) \bfe_{1} ) \right)  \right\Vert_{L^{2} \cap L^1_*  } \le \varepsilon \qquad\mbox{for all} \qquad t > 0 
			\end{split}
		\end{equation} for some $ (p_{r}(t),p_{\ell}(t))$ satisfying with some  $C = C(A, {\kpp_{r}, \kpp_{\ell}},{\mu_{r}},\mu_{\ell})>0$ \begin{equation*}
			\begin{split}
				| p_{i}(t) - \bar{p}_{i} - C_{L}\kpp_{i}\, t| \le C\varep^{1/2}(1+t), \qquad i \in \left\{ r, \ell \right\}.
			\end{split}
		\end{equation*} 
	\end{theorem}
	
	\smallskip \noindent \textit{Borderline and bootstrap hypotheses}. In the result above, the subscript $r$ and $\ell$ stands for right and left, respectively. Under the assumption that each Lamb dipole travels with its own velocity, which is $C_{L}$ times the $L^{2}$ norm, it is natural to take the borderline $L(t)$ to be the ``midpoint'' (Figure \ref{fig:dipoles}) \begin{equation*}
		\begin{split}
			L(t) = \frac{(\bar{p}_{r} + \bar{p}_{\ell}) + (C_L \kpp_{r} + C_{L} \kpp_{\ell})t}{2} 
		\end{split}
	\end{equation*} and decompose the solution by $\omg_{r}(t) := \omg(t) \mathbf{1}_{\{x_1> L(t)\}}$ and $\omg_{\ell}(t) = \omg(t) - \omg_{r}(t)$. Then, we define the quantities \begin{equation*}
		\begin{split}
			\kpp_{i}(t) = \kpp[ \omg_{i}(t) ], \quad \mu_{i}(t) = \mu[ \omg_{i}(t) ], \quad E_{i}(t) = E[ \omg_{i}(t) ]
		\end{split}
	\end{equation*} for $i \in \{ r, \ell \}$. The bootstrap hypotheses used in \cite{AJY} (modulo some simplifications) were: \begin{equation}\label{eq:bootstrap-AJY}
		\begin{split}
			|\kpp_{r}(t) - \kpp_{r}| \le C\dlt, \quad |\mu_{r}(t) - \mu_{r}| \le C\dlt, \quad |E_{r}(t) - C_{L}\kpp_{r}\mu_{r}| \le C\dlt, 
		\end{split}
	\end{equation} and similarly for the left part, for some $C>0$. We recall that $C_{L}\kpp_{r}\mu_{r} = E[ \omg_{Lamb}^{\kpp_r,\mu_r} ]$. These hypotheses, if proven, gives the stability \eqref{eq:varep-close} by the variational principle.

	\smallskip \noindent \textit{Quadratic decay of the Biot--Savart kernel}. The first observation to make is that under the {support confinement} hypothesis, \eqref{eq:bootstrap-AJY} can be closed: more precisely, assume that the supports of $\omg_{r}(t)$ and $\omg_{\ell}(t)$ completely separate from each other by distance $\gtrsim D_{0} + t \ge \dlt^{-1} + t$. In this case, one has from decay of the Biot--Savart kernel in $\bbR^2_+$ that \begin{equation}\label{eq:ODE-naive}
		\begin{split}
			\left| \frac{d}{dt} \kpp_{r}(t) \right| + 	\left| \frac{d}{dt} \mu_{r}(t) \right| + 	\left| \frac{d}{dt} E_{r}(t) \right| \lesssim \frac{1}{ (\dlt^{-1} + t)^2 } .
		\end{split}
	\end{equation} Integrating the above from $t = 0$ to $\infty$ shows that the maximum total variation of the enstrophy, impulse, and energy is  $O(\dlt)$. This is observed in \cite{ILL2003} for point vortices in $\bbR^2_+$. In the case of multiple Lamb dipoles, one sees that the interaction among the ``bulk'' vortices are therefore negligible. As we shall explain now, all the technical challenges arise from the interaction of ``errors.'' (This also shows why studying stability of the Lamb dipole without the sign condition retains all the difficulties from the case of multiple Lamb dipoles under the sign condition.)

	\smallskip \noindent \textit{Filamentation: vorticity transfer from front to rear}. In reality, the support confinement is not expected to hold, even when the initial data is \textit{exactly} the superposition of two Lamb dipoles far away from each other. The reason is generic occurrence of \textbf{filamentation}, or creation of long and thin tails, behind each Lamb dipole. This occurs simply because the velocity field generated by each Lamb dipole ``flushes away'' all the fluid outside its support; $v_{Lamb,1} \le 0$ in $\bbR^2_+\backslash B_{1}$ where $v_{Lamb}$ is the velocity of the Lamb dipole in the moving frame; see Figure \ref{fig:lamb}. This makes two Lamb dipoles eventually ``connected'' with each other, since the filaments travel very slowly. Furthermore, it tends to decrease $\kpp_{r}(t), \mu_{r}(t), E_{r}(t)$.
	
	The point here is that, in reality in the right-hand side of \eqref{eq:ODE-naive}, one has contribution from particles that have filamented away before time $t$. While this contribution is easily shown to be small for each fixed time, the highly nontrivial task is to show that it is still $O(\dlt)$ small after integrating from time $0$ to $\infty$. 
	
	\smallskip \noindent \textit{Lift-up: non-local impulse transfer from rear to front}. The analysis is significantly complicated by the fact that the following dipole $\omg_{\ell}$ \textbf{lifts up} the front dipole $\omg_{r}$, which is a \textit{non-local} effect unlike filamentation. This can be also seen from the Lamb velocity: $v_{Lamb,2} \ge 0$ for $x_{1} \ge 0$. In particular, $\mu_{r}(t)$ decreases from filamentation but increases from this lift-up effect from $\omg_{\ell}(t)$. The key computation in \cite{AJY} is that the total amount of lift-up is $O(\dlt)$ small: this is based on switching to Lagrangian coordinates and integrating the contribution to lift-up during the ``lifetime'' of a particle trajectory (i.e. $t \in [0,\infty)$) over all particles which move from $\{ x_{1} > L(t) \}$ to $\{ x_{1} \le L(t) \}$. They are referred to as \textit{gain} particles. 
	
	\smallskip \noindent \textit{Filamentation energy favorable regime: energy trapping}. Lastly, we need to recall that in the case of single Lamb dipole, we can deduce smallness of total filamentation thanks to variational principle; if too much vorticity filaments away, it contradicts energy conservation thanks to the sharp energy inequality \eqref{eq: SEI}. It is not known how to get this smallness \textit{dynamically}. In this regard, the case of two Lamb dipoles faces a new challenge, which is the presence of \textbf{filamentation energy favorable regime}. Large filamentation is energy unfavorable (that is, contradicts total conservation of energy) only under the additional condition that the radius of $\omg_{Lamb}^{\kpp_r,\mu_r}$ is larger than that of $\omg_{Lamb}^{\kpp_\ell,\mu_\ell}$. This potentially problematic large energy escape from $\omg_{r}(t)$ to $\omg_{\ell}(t)$ is ruled out by an ``energy trapping'' computation, which directly links $\frac{d}{dt} E_{r}(t)$ with $\frac{d}{dt} \mu_{r}(t)$ by comparing the relevant interaction kernels; it turns out that, under the above bootstrap hypotheses, energy change has to be much smaller than impulse change. This dynamically (and not kinematically) rules out large energy escape. 
	
	\begin{figure}
		\centering
		\includegraphics[width=0.6\linewidth]{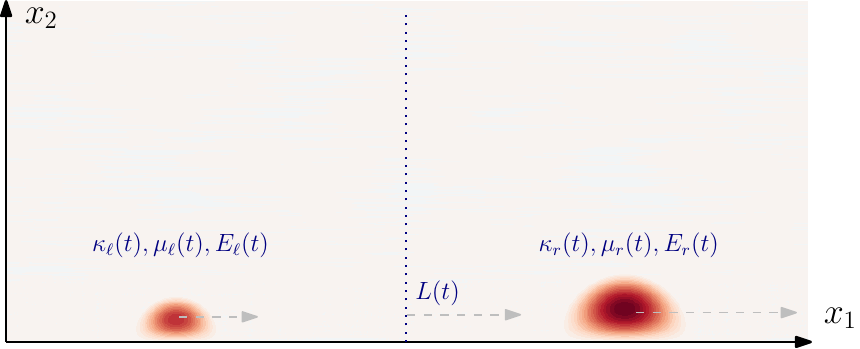}
		\caption{Decomposition in \cite{AJY} for stability of two Lamb dipoles}
		\label{fig:dipoles}
	\end{figure}
	
	\smallskip 
	
	All of these difficulties arise in the same way for our problem, and now we are ready to present additional difficulties due to having mixed-sign vorticity. 
	
	\medskip \noindent \textbf{Further difficulties in the mixed-sign case and resolutions}. To begin with, we recall the key difficulty: while the variational principle using impulse and enstrophy is the only known way to prove nonlinear stability of the Lamb dipole, in the mixed-sign case, the impulse is not coercive anymore. Trying to keep using the impulse by further restricting the admissible class of vorticities would not work, since it seems that the Lamb dipole fails to be \textit{local} energy maximizers even for localized, zero-impulse mixed-sign perturbations. 
	Therefore, we use the norm $\nrm{\omg}_{L^1_*}$ in the admissible class, and face the difficulty that it grows in time.

	\smallskip \noindent \textit{Weighted norm growth: tight borderline in the moving frame}. Our approach is first to decompose the solution by a borderline, as in the stability proof of two Lamb dipoles. In that case, the borderline $L(t)$ defining the decomposition was at least $O(\dlt^{-1}+t)$ far from the Lamb dipole parts. In our case, since we would like to apply the variational principle to the part right of the borderline, we face a new restriction: the norm $\nrm{\omg}_{L^1_*}$ should remain essentially the same in this region. Hence, we cannot take the borderline too far from the Lamb dipole: otherwise, $\nrm{\omg}_{L^1_*}$ will grow too much, destroying the variational characterization of the Lamb dipole in the mixed-sign case. However, it cannot be too close either; then the right part of the solution will interact too much with the left part, destroying the bootstrapping hypotheses. 
	
	Surprisingly, it turns out that there is a \textbf{unique} choice of such a borderline (up to a multiplicative constant), which is to take it $D := K \dlt^{-1/2}$ away in $x_{1}$ from an approximate Lamb dipole location, for a sufficiently large 
	constant $K>0$ depending only on $A$.  
	Unlike the two Lamb dipole case, the borderline needs to have this fixed distance from the Lamb dipole part; to do this systematically we introduce the moving frame and change variables ($Q$ stands for quadrant): \begin{equation*}
		\begin{split}
			\rho(t,x) = \omg(t, x - p_{1}(t) \bfe_1), \quad Q = \left\{ x \, : \, x_{1} \ge -K \dlt^{-1/2} \right\}, \quad \rho_{Q}(t) = \rho(t)\mathbf{1}_{Q}, \quad \tld{\rho}(t) = \rho(t) - \rho_{Q}(t). 
		\end{split}
	\end{equation*}

	\smallskip \noindent \textit{The additional bootstrap hypotheses}. Let us now introduce our list of bootstrap hypotheses; with $\kpp_{Q}(t) = \nrm{\rho_{Q}(t)}_{L^2}, |\mu|_Q(t) = \nrm{\rho_Q(t)}_{L^1_*}$, $E_Q(t) = E[\rho_{Q}(t)]$ and some $\lmb = O(\dlt)$ to be chosen: \begin{enumerate}
		\item $\nrm{ \rho_{Q}(t,\cdot) - \omg_{Lamb} }_{L^2 \cap L^1_*} \leq  {\frac{\varep}{10}}$ after a Lipschitz-in-time shift $p_{1}(t)$. 
		\item ${\int_{Q}x_2|\rho(t,x)|\mathbf{1}_{\Phi(t,A^{far,t})}(x) dx\leq\frac{\lambda}{10}.		}$
		\item $\kpp_{Q}(t) \le \kpp[\omg_{Lamb} ] + \lambda . $
		\item ${|\mu|_{Q}(t) \le \mu[\omg_{Lamb}] + \lambda.} $ 
		\item $E_{Q}(t) \ge E[\omg_{Lamb} ] - \lambda. $
	\end{enumerate} For precise definitions, see \S \ref{subsubsec:boot} below. The first one ensures that the moving frame is well-defined. Compared to the signed vorticity case, the differences are: we replaced $\mu_Q$ with $|\mu|_Q$ in (4) and to close (4), we introduced a new hypothesis (2). Let us now explain $\Phi(t,A^{far,t})$, which refers to ``far'' particles at time $t$. The assumption (2) says that the amount of $L^1_*$ norm contained in these far particles is always small. 
	
	\smallskip \noindent \textit{The moment condition and local versus far particles} Again, the most difficult issue is to keep the bootstrap assumption (4) on the $L^1_*$ bound in $Q$. The  particles may stay in $Q$ for an arbitrarily long time (depending on their initial location), during which they can contribute to growth of $L^1_*$ by traveling in the positive $x_{2}$ direction. At least, we see that if a particle goes far in $x_{2}$, it should cross $\partial Q$ from right to left, since the moving frame defining $Q$ moves to right with approximately unit speed and such far particles have $O(\varepsilon)$ speed via the bootstrap assumption (1). 
	
	The key idea introduced here is to divide \textit{particles at each $t$} those who were in $Q$ at the initial time     
	into two categories: local versus far particles. Local particles $\Phi(t,A^{loc,t})$ are defined by those whose trajectories were ``close'' to the Lamb dipole part for some instant of time. All the others are defined to be far particles $\Phi(t,A^{far,t})$. This dichotomy is precisely defined by introducing a large square $S_{R} = [-R,R] \times [0,R]$ (where we take $R = 100000$) in the moving frame: local particles are those who intersects $S_{R}$
	before time $t$.  
	\begin{itemize}
		\item Controlling local particles: the relevant key lemmas are  Lemma \ref{lem:vel-Q-decomp} and Lemma \ref{lem:S_R_to_H}, which state that 
		after a local particle intersects $S_{R}$,
		its height ($x_{2}$ coordinate) is \textit{uniformly} bounded by a constant which is denoted by $H$
		as long as it stays in $Q$. 
		This is nontrivial: the timescale a particle travels from $S_{R}$ to $\partial Q$ is $O( \dlt^{-1/2})$ and since perturbations may travel in $x_{2}$ with speed $\varepsilon$, the height can be in principle $\varepsilon\dlt^{-1/2}$, which is huge ($\dlt \ll \varepsilon$). Regarding this issue, the main observation is that the Lamb dipole velocity ``pushes down'' fluid behind it (again, see Figure \ref{fig:lamb}), and this effect is dominant when the height is not too large. Based on this, the lemma is proved by decomposing $Q$ into three regions, and bounding the velocity differently in each of these regions (Figure \ref{fig:trajectory}).

		\item Controlling far particles: this is the most technically involved part of the paper, which is treated in Lemma \ref{lem:v_2_on_S} and Lemma \ref{lem:far_cont}. Unlike local particles, the difficulty is that a far particle can gain arbitrarily large amount of impulse 
		during its lifetime in $Q$.
		Here is where the moment smallness \eqref{assump_moment} is necessary: roughly speaking, if one computes the integral of $|x|\rho(t,x)$ along particle trajectories, potential increase in $x_{2}$ is balanced with the decrease of $x_{1}$, because we are in a moving frame. While this is the heuristic reason why we can close the assumption (2), the most challenging part is to treat the interaction between a far field particle near the borderline $\{ x_{1} = -K\dlt^{-1/2} \} $ with the ``error'' which is located close to the left of this borderline: there is no bootstrap assumption giving smallness of this error part. To handle this interaction, we decompose the error into three parts (see Figure \ref{fig:far}), and bound the corresponding interaction differently in each part.
	\end{itemize}
	
	Lastly, we note that the support area boundedness is not assumed in \eqref{eq:ini-assumption}: we use initial weighted $L^1$ smallness and a lemma (Lemma \ref{lem:L2-L1}) which allows one to pass from $L^2$ stability to $L^1$.

	\subsection{Potential Applications and Extensions}\label{subsec:pa}
	We believe that our result and the proof strategy have several applications and extensions, and we list a few below. 
	
	To begin with, our stronger stability statement for the Lamb dipole (namely, without the sign assumption) gives \textit{gradient growth} for a larger class of perturbations of the Lamb dipole: roughly speaking, proof of instability requires stability. More precisely, \cite{JYZ} obtained superlinear gradient growth for the signed vorticity in $\bbR^2_+$ vanishing on the boundary (improving an earlier work \cite{CJ-Lamb} which gives linear growth). Since our stability does not require the sign condition, it implies that the result of \cite{JYZ} holds for an \textit{open} set in $L^1 \cap L^\infty \cap L^1_{|x|}(\bbR^2_+)$ of the Lamb dipole, where $L^1_{|x|}$ denotes the $|x|$ weighted $L^1$ norm. 
	
	Naturally, we expect the result to hold for more general vortex dipoles which are obtained by maximizing the energy under a few natural constraints, for instance those appear in \cite{AC2019,AJQW}. We plan to discuss this extension in a future work. 
	
	Furthermore, our result can be the starting point for studying interactions of dipoles in the plane, at least when each of them is odd symmetric in $x_{2}$, so that they all move parallel along the $x_{1}$ axis. More concretely, our strategy should allow one to prove stability of multiple Lamb dipoles (moving in parallel with each other) under the assumption that some of them have negative signs in $\bbR^2_+$, extending earlier works \cite{ACJ,AJY}. Moreover, it will be interesting to extend the theory of Iftimie--Lopes Filho--Nussenzveig Lopes \cite{ILL2003} to the mixed sign case, see \S \ref{subsec:refs} below. 
	
	Lastly, all of this considerations can be done for the axisymmetric Euler equations without swirl, in particular one may study stability of the Hill's vortex (\cite{Choi-Hill}) without the sign condition on the vorticity. It would be interesting to see whether the ideas developed here can be applied to other equations as well. 
	
	\subsection{Related works}\label{subsec:refs}
	
	\subsubsection{Remarks for vortex stability in the plane} We note that the stability of vortices defined in the whole plane $\bbR^{2}$ has a similar restriction on the sign on the vorticity: for instance,
	stability of the Rankine vortex patch is known only under the sign condition \cite{SV09,MR4417385,WP85}. Similarly, stability of nonnegative, monotone radial vortices is known only for nonnegative perturbations. For rotating patches, the situation is even worse: the stability results for Kirchhoff ellipses \cite{Tang,Wan86} and Kelvin waves \cite{MR4493276} require not only sign condition but also additional conditions on the support of the perturbations.
	
	On bounded domains, nonlinear stability can hold without a sign condition \cite{MP85,Bu05}; vorticity cannot escape to infinity, so the difficulty discussed in \S \ref{subsec:ideas} does not arise. Burton \cite{Bu05} established a variational stability criterion for steady vorticities that strictly maximize or minimize the kinetic energy within their rearrangement class. More recently, Wang \cite{Wang25} proved orbital stability, up to rotations, of a class of steady flows in a disk, including the truncated Lamb dipole. 

	\subsubsection{Dipole stability} Experiments and numerical simulations of the two-dimensional vorticity equation show emergence of vortex dipoles  \cite{VF89,FV94,Hesthaven,Niel2}, suggesting a very strong form of stability for them. 
	
	In the mathematical literature, there are numerous works on the existence of vortex dipoles \cite{Norbury75,Tu83,Burton88, Burton89, Van10,CLW14,CLZ21,HM17,HaHm, GaHa, ACJSW, CJS, HT} by various methods. Stability was obtained for variationally constructed vortex dipoles in the pioneering work \cite{BNL} under sign and symmetry assumptions, see also \cite{Burton21, AC2019, Wang.2024,CQZZ25}.  
	
	In the specific case of the Lamb dipole, uniqueness of the maximizer is known in several variational formulations \cite{Burton96,Burton05b,AC2019,ACJ}. The concentrated vortex patches case was settled in \cite{CQZZ25}. More recently, Li--Song--Zhou \cite{LSZ2026}  proved quantitative orbital stability of Lamb dipole. They refined Theorem~\ref{thm:Lamb-sign} by proving  a uniform-in-time $O(\delta)$ bound for the modulated error when its initial size is $\delta$. However, their nonlinear stability result still assumes nonnegative vorticity in the half-plane. Without sign or symmetry assumptions, nonlinear orbital stability of the Lamb dipole remains unclear, although some progress has been made at the linearized level. Protas \cite{Pro} reported a singular unstable mode in a numerical and asymptotic study. Li--Song--Zhou \cite{LSZ2026} proved mode stability in $L^2$ and spectral stability in weighted $L^2$ spaces, without symmetry or sign restrictions. Numero--Ventura \cite{NumeroVentura26} obtained polynomial bounds for the linear evolution in $L^1\cap L^p$, $p>2$, with growth arising from core circulation and generalized eigenvectors at zero.

	\subsubsection{Dynamical results on the half-plane}
	A nonnegative vorticity defined in the half-plane has the tendency to move towards the positive $x_{1}$ axis, and there are a few rigorous results in this direction. For nontrivial, nonnegative initial vorticity $\omg_0\in L^\infty(\bbR^2_+)$ with compact support, the center of vorticity moves to the right with speed bounded away from zero \cite{Iftimie1}. Moreover, the support satisfies $x_2\leq C(t\log t)^{1/3}$ and $x_1\geq-C(t\log t)^{1/2}$ for sufficiently large $t$, where $C$ depends on the initial data \cite{Iftimie1,ILL2003}. 
	
	Iftimie--Lopes Filho--Nussenzveig Lopes \cite{ILL2003} obtained a conditional scattering result for the rescaled vorticity $t^2\omg(t,t\cdot)$, which preserves the $L^1$ norm. More precisely, if the rescaled vorticity converges weakly as a measure on $\overline{\bbR^2_+}$ as $t\to\infty$, then  $$t^2\omg(t,t\cdot)\rightharpoonup\sum_{i=1}^{N}\Gamma_i\delta_{(a_i,0)},$$
	where $1\leq N\leq\infty$, $\Gamma_i>0$, and the nonnegative $a_i$ are bounded and can accumulate only at $0$. For a single traveling dipole, $a_1$ is its speed and $\Gamma_1$ is its circulation in $\bbR^2_+$. 
	
	
	\subsubsection{Multi-vortex solutions}
	Towards the above scattering result, it is natural to study dynamics of vorticities which are concentrated in several points in the half-plane, which we refer to as ``multi-vortex'' solutions. 
	
	At the point vortex level, there are a lot of computational and analytical studies for interactions of dipoles; see \cite{ArefStremler1999,Eckhardt1988Integrable,LydonNazarenkoLaurie2022,TophojAref2008,Price1993,Eckhardt1988Scattering,ManakovShchur1983,EckhardtAref1988}. In particular, \cite{Yang} studied the motion of two vortices in the half-plane without a sign restriction and proved that they neither collide with each other nor reach the boundary in finite time. In this case, there are basically two regimes, one which the vortices move in almost straight lines with weak interaction, and the other which they interact strongly and \textit{leapfrog}. This leapfrogging dynamics has been intensively investigated by experimental and computational studies (\cite{BehringGoodman2019,WhitchurchEtAl2018,TophojAref2013,Acheson2000}). In the mathematical literature, global solutions approaching two vortex pairs traveling in opposite directions were constructed by a gluing method in \cite{DdPMP2}; see also \cite{DdPMP}. The work \cite{HHM} constructed four leapfrogging vortex patches whose motion is periodic in a translating frame, using KAM theory and a Nash--Moser iteration. These works do not establish orbital stability of the constructed solutions, which seems to be a very challenging problem for leapfrogging solutions. 
	
	Stability of highly concentrated vortex quadrupoles under odd symmetry in both variables was studied in  \cite{CJY} under nonnegativity in the first quadrant. For Lamb dipoles,  \cite{CJY} proved forward-in-time stability of two sufficiently separated, oppositely signed dipoles under odd symmetry in both variables and nonnegativity in the first quadrant. 
	
	\subsubsection{Viscous dipoles} So far we have been only concerned with mathematical works for the ideal vorticity equation, and another interesting topic is to study viscous vortices, in particular viscous dipoles. Starting with a general point vortex configuration, \cite{Gallay11} gives $O(1)$ (i.e. uniform in viscosity) time asymptotics of the Navier--Stokes solutions. In the specific case of traveling vortex pairs, we refer to very recent work \cite{DoGa} for the long-time evolution of viscous dipoles; see also \cite{DG,ZhZh,DD} for related developments.

	\subsection{Open problems}\label{subsec:open} In our opinion, the mathematical analysis of vortex dipoles is still at an early stage. Here we list a few open problems.
	
	\begin{enumerate}
		\item Odd symmetry condition in $\bbR^{2}$: Our result gives stability of the Lamb dipole in the whole plane $\bbR^{2}$ with respect to odd symmetric perturbations in $x_{2}$. Stability of the Lamb dipole could be simply false without this odd symmetry condition; see \cite{Pro}. 
		\item Stable vortex dipole with mixed sign in $\bbR^2_+$: The existence of stable (in a similar sense with Theorem \ref{thm:main}) traveling waves with mixed-sign on $\bbR^2_+$ is not known. There are a few families of traveling waves with mixed signs in $\bbR^2_+$. One family comes from so-called \textit{shielded} Lamb dipoles, which are obtained from truncating the Bessel's function $J_{1}$ at its second, third, and so on zeroes. They are expected to be highly unstable (\cite{OVH94}). Another family can be obtained by desingularization of point vortex equilibria (\cite{CQZZ-doubly}). 
		\item Asymptotic stability: an outstanding open question is \textit{asymptotic} stability of the Lamb dipole. However given the result of \cite{FHM2026}, it is conceivable that there exist permanent oscillations near the Lamb dipole as well. 
		\item Collision of Lamb dipoles: Instead of considering Lamb dipoles moving away from each other as in \cite{CJY}, one may study situations when they approach each other. Then, the challenging problem is to describe dynamics after the ``collision.''
		In a highly asymmetric regime, however, the present result (Theorem \ref{thm:main}) already gives some information. For instance, consider an order-one Lamb dipole traveling to the right and a much smaller, oppositely signed Lamb dipole of size $O(\delta)$  traveling to the left, with the smaller dipole initially located to the right of the larger one, see Figure \ref{fig:difficulty}. Our theorem implies that at least the order-one dipole persists even after (possible) collision. Beyond this highly asymmetric regime, even head-on collisions between Lamb dipoles of comparable but unequal sizes remain a challenging problem. Collisions with non-parallel trajectories appear to be even more challenging. See \cite{Or92} for numerical simulations of collisions of Lamb dipoles.
	\end{enumerate}

	\subsection{Notation}\label{subsec:nota} 
	We use the following notation in this paper. 
	\begin{itemize}
		\item The integrals $\int$, points $x,y$ and norms $\|\cdot\|$ will be taken in the upper half-plane $\bbR^2_+$, unless otherwise specified.
		
		\item For $x = (x_1,x_2) \in \bbR^2_+$, we write $\bar{x} = (x_1,-x_2)\in\bbR^2$. With this notation, for two points $x, y \in \bbR^2_+$, we have $|x-\bar{y}| = |\bar{x}-y| \ge |x_2+y_2|$. 
		\item For a measurable set $A \subset \ohp$, we write $\mathbf{1}_{A}$ to be the indicator function on $A$. 
		\item Given a measurable function $f$ on $\bbR^2_+$, we consider its decomposition into the positive and negative parts; we write $f = f^{(+)} - f^{(-)}$ with $f^{(+)}(x) := f(x) \mathbf{1}_{ \{ f > 0 \} }$ so that $f^{(+)}, f^{(-)} \ge 0$. 
	\end{itemize}

	\subsection{Organization of the paper}\label{subsec:orga}
	
	The rest of the paper is organized as follows. 
	After some preparations in \S \ref{sec:prelim}, the bootstrap hypotheses are introduced in \S \ref{sec:Boot}. This section also includes the proof that these hypotheses hold at least for a short time interval. Next, we prove in \S \ref{sec:conseq} various estimates on the solution, assuming the bootstrap hypotheses. Lastly, all of the bootstrap hypotheses are closed in \S \ref{sec:closing}.

	\subsection*{Acknowledgments}
	KA is supported by the JSPS through the Grant-in-Aid for Scientific Research (C) 24K06800 and MEXT Promotion of Distinctive Joint Research Center Program JP MXP0619217849. 
	KC was supported by the National Research Foundation of Korea (NRF) grant funded by the Korean government (MSIT) (No. RS-2023-00274499).
	IJ was supported by the NRF grant from the Korea government (MSIT), No. 2022R1C1C1011051, RS-2024-00406821 and the Asian Young Scientist Fellowship. 
	GQ was supported by National Key R\&D Program of China (Grant 2025YFA1018400) and NNSF of China (Grant 12471190).
	
	\subsection*{AI Declarations} We declare that no AI tools were used in any part of this work.

    \subsection*{Data Statement} Data sharing is not applicable to this article as no datasets were generated or analyzed during the current study.

	\section{Preliminaries}\label{sec:prelim}

	\subsection{Lamb dipole} \label{subsec:Lamb} 
	
	Let us now provide some detailed computations regarding the normalized Lamb dipole. The stream function $\psi_{Lamb}$ can be written explicitly by \begin{equation}\label{eq:Lamb-psi}
		\psi_{Lamb}(x) =   
		\left\{
		\begin{aligned}
			x_{2}\left( 1 - \frac{{2}J_{1}(c_{L}|x|)}{ c_{L}|x| J_{0}(c_{L}) } \right)	, & \quad |x| \le 1 \\
			x_{2}|x|^{-2},	& \quad |x| > 1 
		\end{aligned}
		\right.,
\end{equation} 
Then from \eqref{eq:Lamb-def}, one can check that the vorticity $\omg_{Lamb}=
-\lap\psi_{Lamb}$ satisfies
\begin{equation}\label{eq:Lamb-vor}
	\begin{split}
		\omg_{Lamb}(x) = c_{L}^2(\psi_{Lamb}(x)-x_2)_{+} = 
		\left\{
		\begin{aligned}
			c_{L}^2(\psi_{Lamb}(x)-x_2), & \quad |x| \le 1 \\
			0,	& \quad |x| > 1 
		\end{aligned}
		\right.. 
	\end{split}
\end{equation} Here $(g)_+ := \max\{ g, 0 \}$. The relation \eqref{eq:Lamb-vor} says that $\omg_{Lamb}$ defines a traveling wave of \eqref{eq:2D-Euler} with velocity $\mathbf{e}_{1}$. 

We also note that $u_{Lamb} := -\nb^\perp \psi_{Lamb}=(\partial_{x_2}\psi_{Lamb}, -\partial_{x_1}\psi_{Lamb})$ is explicitly given.
In this paper, we are interested in the dynamics \textit{outside} the support of the Lamb dipole, where the velocity field takes the following simple form:
for $|x|>1$,  \begin{equation}\label{eq:Lamb-vel}
	\begin{split}
		u_{Lamb}(x) = 
		\left(		(x_{1}^2 - x_{2}^{2})|x|^{-4},\, 
		2x_{1}x_{2}|x|^{-4}\right).
	\end{split}
\end{equation}

\subsection{Bounds for the stream function and velocity}    Given the vorticity $\omg$, we define the stream function  $\psi = (-\lap)^{-1}\omg$ and the velocity $u = (u_1,u_2)^\top$ by $u =- \nb^\perp\psi=(\partial_{x_2}\psi, -\partial_{x_1}\psi)
$. Introducing the kernels \begin{equation}\label{eq:G-def}
	\begin{split}
		\bfG(x,y) := \frac{1}{4\pi} \log \left( 1 + \frac{4x_{2}y_{2}}{|x-{y}|^{2}} \right)  >0,\quad x\neq y \in\mathbb{R}^2_+,
	\end{split}
\end{equation}
and \begin{equation}\label{eq:K-def}
	\begin{split}
		\bfK(x,y) = (\bfK_1(x,y),\bfK_2(x,y)) := -\nb^\perp_{x}\bfG(x,y)=
		\left(\frac{-y_{2} \left( x_{2}^{2} - y_{2}^{2} - (x_{1}-y_{1})^{2} \right) }{\pi |x-y|^{2} |x-\bar{y}|^{2}},
		\frac1{2\pi}\frac{x_1-y_1}{|x-y|^{2}} \frac{4x_2y_2}{|x-\bar{y}|^2}\right),
	\end{split}
\end{equation} we have the representation formulas  \begin{equation}\label{eq:psi-G}
	\begin{split}
		\psi(x)
		& = G[\omega](x)=  \int_{\mathbb{R}^2_{+}} \bfG(x,y) \omega(y)\,dy , \quad u(x)=K[\omega](x) =\left(K_1[\omega](x),K_2[\omega](x)\right)= \int_{\bbR^2_+} \bfK(x,y) \omg(y) \, dy . 
	\end{split}
\end{equation}

\begin{lemma}\label{lem:vel-decay}
	For any $0<\alp\le1$, there is a constant $C_\alp>0$   such that 
	\begin{equation*}
		\begin{split}
			0\leq \bfG(x,y) \le C_{\alp}  \frac{(x_2y_2)^\alp}{|x-{y}|^{2\alp}} 
		\end{split}
	\end{equation*}  for all $x \ne y$. The velocity kernel $\bfK$ satisfies \begin{equation}\label{eq:K-decay}
		\begin{split}
			|\bfK(x,y)| \le \frac{Cy_{2}}{|x-y||x-\bar{y}|}
		\end{split}
	\end{equation} for $x \ne y$.
\end{lemma} 
\begin{proof}
	The proofs  are straightforward from the definitions of the kernels from \eqref{eq:G-def} and \eqref{eq:K-def}.
\end{proof}
\begin{lemma}\label{lem:vel-L-infty}
	For any $0<\alp \le 1$, there exists a constant $C_{\alp}>0$ such that 
	\begin{equation}\label{eq:vel-L-infty-alpha2}
		\begin{split}
			\nrm{x_{2}^{-1}\psi}_{L^{\infty}}, \,  \nrm{u}_{L^{\infty}} \le C_{\alp}\nrm{\omg}_{L^{2}}^{1-\alp}  \nrm{\omg}_{L^{\infty}}^{\alp/2} \nrm{ \omg}_{L^{1}}^{\alp/2} . 
		\end{split}
	\end{equation}
\end{lemma}  
\begin{proof}
	Since $u_{1} = \rd_{x_2}\psi$ and $\psi = 0$ on $\{ x_{2} = 0 \}$, we have from Hardy's inequality that \begin{equation*}
		\begin{split}
			\nrm{x_{2}^{-1}\psi}_{L^{\infty}} \le  \nrm{u_1}_{L^{\infty}} .
		\end{split}
	\end{equation*} It hence suffices to show $\nrm{u}_{L^{\infty}} \le C_{\alp}\nrm{\omg}_{L^{2}}^{1-\alp}  \nrm{\omg}_{L^{\infty}}^{\alp/2} \nrm{ \omg}_{L^{1}}^{\alp/2}$. This is well-known in the case $\alpha = 1$. To prove it for general $\alpha \in (0,1)$, we use $|\bfK(x,y)| \le C|x-y|^{-1}$ (which follows from \eqref{eq:K-decay} and $y_{2} \le |x-\bar{y}|$) and for some $0<R_1 \le R_2$ to be determined, write \begin{equation*}
		\begin{split}
			|u(x)| \lesssim \left[ \int_{ B_{x}(R_{1}) } +  \int_{ B_{x}(R_{2}) \backslash B_{x}(R_{1}) }  + \int_{ \bbR^2_+ \backslash B_{x}(R_{2}) }     \right]  \frac{|\omg(y)|}{|x-y|} \, dy. 
		\end{split}
	\end{equation*} Using the $L^\infty$, $L^2$, and $L^1$ norms of the vorticity in each of the regions, we bound \begin{equation*}
		\begin{split}
			|u(x)| \lesssim R_{1}\nrm{\omg}_{L^\infty} + \sqrt{\log\left( R_{1}^{-1}R_{2} \right) }\nrm{\omg}_{L^2} + \frac{1}{R_{2}}\nrm{\omg}_{L^{1}} 
			\lesssim_{\beta} R_{1}\nrm{\omg}_{L^\infty} + \left( \frac{R_{2}}{R_{1}} \right)^{\beta} \nrm{\omg}_{L^2} + \frac{1}{R_{2}}\nrm{\omg}_{L^{1}} 
		\end{split}
	\end{equation*} for any $\bt>0$. Given $\bt$, we select $0<R_1\le R_2$ to satisfy \begin{equation*}
		\begin{split}
			\nrm{\omg}_{L^1} = R_{1}R_{2} \nrm{\omg}_{L^\infty}, \qquad \left( \frac{R_1}{R_2} \right)^{1+2\bt} \nrm{\omg}_{L^1}  \nrm{\omg}_{L^\infty} = \nrm{\omg}_{L^2}^{2}. 
		\end{split}
	\end{equation*} This is possible since we always have $\nrm{\omg}_{L^1}  \nrm{\omg}_{L^\infty} \ge \nrm{\omg}_{L^2}^{2}$. Solving this system of equations for $R_{1}, R_{2}$ and plugging them back into the previous bound for $|u(x)|$, we obtain \begin{equation*}
		\begin{split}
			|u(x)| \lesssim_{\beta} \left(\nrm{\omg}_{L^1}  \nrm{\omg}_{L^\infty} \right)^{ \frac12 \frac{2\bt}{1+2\bt} } \nrm{\omg}_{L^2}^{\frac{1}{1+2\bt}}
		\end{split}
	\end{equation*} which finishes the proof. 
\end{proof}

\begin{remark}
	For any $f,g\in L^1\cap L^\infty(\mathbb{R}^2_+)$, we have the identity
	$$
	\int_{\mathbb{R}^2_+}K_2[f]\cdot g\,dx=
	-\int_{\mathbb{R}^2_+}f\cdot K_2[g]\,dx
	$$ thanks to the anti-symmetry of the kernel $K_2$. As a consequence,
	$$
	\int_{\mathbb{R}^2_+}K_2[f]\cdot f\,dx=0. 
	$$ 
\end{remark}

\subsection{Kinetic energy}\label{subsubsec:energy} For $\omg \in L^1 \cap L^\infty (\bbR^2_+)$, we recall the kinetic energy $E=E[\omega]$ defined by \begin{equation*}
	\begin{split}
		E[\omega]=  \frac12 \int_{\mathbb{R}^2_+}\psi(x)\omega(x)\,dx= \frac12 \int_{\mathbb{R}^2_+}|u(x)|^{2} \,dx = \frac12 \iint_{\bbR^2_+ \times \bbR^2_+} \omg(x)\bfG(x,y) \omg(y) dxdy, 
	\end{split}
\end{equation*} where $\bfG$ is defined in \eqref{eq:G-def}. Similarly, the \emph{interaction energy} is defined by 
\begin{equation}\label{eq:E-int-def}
	\begin{split}
		E_{inter}[\omg,\tld{\omg}] = \frac12 \iint_{\bbR^2_+ \times \bbR^2_+} \omg(x) \bfG(x,y) \tld{\omg}(y) dxdy = E_{inter}[\tld\omg,{\omg}].
	\end{split}
\end{equation} We note
$$
E[f+g]=E[f]+E[g]+2E_{inter}[f,g].
$$

For the energy, we recall the following estimates from \cite[Proposition 2.4]{AJY} or
\cite[Lemma 2.4]{ACJSW}.

\begin{lemma}
	\label{prop:Energy-X}
	Let $\omg, \Omega \in L^{2} \cap L^{1}_{*}(\bbR^2_+)$ and $\psi $ be the corresponding stream function of $\omg$. Then, we have 
	\begin{equation}\label{eq:E-inter-unsymm}
		\begin{split}
			\left|E_{inter}[\omg,\Omega]\right| \le \frac12   \nrm{x_{2}^{-1} {\psi}}_{L^{\infty}(\mathbb{R}^2_+)}\nrm{\Omega}_{L^{1}_{*}(\bbR^2_+)} 
		\end{split}
	\end{equation} 
	and\begin{equation}\label{eq:Energy-difference-X}
		\begin{split} 
			|E[\omg] - E[\Omega]| \le C(\nrm{\omg-\Omega}_{L^{1}_{*}(\bbR^2_+)} \nrm{\omg-\Omega}_{L^{2}(\bbR^2_+)}\nrm{\omg+\Omega}_{L^{1}_{*}(\bbR^2_+)} \nrm{\omg+\Omega}_{L^{2}(\bbR^2_+)})^{1/2}.\end{split}		\end{equation}
\end{lemma}

\subsection{Variational principle without sign condition}

We introduce the following admissible class for functions on $\bbR^2_+$,  for given $\kpp, \mu > 0$:  \begin{equation*}
	\begin{split}
		\widetilde{\calA}_{\kpp,\mu} := \left\{ \omg \mbox{ defined on } \bbR^2_+ \, : \, \nrm{\omg}_{L^2} \le \kpp,  \nrm{\omg}_{L^1_{*}} \le \mu \right\}.
	\end{split}
\end{equation*} We consider  the Lamb dipole $\omg_{Lamb}^{\kappa,\mu}$ in \eqref{eq:lamb-rescale-def} satisfying
$\nrm{\omg_{Lamb}^{\kappa,\mu}}_{L^2} = \kpp,  \nrm{\omg_{Lamb}^{\kappa,\mu}}_{L^1_{*}} =\mu$. 
As a corollary of Proposition \ref{prop:en_eq},  we note that it maximizes the kinetic energy in this class $\widetilde{\calA}_{\kpp,\mu}$:
$$
\max_{\omg\in \widetilde{\calA}_{\kpp,\mu}}E[\omg]=E[\omg_{Lamb}^{\kappa,\mu}]=C_L\cdot\kpp\cdot\mu.
$$

We present an extension of the compactness statement \cite[Proposition 1.3, Theorem 2.7]{AJY} which removes the sign condition.  
It is a reformulation of 
\cite[Theorems 1.3 and 1.5]{AC2019}.

\begin{proposition}\label{thm:AC2019}
	Let $\kpp,\mu>0$ 
	and $M > 0$.
	Then, for any $\varep_{1}>0$, there exists $\lambda_{1} = \lambda_{1}(\kpp,\mu, M,
	\varep_{1})>0$ such that the following holds: if 
	$\nrm{\omg}_{L^{1}(\bbR^2_+)} \le M$ satisfies  \begin{equation*}
		\begin{split}
			\omg \in \widetilde{\mathcal{A}}_{ \kpp +\lambda_{1}, \mu +\lambda_{1} } \qquad\mbox{and}\qquad E[\omg] \ge E[\omg_{Lamb}^{\kpp,\mu}] - \lambda_{1}, 
		\end{split}
	\end{equation*} then we have \begin{equation*}
		\begin{split}
			\min\left\{ \inf_{\tau\in\bbR} \nrm{ \omg - \omg_{Lamb}^{\kpp,\mu}(\cdot - \tau \mathbf{e}_{1}) }_{L^{2} \cap L^{1}_{*} }  \, , \,  \inf_{\tau\in\bbR} \nrm{ \omg +  \omg_{Lamb}^{\kpp,\mu}(\cdot - \tau \mathbf{e}_{1}) }_{L^{2} \cap L^{1}_{*} }  \right\} \le \varep_{1}.  
		\end{split}
	\end{equation*} 
\end{proposition}

\begin{proof}
	We instead show the following. For any sequence $\{ \omg_{n} \}_{n\ge 0}$ satisfying $ \omg_{n} \in \widetilde{\mathcal{A}}_{\kpp+(1/n),\mu+(1/n)}$,  $\nrm{\omg_n}_{L^{1}} \le M$,  and $E[\omg_{n}] \geq  E[\omg_{Lamb}^{\kpp,\mu}]-(1/n)$ for each $n$, there is a subsequence $\{ \omg_{n_{k}} \}$ that $\omg_{n_{k}}$ converges to either $\omg_{Lamb}^{\kpp,\mu}$ or $(-\omg_{Lamb}^{\kpp,\mu})$ in $L^{2} \cap L^1_*$ after shifts. 
	
	
	We decompose $\omg_{n} = \omg_{n}^{(+)} - \omg_{n}^{(-)}$ and note that \begin{equation*}
		\begin{split}
			E[\omg_n] = E[\omg_{n}^{(+)}] + E[\omg_{n}^{(-)}] - 2\underbrace{E_{inter}[\omg_{n}^{(+)},\omg_{n}^{(-)}]}_{\geq 0} \le E[\omg_{n}^{(+)}] + E[\omg_{n}^{(-)}]. 
		\end{split}
	\end{equation*}
	Assume without loss of generality that (by taking a subsequence if necessary) $\nrm{\omg_{n}^{(+)}}_{L^{2}} \ge \nrm{\omg_{n}^{(-)}}_{L^{2}}$, which implies that $\nrm{\omg_{n}^{(-)}}_{L^{2}} \le (\kpp+(1/n))/\sqrt{2}$. Next, towards a contradiction, assume that $$\liminf_{n\to\infty}\nrm{\omg_{n}^{(-)} }_{L^1_*} =: \nu > 0.$$ Then along the subsequence achieving $\liminf$, we obtain  \begin{equation*}
		\begin{split}
			C_{L}\kpp\mu \le C_{L}\kpp(\mu-\nu) + \frac{C_{L}}{\sqrt{2}}\kpp\nu
		\end{split}
	\end{equation*} in the limit, which is a contradiction. Therefore, $\nu = 0$. This gives that \begin{equation*}
		\begin{split}
			C_{L}\kpp\mu \le \limsup_{n\to\infty} C_{L}\nrm{\omg_{n}^{(+)}}_{L^{2}} \mu  
		\end{split}
	\end{equation*} and in particular $\nrm{\omg_{n}^{(-)}}_{L^{2} \cap L^1_*} \to 0$ as $n\to\infty$ with $E[\omg_{n}^{(+)}] \to C_{L}\kpp\mu$ (up to a subsequence). Since $0 \le \omg_n^{(+)} \in \widetilde{\mathcal{A}}_{\kpp+(1/n),\mu+(1/n)}$, we can use  \cite[Proposition 1.3]{AJY} so that there is a subsequence $\omg_{n_{k}}^{(+)} \to \omg_{Lamb}^{\kpp,\mu}$ in $L^{2} \cap L^1_*$ after shifts. So does $\omg_{n_{k}}$.	\end{proof}

\subsection{Stability from $L^2$ to $L^1$} 
\begin{lemma}\label{lem:L2-L1}
	Suppose $f, g\in (L^1\cap L^2)(\mathbb{R}^2_+)$ with
	$\|f\|_{L^1}\leq\|g\|_{L^1}$. Then we have
	\begin{equation}\label{L2_to_L1}
		\|f-\omega_{Lamb}\|_{L^1}\leq 4\|f-\omega_{Lamb}\|_{L^2}+\|g-\omega_{Lamb}\|_{L^1}. 
	\end{equation} 
\end{lemma}
\begin{proof}
	We recall $B^+_1=\{x\in\mathbb{R}^2_+\,:\,|x|<1\}$, which is the support of $\omega_{Lamb}$. 
	We compute
	\begin{equation*}\begin{split}
			\|f-\omega_{Lamb}\|_{L^1}&
			\leq \|f-\omega_{Lamb}\|_{L^1(B^+_1)}+\|f\|_{L^1(\mathbb{R}^2_+\setminus B^+_1)}\\ 
			&= \|f-\omega_{Lamb}\|_{L^1(B^+_1)}+\|f\|_{L^1}-\|f\|_{L^1(B^+_1)} \\ 
			&\leq \|f-\omega_{Lamb}\|_{L^1(B^+_1)}+\|g\|_{L^1} -\|f\|_{L^1(B^+_1)}\\ 
			&\leq \|f-\omega_{Lamb}\|_{L^1(B^+_1)}+\|g-\omega_{Lamb}\|_{L^1} +\|\omega_{Lamb}\|_{L^1} -\|f\|_{L^1(B^+_1)}\\ 
			&\leq 2\|f-\omega_{Lamb}\|_{L^1(B^+_1)}+\|g-\omega_{Lamb}\|_{L^1}\\ 
			&\leq 4\|f-\omega_{Lamb}\|_{L^2}+\|g-\omega_{Lamb}\|_{L^1},
		\end{split}
	\end{equation*} where we used the H\"older inequality in the last line.
\end{proof}







\section{Bootstrap hypotheses}\label{sec:Boot}

This section is organized as follows. To begin with, in \S \ref{subsec:redu}, we reduce the proof to the case of $t\ge0$ (rather than $t\in\bbR$) and for sufficiently regular data $\omg_0\in H^{10}(\mathbb{R}^2_+)$ (say). In \S \ref{subsec:ini}, we recall the assumptions on the initial data and obtain several immediate consequences. We systematically introduce all the parameters in the proof in \S \ref{subsec:params}. Then in \S \ref{subsec:bootstrap}, we introduce a moving frame and fix the set of bootstrap hypotheses. For this purpose, several parameters are introduced, and the bootstrap hypotheses are stated in terms of a decomposition of the solution into ``main'' part and remainder, in the moving frame. In \S \ref{subsec:bootstrap-short}, we show that these hypotheses are valid at least for a short time interval containing $t = 0$. Consequences of the bootstrap hypotheses are then collected in the next section \S \ref{sec:conseq}.  

\subsection{Reductions}\label{subsec:redu} First, we recall the following symmetry of 2D Euler equations \eqref{eq:2D-Euler}: if $\omg(t,x)$ is a solution, then $\omg^*(t,x):=\omg(-t,-x_{1},x_{2})$ is also a solution of \eqref{eq:2D-Euler}. Since the Lamb dipole is even symmetric in $x_{1}$, for any initial data $\omg_{0}$ satisfying the assumptions of Theorem \ref{thm:main}, the even reflection $\omg_{0}(-x_{1},x_{2})$ also satisfies all of the assumptions. Therefore, owing to the symmetry, it suffices to prove Theorem \ref{thm:main} only for $t \ge 0$. 

Next, we will prove the theorem for smooth $\omega_0$ during which we obtain the estimates $|\dot{p}_1(t) - 1| \le \sqrt{\varepsilon}$ for all time and $|p_1(0)| \le \sqrt{\varepsilon}$. For a general $\omega_0$, the result follows by a standard approximation argument. Indeed, an application of the Arzela--Ascoli theorem yields a Lipschitz continuous function $p_1$ that satisfies the identical estimates.

\subsection{Consequences of initial assumptions}\label{subsec:ini} Here, we collect some immediate consequences of the assumptions  \eqref{eq:mainthm-assumptions} and
\eqref{assump_moment}
on the initial data $\omg_{0}(x)$, which we recall here for convenience: with given $A> 0$ and with $\dlt>0$, which will be determined later, \begin{itemize}
	\item $\nrm{\omg_{0} }_{ L^{\infty}(\ohp)} \le A $,
	\item $\int_{ \ohp\setminus B^+_1}|x||\omg_0(x)|dx<\dlt$, 
	\item $\nrm{\omg_{0} -  {\omg}_{Lamb}}_{L^{2} \cap L^{1}_{*}(\ohp)} < \dlt$, 
\end{itemize} 
We may assume $A \ge 1000
$ and $\dlt<1$ so that  \begin{equation}
	\begin{split}\nrm{\omg_{0} }_{ L^{1}(\ohp)}&\leq
		\nrm{\omg_{0} -{\omg}_{Lamb}}_{ L^{1}(B^+_1)}+\nrm{{\omg}_{Lamb} }_{ L^{1}(B^+_1)}+\nrm{\omg_{0} }_{ L^{1}(\ohp\setminus B^+_1 )}\\&\leq
		\sqrt{\pi/2}\nrm{\omg_{0} -{\omg}_{Lamb}}_{ L^{2}(B^+_1)}+\sqrt{\pi/2}\nrm{{\omg}_{Lamb} }_{ L^{2}(B^+_1)}+\nrm{|x|\omg_{0} }_{ L^{1}(\ohp\setminus B^+_1)}\le A. \end{split}\end{equation}
\subsubsection{Positive and negative parts}
For each $t \ge 0$, we decompose the unique Yudovich solution $\omg(t,x)$ corresponding to $\omg_{0}(x)$ into the positive and negative parts \begin{equation*}
	\begin{split}
		\omg(t,\cdot) = \omg^{(+)}(t,\cdot) - \omg^{(-)}(t,\cdot)
	\end{split}
\end{equation*} where $\omg^{(+)}(t,\cdot) := \omg(t,\cdot) \, \mathbf{1}_{ \{ \omg(t,\cdot) \ge 0\} } \geq 0$ so that $\omg^{(-)}(t,\cdot) \ge 0$ everywhere on $\bbR^2_+$ as well.
We note the following pointwise inequality due to nonnegativity of $\omg_{Lamb}$:
$$|\omg^{(+)}(t,x)-\omg_{Lamb}|\leq|\omg(t,x)-\omg_{Lamb}(x)|$$ which implies, by \eqref{eq:mainthm-assumptions},
$$\nrm{\omg_{0}^{(+)} -  {\omg}_{Lamb}}_{L^{2} \cap L^{1}_{*}(\ohp)} < \dlt.$$
For $\omg_{0}^{(-)}$, we observe
\begin{equation*}
	\begin{split}
		\nrm{\omg_{0}-  {\omg}_{Lamb}}_{ L^{1}_{*}}  = \int_{ \bbR^2_+ \backslash \supp( \omg_{0}^{(-)} ) } x_{2}|\omg_{0}^{(+)}(x)-  {\omg}_{Lamb} (x)| \, dx + \int_{  \supp( \omg_{0}^{(-)} ) } x_{2}(\omg_{0}^{(-)}(x)+  {\omg}_{Lamb} (x)) \, dx,
	\end{split}
\end{equation*}
which gives $\nrm{\omg_{0}^{(-)} }_{ L^{1}_{*}(\bbR^2_+)} < \dlt$. 
Similarly, we have 
\begin{equation*}
	\begin{split}
		\nrm{\omg_{0}-  {\omg}_{Lamb}}_{ L^{2}}^{2} = \int_{ \bbR^2_+ \backslash \supp( \omg_{0}^{(-)} ) } (\omg_{0}^{(+)}(x)-  {\omg}_{Lamb} (x))^{2} \, dx + \int_{  \supp( \omg_{0}^{(-)} ) } (\omg_{0}^{(-)}(x)+  {\omg}_{Lamb} (x))^{2} \, dx,
	\end{split}
\end{equation*}
which gives  $\nrm{\omg_{0}^{(-)} }_{ L^{2}(\bbR^2_+)} < \dlt.$ 
For the $L^1$-norm, we use \eqref{assump_moment} to get
\begin{equation*}
	\begin{split}
		\|\omg_0^{(\pm)}\|_{L^1(\ohp\setminus B^+_1)} &\leq  \|\omg_0\|_{L^1(\ohp\setminus B^+_1)}\, dx \leq  \int_{ \ohp\setminus B^+_1 } |x||\omg_{0}| \,dx \leq  \dlt,
	\end{split}
\end{equation*}
\begin{equation*}
	\begin{split}
		\|\omg_0^{(+)}-\omg_{Lamb}\|_{L^1}\leq	\|\omg_0-\omg_{Lamb}\|_{L^1} &\leq \sqrt{\pi/2}\|\omg_0-\omg_{Lamb}\|_{L^2(B_1^+)}+\|\omg_0\|_{L^1(\ohp\setminus B_1^+)}\leq 3\delta,
	\end{split}
\end{equation*} 
and
\begin{equation*}
	\begin{split}
		\|\omg_0^{(-)}\|_{L^1}=\|\omg_0^{(-)}\|_{L^1(B^+_1 )}
		+\|\omg_0^{(-)}\|_{L^1(\ohp\setminus B^+_1)} &\leq  \int_{ B^+_1  } |\omg_{0}(x)-  {\omg}_{Lamb} (x)| \, dx + \int_{ \ohp\setminus B^+_1 } |\omg_{0}| \, dx\\
		& \leq \sqrt{\pi/2}\nrm{\omg_{0}-  {\omg}_{Lamb}}_{ L^{2}(B^+_1 )} + \delta \, < 3 \dlt.
	\end{split}
\end{equation*}
As a corollary, we have
\begin{lemma}\label{lem:minus-small}
	For all $t\ge 0$, we have  \begin{equation}\label{eq:L2-neg-small}
		\begin{split}
			\nrm{\omg^{(-)}(t,\cdot) }_{ L^{2}(\bbR^2_+)} < \dlt \qquad \mbox{and}\qquad\nrm{\omg^{(-)}(t,\cdot) }_{ L^{1}(\bbR^2_+)} <  3\dlt. 
		\end{split}
	\end{equation}
\end{lemma}
\begin{proof}
	We observe that, if we denote the flow map from $u=K[\omg]$ by $\phi$, then \begin{equation*}
		\begin{split}
			\omg^{(-)}(t,\phi(t,x)) = \omg_{0}^{(-)}(x) 
		\end{split}
	\end{equation*} so that $	\nrm{\omg^{(-)}(t,\cdot) }_{ L^{p}(\bbR^2_+)}  = 	\nrm{\omg_0^{(-)} }_{ L^{p}(\bbR^2_+)}$, since $\phi(t,\cdot)$ is area-preserving. \end{proof}

\subsection{Parameters}\label{subsec:params}
Before proceeding with the main argument, we clarify the hierarchy and dependencies of the parameters used throughout the proof.  

First, we fix the absolute   constants  
\begin{equation}\label{defn_C_0}
	R := 100000\quad \mbox{and}\quad C_0 := 10R.
\end{equation}
We also denote the Lipschitz constant for the Lamb dipole $\omega_{Lamb}$:  
\begin{equation}\label{defn_C_Lip}
	C_{Lip}:=\|\nabla \omg_{Lamb}\|_{L^\infty}\in(35,40) 
\end{equation}
and the total mass of the dipole
\begin{equation}\label{defn_m_L}
	m_L:=\int_{\mathbb R_+^2} \omega_{Lamb} (x)  dx\simeq 6.831>0.
\end{equation}
Based on these, we define an absolute constant 
\begin{equation}\label{defn_M_0}
	M_0 := 500\cdot\left(\frac{10C_{Lip}}{m_L}+1\right)\cdot (12R).
\end{equation}

Next, we introduce large UNIVERSAL constants $C_1, C_2, C_3, C_4 > 1$. Their exact values are not determined immediately but will be chosen to be sufficiently large during the course of the proof. With these constants in hand, our main parameters $A, \varep, \lambda$, and $\dlt$ are chosen sequentially:

\begin{itemize}
	\item Let $A \ge 1000$ be given. Without loss of generality, we may shrink $\varep > 0$ to satisfy 
	\begin{equation}\label{small_varep}
		C_1\varep \leq 1/A.
	\end{equation}    
	We also choose a large height parameter $H\geq A$ such that 
	\begin{equation}\label{large_H}
		H \geq C_2 A.\end{equation}       
	
	\item Next, we choose an intermediate parameter $\lambda \in(0,\varep^2/(AH)^3)$ small enough to satisfy
	\begin{equation}\label{small_lambda}
		C_3 \lambda \leq \varep^2 / (AH)^3 \quad \mbox{and} \quad
		\lambda \leq \lambda_1,\end{equation}       
	where\footnote{From a recent work \cite{LSZ2026}, it turned out that $\lambda_1$ 
		can be chosen explicitly in terms of $\varep$
		(after modulating the Lamb dipole), while this fact is not needed for our result.} $$ \lambda_1:=\lambda_1(\kappa,\mu,M,\varep_1)|_{\kappa=\kappa[\omega_{Lamb}], \, \mu=\mu[\omega_{Lamb}],\,M=10A, \,\varep_1=\varep/M_0}$$ is obtained from Proposition \ref{thm:AC2019}.
	\item Finally, we choose our ultimate parameter $\dlt \in(0,\lambda/AH)$ small enough to satisfy
	\begin{equation}\label{small_dlt}
		C_4\delta \leq \lambda/AH.\end{equation}      
	This completes the parameter selection.
\end{itemize}

In summary, given $A \ge 1000$, the parameters are chosen sequentially to satisfy the following:
\begin{equation*}\begin{split}
		\mbox{Choose}\quad H\gg_{C_2} A,\quad \varep \ll_{C_1} \frac{1}{A},\quad \lambda\ll_{C_3} \frac{\varep^2}{(AH)^3},\quad\mbox{and}\quad  \dlt\ll_{C_4} \frac{\lambda}{AH}.
\end{split}\end{equation*}
\subsection{Moving frame and decomposition}\label{subsec:bootstrap} It will be convenient to work in a moving frame, so that the ``Lamb dipole part'' of the solution is always centered at the origin. To this end, we use the idea from \cite{JYZ}: fix a smooth function $g:\bbR\to\bbR$ satisfying \begin{equation*}
	g(x_1) = \left\{
	\begin{aligned}
		x_1, \quad & |x_1| \le 3, \\
		4, \quad & x_1 \ge 4,  \\
		-4, \quad & x_{1} \le -4
	\end{aligned}
	\right.
\end{equation*} and $g'(x_{1}) \ge 0$, $g(x_1) = -g(-x_1)$ for every $x_{1}$. We may require that $ \nrm{g}_{L^\infty(\mathbb R)}\leq 4$ and $ \nrm{g'}_{L^\infty(\mathbb R)}\leq 2$. We define $h(t,p)$ by \begin{equation}\label{eq:h}
	\begin{split}
		h(t,p) := 	\int_{\bbR^2_+} \omg^{(+)}(t,x) g(x_{1} -p) dx 
	\end{split}
\end{equation} and define $p_{1}(t)$ by the property $h(t,p_{1}(t)) = 0$, namely\footnote{{We note that for every $(t, p)$, $\rd_{p}h \le 0 $ by $\omg^{(+)} \ge 0$ and $g' \ge 0$. This is why we define $h(t,p)$ using $\omg^{(+)}$ and not $\omg$. 
		We note $h(t,p) < 0$ as $p\to\infty$ and $h(t,p) > 0$ as $p\to-\infty$. By the intermediate value theorem, for each $t$ there is   a   $p_{1}(t)$ satisfying \eqref{eq:p1-def}. However, uniqueness and continuity in time of $p_{1}(t)$ is not trivial at this stage, and we prove these properties with a bootstrap scheme in the following sections.}}  \begin{equation}\label{eq:p1-def}
	\begin{split}
		\int_{\bbR^2_+} \omg^{(+)}(t,x) g(x_{1} - p_{1}(t)) dx = 0. 
	\end{split}
\end{equation}  As a part of our bootstrap assumptions, it will be shown that such a $p_1(t)$ is unique for each $t$, and it is a differentiable function of time  with  \begin{equation}\label{eq:p1-def2}
	\begin{split}
		|\dot{p}_{1}(t) - 1| \le \varepsilon^{1/2}. 
	\end{split}
\end{equation}
Assuming that $p_{1}(t)$ is given as a differentiable function, we make the following change of variables:  \begin{equation}\label{eq:rho-def}
	\begin{split}
		\rho(t,x) := \omg(t, x + p_{1}(t)\bfe_{1}). 
	\end{split}
\end{equation}
Then, $\rho$ is a solution of   \begin{equation*}
	\left\{
	\begin{aligned}
		& \rd_{t}\rho + v\cdot\nb \rho = 0, \\
		& v =K[\rho] - \dot{p}_{1}(t)\bfe_{1}. 
	\end{aligned}
	\right.
\end{equation*}

From now on, we denote $\Phi = (\Phi_1,\Phi_2)$ as the flow map defined on the half-plane with the velocity $v$ (as long as $p_1$ is well-defined,	
): namely, for each $x \in \bbR^2_+$, we have
\begin{equation*}
	\begin{split}
		\frac{d}{dt} \Phi(t,x)  = v(t,\Phi(t,x)), \qquad \Phi(0,x) = x. 
	\end{split}
\end{equation*}
This flow map is obtained by appropriately shifting the flow map associated with the original velocity $u$ in the $x_{1}$ axis.  For each $t\ge0$, we denote the inverse of $\Phi(t,\cdot): \bbR^2_+ \to \bbR^2_+$ by $\Phi_{t}^{-1}$.
With the intermediate parameter $\lambda>0$, we now decompose $\rho$ into the ``quadrant part'' $\rho_{Q}$ and the small remainder $\tld{\rho}$, and decompose the velocity correspondingly:
\begin{equation*}
	\begin{split}
		\rho = \rho_{Q} + \tld{\rho}, \qquad v = v_{Q} + \tld{v} - \dot{p}_{1} \bfe_{1}, \qquad v_Q := K[\rho_Q], \qquad \tld{v} = K[\tld\rho], 
	\end{split}
\end{equation*}
where we write $$Q := \{ x \in \bbR^2_+ \, : \, x_{1} > -D \}$$ and 
\begin{equation}\label{eq:D-def}
	\begin{split}
		\rho_{Q}(t,x) := \rho(t,x)\mathbf{1}_{ \{ x_{1} > - D \} }, \qquad D = D(\lambda) := \frac{C_{0}}{\lambda^{1/2}}.
	\end{split}
\end{equation}

\subsection{Bootstrap hypotheses} 
We introduce the bootstrap hypotheses in this section.

\subsubsection{Local versus far field particles}
Given $T>0$, we consider $t\in[0,T]$ and distinguish two types of particles initially at $Q=\{x_1>C_0/\sqrt{\lambda}\}$. The ones which have ever entered the square $$S_{R} := [-R,R]\times [0,R]\subset Q$$ within $[0,t]$ are considered as \textit{local} particles $A^{loc,t}$. We recall that $R >0$ is large (see \eqref{defn_C_0}) but is fixed independently of $\lambda$. 
It will be shown that
these particles, once they visited at the square $S_R$ in their past,  their height is uniformly bounded until they exit $Q$ by some large constant $H$, which is independent of smallness of $\lambda$. The others are defined by \textit{far field} particles $A^{far,t}$: More precisely,
\begin{equation}\label{defn_A_loc_far}\begin{split}
	&A^{loc,t}:=\{z\in Q\,|\, \exists s\in[0,t] \quad\mbox{such that}\quad  \Phi(s,z)\in S_R\},\\
	&A^{far,t}:=Q\setminus A_{loc,t}=\{z\in Q\,|\,\forall s\in[0,t],\quad  \Phi(s,z)\notin S_R\}. 
\end{split}
\end{equation} Note that $A^{loc,0} = S_{R}$ and $A^{loc,t'} \supset A^{loc,t}$ for $t'>t$. 

\subsubsection{Notations} 
We define the ``macroscopic'' quantities in $Q$ \begin{equation*}
\begin{split}
	\mu_Q(t) := \mu[\rho_{Q}(t,\cdot)] = \int_{ Q } x_{2}\rho(t,x) dx,
\end{split}
\end{equation*}
\begin{equation*}
\begin{split}
	|\mu|_Q(t) := \mu[|\rho_{Q}(t,\cdot)|] = \int_{ Q } x_{2}|\rho(t,x)| dx=\|\rho_Q(t)\|_{L^1_*},
\end{split}
\end{equation*}
\begin{equation*}
\begin{split}
	\kpp_Q(t) := \kpp[\rho_{Q}(t,\cdot)] = \left(\int_{ Q } |\rho(t,x)|^{2} dx \right)^{1/2},  
\end{split}
\end{equation*}
\begin{equation*}
\begin{split}
	E_Q(t) := E[\rho_{Q}(t,\cdot)] = \frac12\int_{ Q } ((-\lap_{\bbR^2_+})^{-1}\rho_{Q}) (t,x) \rho(t,x) dx .  
\end{split}
\end{equation*}

\subsubsection{Choice of bootstrap hypotheses} \label{subsubsec:boot}
For some $T>0$, we assume that the following holds for the time interval $[0,T]$: for each $t\in[0,T]$,
\begin{equation}\label{eq:B1} \tag{B1}
\begin{split}
	p_{1}(t)\quad \mbox{ satisfies }\quad{\eqref{eq:p1-def}\quad\mbox{and}\quad \eqref{eq:p1-def2}} \quad \mbox{ and }\quad  \nrm{ \rho_{Q}(t,\cdot) - \omg_{Lamb} }_{L^2 \cap L^1_*} \leq  {\frac{\varep}{10}}. 
\end{split}
\end{equation}\begin{equation}\label{eq:B2} \tag{B2}
\begin{split}
	{\int_{Q}x_2|\rho(t,x)|\mathbf{1}_{\Phi(t,A^{far,t})}(x) dx\leq\frac{\lambda}{10}.		}
\end{split}
\end{equation}
\begin{equation}\label{eq:B3} \tag{B3}
\begin{split}
	\kpp_{Q}(t) \le \kpp[\omg_{Lamb} ] + \lambda . 
\end{split}
\end{equation}\begin{equation}\label{eq:B4} \tag{B4}
\begin{split}
	{|\mu|_{Q}(t) \le \mu[\omg_{Lamb}] + \lambda.} 
\end{split}
\end{equation}\begin{equation}\label{eq:B5} \tag{B5}
\begin{split}
	E_{Q}(t) \ge E[\omg_{Lamb} ] - \lambda. 
\end{split}
\end{equation} \begin{remark}
The integral in \eqref{eq:B2} is equal to
$$
\int_{A^{far,t}\cap\left( \Phi^{-1}_{t}(Q)\right)} \Phi_2(t,z)|\rho_0(z)|\,dz.
$$
\end{remark}

\subsection{Bootstrap hypotheses for a short time}\label{subsec:bootstrap-short}

Based on the initial assumptions and continuity of the solution in time, we show that the assumptions  hold at least for a short time interval. 
\begin{proposition}\label{pro:starting time}
Let $\omg_{0}$ satisfy the assumptions of Theorem \ref{thm:main}. Then 
\begin{equation}\label{est_p_1(0)}
	|p_1(0)|\leq 5\dlt 
\end{equation} and
there exists $T_{0}>0$ such that all the bootstrap assumptions  \eqref{eq:B1}--\eqref{eq:B5} hold on $[0,T_{0}]$. 
\end{proposition}
\begin{proof}
By recalling $\nrm{\omg_{0}^{(+)} - \omg_{Lamb} }_{L^2 \cap L^1_{*}} < 4\dlt$, we simply assume $\delta>0$ small (if necessary) to have 
$4\delta\leq \varep/100$, which is equivalent to make $C_4$ large enough in \eqref{small_dlt}.
Then by continuity of $\nrm{\omg^{(+)}(t,\cdot) - \omg_{Lamb} }_{L^2 \cap L^1_{*}} $ in $t$, 
there exists $T_{0} = T_{0}(\omg_{0},\varep) > 0$ such that \begin{equation*}
	\begin{split}
		\nrm{\omg^{(+)}(t,\cdot) - \omg_{Lamb} }_{L^2 \cap L^1_{*}} < \varep/50, 
		\quad t \in [0,T_{0}]. 
	\end{split}
\end{equation*}
By \eqref{L2_to_L1}, we get for $t \in [0,T_{0}]$ 
\begin{equation}\label{eq:near-initial-closeness_}
	\nrm{\omg^{(+)}(t,\cdot) - \omg_{Lamb} }_{L^1}
	\leq 4 \nrm{\omg^{(+)}(t,\cdot) - \omg_{Lamb} }_{L^2} 	+\nrm{\omg^{(+)}(0,\cdot) - \omg_{Lamb} }_{L^1} 
	< \frac{4\varep}{50}+\delta\leq \frac{\varep}{10}. \end{equation}

We first check \eqref{eq:B1}; note from the definition \eqref{eq:h} of $h(t,p)$ that \begin{equation*}
	\begin{split}
		(\rd_t h)(t,p) = -\int \nb\cdot( u \omg^{(+)} )(t,x) g(x_1 - p) dx = \int (u_{1}\omg^{(+)})(t,x) g'(x_1 - p) dx , 
	\end{split}
\end{equation*} which shows that $\rd_t h$ is a continuous function of $t$ and $p$. Moreover, \begin{equation*}
	\begin{split}
		(\rd_p h)(t,p) = - \int \omg^{(+)}(t,x) g'(x_{1}-p) dx 
	\end{split}
\end{equation*} is also continuous in $t$ and $p$, which implies that $h \in C^{1}_{t,p}$. 

Assuming that $|p| \le 2$, we observe that 
\begin{equation*}
	\begin{split}
		h(t,p) &=  \int ( \omg^{(+)}(t,x) - \omg_{Lamb}(x)) g(x_{1}-p) dx  + \int \omg_{Lamb}(x) g(x_{1}-p) dx \\ 
		& = 	 \int ( \omg^{(+)}(t,x) - \omg_{Lamb}(x)) g(x_{1}-p) dx - p\int \omg_{Lamb}(x) dx 
	\end{split}
\end{equation*}and 
\begin{equation*}
	\begin{split}
		(\rd_{p}h)(t,p) &=  -\int ( \omg^{(+)}(t,x)-  \omg_{Lamb}(x)) g'(x_{1}-p) dx  - \int \omg_{Lamb}(x) g'(x_{1}-p) dx \\ 
		& = 	- \int ( \omg^{(+)}(t,x) - \omg_{Lamb}(x)) g'(x_{1}-p) dx  -	\int \omg_{Lamb}(x) dx . 
	\end{split}
\end{equation*} Both of these identities follow from that $g(x_{1}-p) = (x_{1}-p)$ and $g'(x_{1}-p) = 1$ on the support of $\omg_{Lamb}$ when $|p|\le2$. By recalling \eqref{defn_m_L},
we estimate, from \eqref{eq:near-initial-closeness_}, 
for $|p|\leq 2$,
\begin{equation*}
	\begin{split}
		|h(t,p)+p m_L| &\leq 	\|g\|_{L^\infty}\cdot
		\nrm{\omg^{(+)}(t,\cdot) - \omg_{Lamb} }_{L^1}\leq 4\cdot \frac{\varep}{10}\leq{\varep}.
	\end{split}
\end{equation*} In particular, we have 
$ \pm h(t, \pm 2) < 0$ for $t \in [0,T_{0}]$.	
Thus, on this time interval,  there exists a position $p_1(t)\in[-2,2]$ satisfying
\eqref{eq:p1-def}. 
Moreover at the origin,
we know \begin{equation}\label{eq:est_h_init}
	|h(t,0)|\leq \varep \quad\mbox{and} \quad  |h(0,0)| \leq 	\|g\|_{L^\infty} 
	\nrm{\omg^{(+)}(0,\cdot) - \omg_{Lamb} }_{L^1}\leq 4 \times  {3\delta}=12\delta.
\end{equation}  
Next, 
we estimate similarly
\begin{equation}\label{eq:est_h_P_init_}
	\begin{split}
		|(\partial_p h)(t,p)+m_L| &\leq 	\|g'\|_{L^\infty} 
		\nrm{\omg^{(+)}(t,\cdot) - \omg_{Lamb} }_{L^1}\leq 2 \times  \frac{\varep}{10}\leq{\varep},
	\end{split}
\end{equation} 
which yields
\begin{equation}\label{eq:est_h_P_init}
	(\partial_p h)(t,p)\leq -m_L/2.
\end{equation}
The estimates
\eqref{eq:est_h_init} and \eqref{eq:est_h_P_init} implies that
there exists $$p_1(0)\in[-\frac{2}{m_L}  12 \delta,+\frac{2}{m_L}  12 \delta]\subset[-5\dlt, 5\dlt]
$$ satisfying \eqref{eq:p1-def}.
Now we have
$(\rd_{p}h)(0,p_{1}(0)) \ne 0$. By the Implicit Function Theorem, by shrinking $T_{0}>0$ if necessary, we have that $p_{1}(\cdot_t) \in C^{1}([0,T_{0}])$, with  \begin{equation}\label{eq:deriv-p}
	\dot{p}_{1}(t) = \frac{(\partial_t  h)(t, p_1(t))}{(-\partial_p h)(t, p_1(t))}.
\end{equation} 
From now on we estimate	$|\dot{p}_{1}(0) - 1 |$.
First we compute 
\begin{equation}\label{eq:deriv-t-H_}
	\begin{split}
		(\partial_t  h)(t, p)&=\int_{\mathbb{R}^2_{+}} u_1(t, x+p\bfe_1)\omega^{(+)} (t, x+p\bfe_1) g'(x_1) dx\\
		&=\int_{\mathbb{R}^2_{+}} u_{Lamb, 1}(  x )\omega_{Lamb} (  x ) g'(x_1) dx\\
		&\quad+\int_{\mathbb{R}^2_{+}} [u_1(t, x+p\bfe_1)-u_{Lamb, 1}(  x )] \omega^{(+)} (t, x+p\bfe_1) g'(x_1) dx\\
		&\quad+\int_{\mathbb{R}^2_{+}} u_{Lamb, 1}(  x ) [\omega^{(+)}(t, x+p\bfe_1) -\omega_{Lamb} ( x ) ]g'(x_1) dx.
	\end{split}
\end{equation}
For the first term on the right hand side of \eqref{eq:deriv-t-H_}, we use the fact that $\omega_{Lamb}$ is a traveling wave solution with propagation speed $1$, which  yields the equation $$(u_{Lamb}-\bfe_1)\cdot \nabla \omega_{Lamb}=0.$$ Testing this equation by $g(x_1)$ and integrating by parts yields 
\begin{equation}\label{ml_integral}
	\int_{\mathbb{R}^2_{+}} u_{Lamb, 1}(  x )\omega_{Lamb} (  x ) g'(x_1) dx=\int_{\mathbb{R}^2_{+}} \omega_{Lamb} (  x ) g'(x_1) dx=\int_{\mathbb{R}^2_{+}} \omega_{Lamb} (  x )  dx=m_L.
\end{equation}
Thus we have
\begin{equation*} 
	\begin{split}
		|(\partial_t  h)(t, p)-m_L| & \leq
		2 \int_{\mathbb{R}^2_{+}}\left| [u_1(t, x+p\bfe_1)-u_{Lamb, 1}(  x )] \omega^{(+)} (t, x+p\bfe_1) \right| dx\\
		&\quad+2\int_{\mathbb{R}^2_{+}} \left|u_{Lamb, 1}(  x ) [\omega^{(+)}(t, x+p\bfe_1) -\omega_{Lamb} ( x ) ]\right| dx\\
		& \leq
		C \|u_1(t, x+p\bfe_1)-u_{Lamb, 1}(  x )\|_{L^\infty}  + 
		C\|\omega^{(+)}(t, x+p\bfe_1) -\omega_{Lamb} ( x ) \|_{L^1}
	\end{split}
\end{equation*}
At the initial time, we have 
\begin{equation*}
	\begin{split}
		&\|\omega^{(+)}(0, x+p_1(0)\bfe_1) -\omega_{Lamb} ( x ) \|_{L^2}\leq\|\omega(0, x+p_1(0)\bfe_1) -\omega_{Lamb} ( x ) \|_{L^2}\\&=\|\omega(0, x) -\omega_{Lamb} ( x-p_1(0)\bfe_1 ) \|_{L^2}\leq
		\|\omega(0, x) -\omega_{Lamb} ( x ) \|_{L^2}+
		\|\omega_{Lamb} ( x )-\omega_{Lamb} ( x-p_1(0)\bfe_1 ) \|_{L^2}\\&
		\leq   \delta +C\cdot  {C_{Lip}\cdot |p_1(0)|}
		\leq C  {\dlt},
	\end{split}
\end{equation*} where $C_{Lip}$ is the Lipschitz constant \eqref{defn_C_Lip},
and similarly
\begin{equation*}
	\begin{split}
		&\|\omega^{(+)}(0, x+p_1(0)\bfe_1) -\omega_{Lamb} ( x ) \|_{L^1}\leq\|\omega(0, x+p_1(0)\bfe_1) -\omega_{Lamb} ( x ) \|_{L^1}\\&=\|\omega(0, x) -\omega_{Lamb} ( x-p_1(0)\bfe_1 ) \|_{L^1}\leq
		\|\omega(0, x) -\omega_{Lamb} ( x ) \|_{L^1}+
		\|\omega_{Lamb} ( x )-\omega_{Lamb} ( x-p_1(0)\bfe_1 ) \|_{L^1}\\&\leq 3\dlt+C\cdot C_{Lip}\cdot |p_1(0)|\leq C\dlt.
	\end{split}
\end{equation*} The above estimates controls
the initial velocity maximum by \eqref{eq:vel-L-infty-alpha2} with $\alpha=1/2$:
\begin{equation*}
	\begin{split}
		&	\|u_1(0, x+p_1(0)\bfe_1)-u_{Lamb, 1}(  x )\|_{L^\infty} \\
		&\quad \leq C\|\omega(0, x+p_1(0)\bfe_1) -\omega_{Lamb} ( x ) \|^{1/2}_{L^2}  \cdot A^{1/4}\cdot\|\omega(0, x+p_1(0)\bfe_1) -\omega_{Lamb} ( x ) \|^{1/4}_{L^1} \\
		&\quad \leq C\dlt^{1/2}\cdot A^{1/4} \cdot \dlt^{1/4}= CA^{1/4} \cdot\dlt^{3/4}.
	\end{split}
\end{equation*}
In sum,  we 
get the following result:
$$|(\partial_t  h)(0, p_1(0))-m_L|\leq CA^{1/4} \cdot\dlt^{3/4}+C\dlt\leq CA^{1/4} \dlt^{3/4}. $$
Recalling 
\eqref{eq:deriv-p} and
combining this bound and \eqref{eq:est_h_P_init_} gives 
\begin{equation*}
	\left|\frac{d}{dt} p_1(0)-1 \right| \leq  CA^{1/4}\dlt^{3/4} \leq  CA^{1/4}\varep^{3/4}  \leq \frac{1}{2}\varep^{1/2} 
\end{equation*} 
by taking $\varep$ smaller if necessary, which is equivalent to make $C_4$ large enough in \eqref{small_varep}. Therefore, by continuity of $\dot{p}_{1}(t)$ in $t$, \eqref{eq:p1-def2} is satisfied automatically on $[0,T_{0}]$ by taking $T_{0}>0$   smaller if necessary.   This means that the change of variables \eqref{eq:rho-def} is well-defined.
In particular, \begin{equation*} 
	\begin{split} \nrm{ \rho_{Q}(0,\cdot) - \omg_{Lamb} }_{L^2 \cap L^1_*} &\leq \nrm{ \rho(0,\cdot) - \omg_{Lamb} }_{L^2 \cap L^1_*} 
		= \nrm{ \omg_0(\cdot+p_1(0)\bfe_1) - \omg_{Lamb} }_{L^2 \cap L^1_*}\\&= \nrm{ \omg_0 - \omg_{Lamb}(\cdot-p_1(0)\bfe_1) }_{L^2 \cap L^1_*} \\&\leq  \nrm{ \omg_0 - \omg_{Lamb}  }_{L^2 \cap L^1_*}+ \nrm{ \omg_{Lamb} - \omg_{Lamb}(\cdot-p_1(0)\bfe_1) }_{L^2 \cap L^1_*}\\& \leq \delta +C\cdot C_{Lip}\cdot|p_1(0)|\leq C\dlt \leq \frac{\varep}{20}\end{split}
\end{equation*} by     shrinking $\dlt$   if necessary.
Then we obtain by continuity in time that $\nrm{ \rho_{Q}(t,\cdot) - \omg_{Lamb} }_{L^2 \cap L^1_*} < {\varep}/10$ holds for $t \in [0,T_{0}]$,   by     shrinking   $T_{0}$ if necessary. This completes the proof that \eqref{eq:B1} holds for $t \in [0,T_{0}]$ for some $T_0>0$.

For \eqref{eq:B2}, we want, for $t \in [0,T_{0}]$, 
\begin{equation}\label{B2_recall}
	\int_{Q}x_2|\rho(t,x)|\mathbf{1}_{\Phi(t,A^{far,t})}(x) dx\leq\frac{\lambda}{10}.
\end{equation}
From $A^{far,0}= Q\setminus S_R,$ we 
simply estimate, at the initial time,
\begin{equation*}
	\begin{split}
		\int_{Q}x_2|\rho(0, x)|\mathbf{1}_{A^{far,0}}(x) dx
		&= \int_{Q\setminus S_R}x_2|\rho(0, x)|  dx\leq  \int_{\hat{Q}\setminus \check{S_R}}x_2|\omg_0(x)|  dx\\&\leq \int_{\{|x|>1\}}|x||\omg_0(x)|  dx\leq \dlt\leq \frac{\lambda}{20},
	\end{split}
\end{equation*}  where we define
$\hat{Q}:=\{x_1>-D-1\}   \supset Q$ and 
$\check{S_R}:=[-R+1, R-1]\times [0,R]\subset S_R$. Here we used $|p_1(0)|\leq 1$. Then continuity 
of the left hand side of \eqref{B2_recall}	 
in time gives \eqref{eq:B2}. 

To check the remaining properties \eqref{eq:B3}--\eqref{eq:B5}, we observe that at the initial time, \begin{equation*}
	\begin{split}
		&	\nrm{ \omg_{0} }_{L^2} < \nrm{ \omg_{Lamb} }_{L^2} + \dlt=\kappa[\omg_{Lamb}] + \dlt\leq \kappa[\omg_{Lamb}] + \frac{\lambda}{2}\qquad\mbox{and}\\
		&|\mu|[\omg_0]=\nrm{ \omg_{0} }_{L^1_*} < \nrm{ \omg_{Lamb} }_{L^1_*} + \dlt = \mu[\omg_{Lamb} ] + \dlt\leq \mu[\omg_{Lamb} ] + \frac{\lambda}{2}. 
	\end{split}
\end{equation*} 
Next, we decompose
\begin{equation*}
	\begin{split}\left| 	E[\rho_Q(0)] - E[\omg_{Lamb}]\right|&\leq
		\left| 	E[\rho_Q(0)] - E[\rho(0)] \right|+\left| 	E[\rho(0)] - E[\omg_{Lamb}]\right|.\\
	\end{split}
\end{equation*}
For the first term, by \eqref{eq:Energy-difference-X}, we estimate
\begin{equation*}
	\begin{split}
		\left| 	E[\rho_Q(0)] - E[\rho(0)] \right| &\le C(\nrm{\tilde\rho(0)}_{L^{1}_{*}(\bbR^2_+)} \nrm{\tilde\rho(0)}_{L^{2}(\bbR^2_+)}\nrm{\rho_Q(0)+\rho(0) }_{L^{1}_{*}(\bbR^2_+)} \nrm{\rho_Q(0)+\rho(0)}_{L^{2}(\bbR^2_+)})^{1/2}\\
		&\le C(\nrm{\omg_0}_{L^{1}_{*}(\{x_1\leq -D+1\})} \nrm{\omg_0}_{L^{2}(\{x_1\leq -D+1\})}\nrm{\rho(0) }_{L^{1}_{*}(\bbR^2_+)} \nrm{\rho(0)}_{L^{2}(\bbR^2_+)})^{1/2}\\
		&\le C(\nrm{\omg_0-\omg_{Lamb}}_{L^{1}_{*}(\ohp)} \nrm{\omg_0-\omg_{Lamb}}_{L^{2}(\ohp)}\nrm{\omg_0 }_{L^{1}_{*}(\bbR^2_+)} \nrm{\omg_0}_{L^{2}(\bbR^2_+)})^{1/2}\\
		&\le C(\nrm{\omg_0-\omg_{Lamb}}_{L^{1}_{*}(\ohp)} \nrm{\omg_0-\omg_{Lamb}}_{L^{2}(\ohp)})^{1/2}
		\leq C(\delta\cdot\dlt )^{1/2}=C\dlt\leq \frac{\lambda}{4}.\\
	\end{split}
\end{equation*}
For the second term, again by \eqref{eq:Energy-difference-X}, we estimate
\begin{equation}\label{eq:ini_en}
	\begin{split}
		&\left| 	E[\rho(0)] - E[\omg_{Lamb}]\right|	
		=\left| 	E[\omg_0] - E[\omg_{Lamb}]\right|\\
		&\leq C(\nrm{\omg_0-\omg_{Lamb}}_{L^{1}_{*}(\bbR^2_+)} \nrm{\omg_0-\omg_{Lamb}}_{L^{2}(\bbR^2_+)}\nrm{\omg_0+\omg_{Lamb}}_{L^{1}_{*}(\bbR^2_+)} \nrm{\omg_0+\omg_{Lamb}}_{L^{2}(\bbR^2_+)})^{1/2}\\
		&\leq C(\delta\cdot\dlt )^{1/2}=C\dlt\leq \frac{\lambda}{4}.\\
	\end{split}
\end{equation}
Thus we get
\begin{equation*}
	\begin{split}\left| 	E[\rho_Q(0)] - E[\omg_{Lamb}]\right|& \leq \frac{\lambda}{2}.
	\end{split}
\end{equation*}  
Lastly, we deduce \eqref{eq:B3}--\eqref{eq:B5} simply from continuity in time of $\kpp_{Q}(t), |\mu|_{Q}(t)$ and $E_{Q}(t)$,  by     shrinking   $T_{0}$ again if necessary. 
\end{proof} 

\section{Consequences of the bootstrap hypotheses}\label{sec:conseq} In this section, we collect several consequences of the bootstrap hypotheses. {This section is organized as follows. In \S \ref{subsec:small-tilde}, we obtain smallness of $\tilde\rho$. Then, in \S \ref{subsec:main-error}, we introduce the main-error decomposition of $\rho_{Q}$ and obtain error smallness. 
In  \S \ref{subsec:vertical-speed}, smallness of the \textit{vertical} component of the  velocity in \textit{local} region is obtained (Lemma \ref{lem:vel-Q-decomp}). This vertical smallness gives \textit{height} control of particle trajectories for \textit{local} particles (Lemma \ref{lem:S_R_to_H}). 
In  \S \ref{subsec:height}, the height control for local particles produces smallness of the {vertical} component of the  velocity in \textit{far} region (Lemma \ref{lem:v_2_on_S}), which gives the height control for far particles (Lemma \ref{lem:far_cont}). 

Even when it is not explicitly mentioned, we always assume  that
all the bootstrap hypotheses \eqref{eq:B1}--\eqref{eq:B5} on the interval
$ [0,T]$ hold. In this section, we will freely use these  hypotheses.}

\subsection{Smallness of $\tilde\rho$}\label{subsec:small-tilde}

We begin with lower bounds for 
$\kpp_{Q}$ and $|\mu|_Q$, which give a control on $\tilde\rho$.
\begin{lemma}\label{lem:rem-small}
We have \begin{equation}\label{eq:kpp-mu-Q-lower-bounds}
	\begin{split}
		\kpp_{Q}(t) \ge \kpp[\omg_{Lamb}] - 40\lambda, \qquad |\mu|_Q(t) \ge \mu[\omg_{Lamb}] - 9\lambda
	\end{split}
\end{equation} and \begin{equation}\label{eq:rem-small}
	\begin{split}
		\nrm{ \tld{\rho}(t,\cdot) }_{L^{2}} \le 30\lambda^{1/2},\quad \nrm{ \tld{\rho}(t,\cdot) }_{L^{1}} \le 60\lambda^{1/2}, \quad   \nrm{ \tld{v}(t,\cdot) }_{L^{\infty}} \le 
		{C}A^{1/2} \lambda^{1/4}.
	\end{split}
\end{equation}
where ${C}$ is a universal constant.	
\end{lemma}
\begin{proof} To see \eqref{eq:kpp-mu-Q-lower-bounds}, we observe from \eqref{eq:B3}--\eqref{eq:B5} and the sharp energy inequality that \begin{equation*}
	\begin{split}
		\kpp[\omg_{Lamb}] \mu[\omg_{Lamb}] - \frac{\lambda}{C_{L} } \le  \frac{E_Q(t)}{C_{L}} \le \kpp_Q(t)\nrm{\rho_{Q}(t,\cdot)}_{L^1_*}  \le \kpp_Q(t) ( \mu[\omg_{Lamb}] + \lambda ) 
	\end{split}
\end{equation*} so that \begin{equation*}
	\begin{split}
		\kpp_Q(t) \ge \kpp[\omg_{Lamb}] \frac{ \mu[\omg_{Lamb}] }{\mu[\omg_{Lamb}]+4\lambda} - \frac{1}{C_{L}( \mu[\omg_{Lamb}] + 4\lambda )} \lambda \ge \kpp[\omg_{Lamb}] - 40\lambda.
	\end{split}
\end{equation*} 
Similarly, we can show 
$$
\nrm{\rho_{Q}(t,\cdot)}_{L^1_*} \ge \mu[\omg_{Lamb}] - 9\lambda.
$$ Then observe that $ \nrm{\omg_{0}}_{L^2}^2=\nrm{\omg(t)}_{L^2}^2=\nrm{\rho(t)}_{L^2}^2 = \nrm{ \tld{\rho}(t,\cdot) }_{L^{2}}^2 + (\kpp_Q(t))^{2}$ implies, with \eqref{eq:mainthm-assumptions}, \begin{equation*}
	\begin{split}
		\nrm{ \tld{\rho}(t,\cdot) }_{L^{2}}^2 \le (\kpp[\omg_{Lamb}] + \dlt)^2 - ( \kpp[\omg_{Lamb}] - 40\lambda )^{2} \le 100  \kpp[\omg_{Lamb}] \lambda. 
	\end{split}
\end{equation*} 
This gives the first inequality of \eqref{eq:rem-small}. 
For the $L^1$-norm, we decompose \begin{equation}\label{est:L2toL1}
	\begin{split}
		\nrm{ \tld{\rho}(t,\cdot) }_{L^{1}} &=
		\int_{\mathbb{R}^2_+\setminus Q}|\tilde{\rho}(t,x)|\mathbf{1}_{\{y\in\mathbb{R}^2_+\,:\,|\Phi^{-1}_{t}(y)|\leq 1 \}}(x)\,dx +	\int_{\mathbb{R}^2_+\setminus Q}|\tilde{\rho}(t,x)|\mathbf{1}_{\{y\in\mathbb{R}^2_+\,:\,|\Phi^{-1}_{t}(y)|> 1 \}}(x)\,dx\\&\leq  \nrm{ \tld{\rho}(t,\cdot) }_{L^{2}} \sqrt{|\{y\in\mathbb{R}^2_+\,:\,|y|\leq 1\}|}+\int_{\{x\in\mathbb{R}^2_+\,:\,|x|>1\}}|x||\omg_0(x)|dx\\&\leq  30\lambda^{1/2} \sqrt{\pi/2}+\dlt\leq 60\lambda^{1/2}.
	\end{split}
\end{equation}
The last inequality of \eqref{eq:rem-small} then follows from \eqref{eq:vel-L-infty-alpha2} with $\alpha=1$. 
\end{proof}

\begin{lemma}[Global velocity estimate] \label{lem:borderline-v} We have 
\begin{equation}\label{eq:global-v}
	\begin{split}
		\nrm{ v(t,\cdot) - u_{Lamb} + \begin{pmatrix}
				1 \\ 0 
		\end{pmatrix} }_{L^\infty}  \le 2\varep^{1/2} .
	\end{split}
\end{equation} In particular, for $|x| \ge 1$ we have that \begin{equation}\label{eq:borderline-v}
	\begin{split}
		|v_{1}(t,x) + 1|, |v_{2}(t,x)|
		\le 2\varep^{1/2} + \frac{2}{|x|^2}
		\le 2\varep^{1/2} + \frac{2}{|x_{1}|^2}.
	\end{split}
\end{equation} 
\end{lemma}
\begin{proof}
We recall that $|v_{1}(t,x) + 1| = |v_{Q,1} + \tld{v}_{1} - \dot{p}_{1} + 1 |$ and $|v_{2}(t,x) | = |v_{Q,2} + \tld{v}_{2}|$.

The estimate on $|\tld{v}|$ follows from \eqref{eq:rem-small} by 
choosing $\lambda$ small depending on $A$ and $\varep$:
\begin{equation*}
	\begin{split}
		\nrm{\tld{v}}_{L^\infty} 
		\leq {C}A^{1/2} \lambda^{1/4}			 
		\le \frac18\varep^{1/2}. 
	\end{split}
\end{equation*} 
The estimate of $|1 - \dot{p}_{1}|$ is given in \eqref{eq:B1}. Next, we estimate $v_{Q}$ by separately considering the contribution from $\omg_{Lamb}$ and $\rho_{Q} - \omg_{Lamb}$: using \eqref{eq:B1} with \eqref{eq:vel-L-infty-alpha2} with $\alpha=1/4$, \begin{equation*}
	\begin{split}
		\nrm{ \nb^\perp\lap^{-1}[ \rho_{Q} - \omg_{Lamb} ] }_{L^\infty} & \le C \nrm{\rho_{Q} - \omg_{Lamb} }_{L^2}^{3/4} \nrm{\rho_{Q} - \omg_{Lamb} }_{L^1 \cap L^\infty}^{1/4} \\
		& \le C\varep^{3/4}(4A)^{1/4} \le \frac18\varepsilon^{1/2} 
	\end{split}
\end{equation*} 
by assuming $\varep$ small depending on $A$, if necessary. This proves \eqref{eq:global-v}. 

Lastly, when $|x|\ge 1$, $u_{Lamb}$ is bounded by $2|x|^{-2} \le 2|x_{1}|^{-2}$ as explicitly seen by \eqref{eq:Lamb-vel}. 
\end{proof}

\begin{remark}
As a corollary, for any $x\notin S_R=[-R,R]\times [0,R]$,  we have
\begin{equation}\label{eq:vel_out_S_R}
	|v_1(t,x)+1|\leq 2\varep^{1/2} +\frac{2}{R^2}\leq \frac{1}{4}\quad\mbox{and}\quad |v_2(t,x)|\leq \frac14.
\end{equation}
Therefore, we have $
v_1(x)<-1/2$  on $\{x_1=-D\},
$ which implies, for any $p$, $
	\|\rho_Q(t)\|_{L^p}\leq \|\rho_Q(0)\|_{L^p}.$ \end{remark} 
	
	\subsection{Main--error decomposition}\label{subsec:main-error} 
	We now decompose $\rho_{Q}$ further:  \begin{equation*}
\begin{split}
	\rho_{Q} = \rho_{Q}^{main} + \rho_{Q}^{err}, \quad \mbox{where} \quad  \rho_{Q}^{main} := \rho_{Q} \mathbf{1}_{ \{ x_{1} > - R \} }, \quad \rho_{Q}^{err} := \rho_{Q} \mathbf{1}_{ \{ -D < x_{1} \le - R  \} }.
\end{split}
\end{equation*} 

\begin{lemma}[Smallness of the error]\label{lem:error-small} We have 
\begin{equation}\label{eq:error-small}
	\begin{split}
		\nrm{ \rho_{Q}^{err}(t)}_{L^1_*} & \le 30\lambda, \, 
		\nrm{ \rho_{Q}^{err}(t)}_{L^2} \le  {30}\lambda^{1/2}, \,
		\nrm{ \rho_{Q}^{err}(t)}_{L^1} \le  {60}\lambda^{1/2}, \, 
		\nrm{ v_{Q}^{err}(t)}_{L^\infty} \leq C A^{1/4} \lambda^{3/8}, 
	\end{split}
\end{equation} where ${C}$ is a universal constant.	

\end{lemma}
\begin{proof}
Define $\rho_{rem}:=\rho_Q-\omega_{Lamb}$ and $\psi_{rem}:=\psi[\rho_{rem}]$. By using the explicit formula of $\psi_{Lamb}$ in \eqref{eq:Lamb-psi}, $R\geq 10$, \eqref{eq:vel-L-infty-alpha2} and the bootstrap assumption $\nrm{\rho_{rem}}_{L^2}\leq \eps$,   we get for $x\in \mathbb R^2_+$ with $x_1\leq-R$,
\begin{equation}\label{eq: str-fun-R}
	\begin{split}
		|\psi_{Q}(x)| \leq |\psi_{Lamb}(x)|+|x_2^{-1}\psi_{rem}(x)| x_2 
		\leq \frac{x_2}{10}.
	\end{split}
\end{equation}
Here we assumed $\varep$ small enough to get
$
C\varep^{1/2} A^{1/2} \leq 1/20
$.
Similarly,
denoting 	 $\psi_{Q}^{err}=\psi[\rho_{Q}^{err}]$ and $\psi_{Q}^{main}=\psi[\rho_{Q}^{main}]$ 
and  using $\nrm{ \rho_{Q}^{err}}_{L^2}\leq \nrm{\rho_{rem}}_{L^2}\leq \eps$ and \eqref{eq: str-fun-R},
we have  	
for  $x_1\leq-R$,
\begin{equation}\label{eq: str-fun2-R}
	\begin{split}
		|\psi_{Q}^{err}(x)| 
		\leq \frac{x_2}{10},\quad \ \text{and}\quad \ |\psi_Q^{main}(x)|\leq \frac{x_2}{5}.
	\end{split}
\end{equation}
We deduce from \eqref{eq:B4} 
that 
\begin{equation*}
	\nrm{\rho_Q^{main}}_{L^1_*}=\nrm{\rho_Q}_{L^1_*}-\nrm{\rho_Q^{err}}_{L^1_*}\leq  \mu[\omg_{Lamb}] + \lambda-\nrm{\rho_Q^{err}}_{L^1_*}.
\end{equation*}
Now we write \begin{equation*}
	\begin{split}
		E_{Q} = \frac12 \int (\rho_Q^{err}+ \rho_Q^{main})(\psi_Q^{err}+ \psi_Q^{main}) dx = \frac12 \int \rho_Q^{err} (\psi_Q^{err}+2 \psi_Q^{main}) dx+\frac12\int \rho_Q^{main} \psi_Q^{main} dx. 
	\end{split}
\end{equation*} Then, using that $\rho^{err}_{Q}$ is supported in $\{ x_{1} \le -R \}$, we can use \eqref{eq: str-fun-R}, \eqref{eq: str-fun2-R} 
to estimate \begin{equation*}
	\begin{split}
		\frac12 \int \rho_Q^{err} (\psi_Q^{err}+2 \psi_Q^{main}) dx \le  \frac14 \nrm{\rho_Q^{err}}_{L^1_*}. 
	\end{split}
\end{equation*} Next, with the sharp energy inequality, \begin{equation*}
	\begin{split}
		\frac12\int \rho_Q^{main} \psi_Q^{main} dx \le C_L\nrm{\rho_Q^{main}}_{L^1_*}\nrm{\rho_Q^{main}}_{L^2} \le C_L(\mu[\omg_{Lamb}] + \lambda-\nrm{\rho_Q^{err}}_{L^1_*})(\kpp[\omg_{Lamb}] + \lambda). 
	\end{split}
\end{equation*} Combining these two estimates, we derive  
\begin{equation}\label{eq: decom-energy}
	\begin{split}
		E_Q
		&\le \frac14 \nrm{\rho_Q^{err}}_{L^1_*} + C_L\nrm{\rho_Q^{main}}_{L^1_*}\nrm{\rho_Q^{main}}_{L^2}  \\
		&\leq \frac14 \nrm{\rho_Q^{err}}_{L^1_*} +  C_L(\mu[\omg_{Lamb}] + \lambda-\nrm{\rho_Q^{err}}_{L^1_*})(\kpp[\omg_{Lamb}] + \lambda)\\
		&\leq E[\omg_{Lamb}]-  \frac34 \nrm{\rho_Q^{err}}_{L^1_*}
		+ {20\lambda},			
	\end{split}
\end{equation} where we used the identity $E[\omega_{Lamb}]=C_L \kappa[\omega_{Lamb}] \mu[\omega_{Lamb}]$ {and} $C_L \kappa[\omega_{Lamb}]=1.$
Then combining this with \eqref{eq:B5} yields 
$$E[\omg_{Lamb}]-\lambda\leq E_Q\leq E[\omg_{Lamb}]-  \frac34 \nrm{\rho_Q^{err}}_{L^1_*}+ {20\lambda},	
$$ 
which gives the impulse bound $$\nrm{\rho_Q^{err}}_{L^1_*}\leq 30 \lambda.$$
Let us denote 
$\kpp_Q^{main}	=\nrm{\rho_Q^{main}}_{L^2}$ and 
$\kpp_Q^{err}	=\nrm{\rho_Q^{err}}_{L^2}$.
Plugging this back into the first inequality of \eqref{eq: decom-energy} and combing \eqref{eq:B5} again, we obtain
$$C_L\mu[\omg_{Lamb}]\kpp[\omg_{Lamb}]-\lambda\leq E_Q\leq  8 \lambda  +C_L(\mu[\omg_{Lamb}] + \lambda)\kpp_Q^{main}\leq C_L\mu[\omg_{Lamb}]\kpp_Q^{main}+10\lambda.$$ 
This gives $$\kpp_Q^{main}\geq \kpp[\omg_{Lamb}]-\frac{11\lambda}{C_L \mu[\omg_{Lamb}] }\geq \kpp[\omg_{Lamb}]-55\lambda.$$
So we have $$(\kpp[\omg_{Lamb}]+\lambda)^2-( \kpp_Q^{err})^2\geq  \kpp_{Q}^2-( \kpp_Q^{err})^2=(\kpp_Q^{main})^2\geq (\kpp[\omg_{Lamb}]-55\lambda)^2,$$
which immediately yields
$\kpp_Q^{err}\leq 30\lambda^{1/2}.$ Then the $L^1$-norm estimate for $\rho_Q^{err}$ follow by a similar argument with 
\eqref{est:L2toL1}. The velocity estimate follows from \eqref{eq:vel-L-infty-alpha2} with $\alpha=1/2$.
\end{proof}

\subsection{Vertical speed and height control I}\label{subsec:vertical-speed}

\subsubsection{Vertical speed in local region}
We divide $Q$ into three time-independent disjoint  regions: \begin{equation*}
\begin{split}
	Q_{R} := \left\{ x_{1} \ge -R \right\}, \quad Q_{I} := (Q\backslash Q_{R}) \cap \left\{ x_{2} \le 2(x_{1} + D) \right\}, \quad Q_{II} := Q\backslash (Q_{R}\cup Q_{I}). 
\end{split}
\end{equation*} Note that $Q_{I} \cup Q_{II}$ is precisely where $\rho_{Q}^{err}$ is supported. We show that in this \textit{error region}, the vertical speed is small. 
\begin{lemma}[Velocity in the error region]\label{lem:vel-Q-decomp}
We have \begin{equation}\label{eq:vel-Q-decomp}
	\begin{split}
		\sup_{ x \in  {\mathbb R_+^2 \backslash  Q_{R}} } v_{2}(t,x) \le  {\lambda^{1/6}}\quad\mbox{and} \quad \sup_{ x \in Q_{I}  } v_{2}(t,x)
		\leq  \frac{80 (Ax_2)^{1/2}+ {400}}{ x_2}	{\lambda^{1/2}}
		+\lambda^{3/4}. 
	\end{split}
\end{equation}
\end{lemma} 
\begin{figure}
\centering
\includegraphics[width=0.5\linewidth]{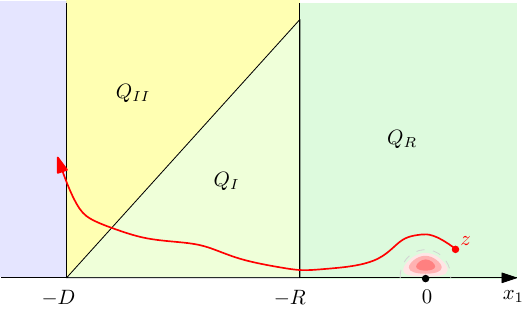}
\caption{Decomposition of $Q$ into $Q_{R}, Q_{I}$, and $Q_{II}$ (Figure not drawn to scale). A particle trajectory escaping $Q$ is shown, illustrating the ``push down'' effect of the Lamb dipole while the particle stays in $Q_{R}$.}
\label{fig:trajectory}
\end{figure}
\begin{remark} It is important that 
these estimates control $v_{2}$ from ABOVE  and \textit{not} $|v_{2}|$. (The estimates would be in general false for $|v_{2}|$.)  
\end{remark}
\begin{proof}  We decompose $\rho_{Q}^{main}=\rho_{Q}^{main,+}-\rho_{Q}^{main,-}$ with $\rho_{Q}^{main,+}=\max\{0,\rho_{Q}^{main}\}$. We analyze the contributions to $v_{2}$ from $\rho_{Q}^{main}$, $\rho_{Q}^{err}$, and $\tilde{\rho}$ separately. The key observation is that the main part $\rho_{Q}^{main,+}$ contributes negative velocity in the vertical direction. Precisely,
\begin{equation*}
	\begin{split}
		v_2(t, x) & = \frac{1}{2\pi}\int \frac{x_1-y_1}{|x-y|^2}\frac{4x_2y_2}{|x-\bar y|^2} \left(\rho_{Q}^{main,+}-\rho_{Q}^{main,-}+\rho_{Q}^{err}+\tilde \rho \right)(t,y) dy\\
		& =: v_{Q,2}^{main,+}+v_{Q,2}^{main,-}+v_{Q,2}^{err}+\tilde v_{2}.
	\end{split}
\end{equation*}

\medskip \noindent \textbf{Contribution from $\rho^{main}_{Q}$}.
For $x\in  \mathbb R_+^2 \backslash  Q_{R}$ and $y\in \mathrm{supp} (\rho_Q^{main})$, one has $x_1-y_1<0$, which gives $v_{Q,2}^{main,+} \le 0.$ Noticing $\nrm{\rho_{Q}^{main,-}}_{L^1\cap L^2( \mathbb R_+^2)}\leq \nrm{\omega^{(-)}}_{L^1\cap L^2( \mathbb R_+^2)} \leq 4\dlt$, we use \eqref{eq:vel-L-infty-alpha2} with $\alpha=1/2$  to get 
$$\nrm{v_{Q,2}^{main,-}}_{L^\infty( \mathbb R_+^2 \backslash  Q_{R})}\leq C\delta^{1/2}A^{1/4}\delta^{1/4}\leq\lambda^{3/4}.$$ Here we assumed $\dlt$ small if necessary, which is equivalent to taking $C_4$ small enough in \eqref{small_dlt}.

\medskip \noindent \textbf{Proof of the first inequality.} We now use $\nrm{\rho_{Q}^{err}+\tilde \rho}_{L^2}\leq 60 \lambda^{1/2}$ and
$\nrm{\rho_{Q}^{err}+\tilde \rho}_{L^1}\leq 120 \lambda^{1/2}$
from \eqref{eq:rem-small} and \eqref{eq:error-small}
to obtain
$$\nrm{v_{Q,2}^{err}+\tilde v_{2}}_{L^\infty(\mathbb R_+^2 \backslash  Q_{R})}\leq 
C\lambda^{3/8}A^{1/4}	
\leq \lambda^{1/4}$$ by using \eqref{eq:vel-L-infty-alpha2} with $\alpha=1/2$.
Here we assumed $\lambda$ small if necessary, which is equivalent to taking $C_3$ small enough in \eqref{small_lambda}.
Combining the above estimates yields the first inequality in \eqref{eq:vel-Q-decomp}.

\medskip \noindent \textbf{Proof of the second inequality.}  Now we prove an improved bound for $v_2$ in $Q_I$.  We decompose 
\begin{equation*}
	\begin{split}
		v_{Q,2}^{err}(t,x)+\tilde v_{2}(t,x) = \frac{1}{2\pi}\left(\int_{\{|x-y|<x_2/2\}}+\int_{\{|x-y|\geq x_2/2\}}\right) \frac{x_1-y_1}{|x-y|^2}\frac{4x_2y_2}{|x-\bar y|^2} \left(\rho_{Q}^{err}+\tilde \rho \right)(t,y) dy. 
	\end{split}
\end{equation*}
For the first integral,
the key observation is that for $x\in Q_I$, we have $\{y:\ |x-y|<x_2/2\}\subset Q$ so that $\tilde{\rho}$ vanishes there and,
using $|x-\bar y|\geq x_{2} + y_{2} \ge  x_2$, we have
\begin{equation*}
	\begin{split}
		\left| \frac{1}{2\pi} \int_{\{|x-y|<x_2/2\}}  \frac{x_1-y_1}{|x-y|^2}\frac{4x_2y_2}{|x-\bar y|^2}\rho_{Q}^{err}(t,y) dy\right|&\leq \frac{1}{x_2} \int_{\{|x-y|<x_2/2\}}  \frac{ y_2 |\rho_{Q}^{err}(t,y)|}{|x-y|} dy\\&
		\leq	\frac{6}{x_2}	 \sqrt{30\lambda}\sqrt{2x_2A}	
		\leq  \frac{80 (Ax_2\lambda)^{1/2}}{x_{2}},
	\end{split}
\end{equation*}
where we have used \eqref{eq:error-small}, \eqref{eq:B2}, and the inequality $$\int_{\mathbb R^2} \frac{1}{|x-y|} f(y) dy \leq 6 \nrm{f}_{L^1(\mathbb R^2)}^{1/2} \nrm{f}_{L^\infty(\mathbb R^2)}^{1/2}$$ with $f(y)=y_2|\rho_{Q}^{err}(t,y)|\mathbf{1}_{\{|x-y|<x_2/2\}}(y)$. \\

For the second integral, we use \begin{equation*}
	\begin{split}
		\frac{x_1-y_1}{|x-y|^2}\frac{4x_2y_2}{|x-\bar y|^2} \le \frac{16}{x_{2}}
	\end{split}
\end{equation*} in the region $\{|x-y|\geq x_2/2\}$ to obtain 
\begin{equation*}
	\begin{split}
		\left| \frac{1}{2\pi} \int_{\{|x-y|\geq x_2/2\}}  \frac{x_1-y_1}{|x-y|^2}\frac{4x_2y_2}{|x-\bar y|^2} 
		\left(\rho_{Q}^{err}+\tilde \rho \right)(t,y)
		dy\right|\leq \frac{8}{\pi x_2}  \nrm{\rho_{Q}^{err}+\tilde \rho}_{L^1}  
		\leq
		\frac{400
			\lambda^{1/2}}{ x_2}.		
	\end{split}
\end{equation*}
This proves the second inequality in \eqref{eq:vel-Q-decomp}.
\end{proof}

\subsubsection{Height control for local particles}
Here we control the height (namely, the $x_{2}$-coordinate of the trajectory) of fluid particles until they stay in $Q$, which visited the square $S_R$. 

\begin{lemma}\label{lem:S_R_to_H}
Let $z\in Q$. Suppose $\Phi(t_0,z)\in S_R$ for some $t_0\in[0,T]$. Then
as long as the particle stays in $Q$, its $x_2$-height is bounded above by $H$.  More precisely,
if 
$$\Phi(t_0,z)\in S_R$$ for some $t_0\in[0,T]$, then for any $\tau \le T$, we get
$$\sup_{t\in[t_{0},\tau]}\Phi_2(t,z)\leq H \qquad \mbox{as long as} \qquad \Phi_1(\tau,z) > -D.$$  
\end{lemma}
\begin{proof}
Let $z$ and $t_0$ satisfy the assumption.
The proof is divided into four steps.

\medskip \noindent \textbf{Claim 1}. We have $
\Phi_1(t,z)\leq R$  for any $t\in[t_0,T].$

\medskip \noindent 	This is immediate from the fact $v_1<-1/2$ on the line $\{x_1=R\}$.

\medskip \noindent \textbf{Claim 2}. As long as the trajectory stays in $Q_{R}$, its height is bounded by $4R$ and $H/6$. Precisely, $$
\Phi_2(t,z)\leq 4R<H/6\quad\mbox{for any}\quad t\in[t_0,T] \quad \mbox{satisfying }\quad \Phi(t,z)\in Q_R.
$$	

\medskip \noindent 	 The inequality $4R<H/6$ holds once we choose $C_2$ large in \eqref{large_H}.  For the first inequality, for a contradiction, we assume that there exists $t_1\in (t_0, T]$ such that $\Phi_2(t_1,z)> 4R$ and $\Phi_1(t_1,z)\in [-R, R]$. We take $t_2\in (t_0, t_1)$ to be largest time before $t_{1}$ such that $\Phi_2(t_2,z)=2R$. The velocity satisfies $v_1<-1/2$ on the line $\{x_1=-R\}$, so for all $t\in (t_2, t_1)$, we have $$ \Phi_1(t,z)\in [-R, R]\quad\mbox{ and }\quad \Phi_2(t,z)>2R.$$ On the other hand, by using \eqref{eq:borderline-v}, for any $y\in [-R, R]\times (2R, +\infty)$, we have \begin{equation}\label{eq:v-ests}
	\begin{split}
		v_1(y)\leq -1/2 \quad \text{ and }\quad |v_2(y)| \leq 1/4.
	\end{split}
\end{equation} 
Therefore, we get $t_1-t_2 \geq 4 (\Phi_2(t_1,z)-\Phi_2(t_2,z))\geq 8R.$ Hence $$-R\leq \Phi_1(t_1,z) \leq \Phi_1(t_2,z)-\frac12 \times 8R\leq -3R <-R,$$ which is a contradiction. This proves \textbf{Claim 2}.

\medskip \noindent \textbf{Claim 3}. As long as the trajectory stays in $Q_{I}$, its height is bounded by $H/4$. Precisely,
$$
\Phi_2(t,z)< H/4\quad\mbox{for any}\quad t\in[t_0,T] \quad \mbox{satisfying }\quad \Phi(t,z)\in Q_I.
$$		 
\medskip \noindent  To derive a contradiction, we  assume that there exists the smallest $t_3\in (t_0, T)$ such that $$\Phi(t_3,z)\in Q_I \quad\mbox{and}\quad \Phi_2(t_3,z)=H/4.$$
Recall that  $6R<H/4$.
Let $t_4\in (t_0, t_3)$ be the largest time before $t_{3}$ such that $\Phi_2(t_4,z)=5R$. Then we must have $$\frac{H}{4}\geq\Phi_2(t,z)\geq 5R \quad \mbox{ and }\quad \Phi(t,z)\in Q_I \quad \mbox { for all }\quad t\in [t_4,t_3],$$ since the velocity still satisfies \eqref{eq:v-ests} on $Q_I$ and in particular any particle will never come back to $Q_{I}$ once it leaves it. Notice that $0<\Phi_1(t_4,z)-\Phi_1(t_3,z)\leq D= C_0\lambda^{-1/2}$ recalling our choice of $D$ in \eqref{eq:D-def}. So we have $$0<t_3-t_4\leq 2 C_0\lambda^{-1/2}.$$ 
We may take $\lambda$ smaller such that  {$2C_0\lambda^{1/4}<R$}. By \eqref{eq:vel-Q-decomp}, for $t\in [t_4,t_3]$, we find 
$$\frac{d}{dt} \Phi_2(t,z)\leq  \frac{80 (A \Phi_2(t,z))^{1/2}+400}{ \Phi_2(t,z)}	{\lambda^{1/2}}
+\lambda^{3/4} 
$$
Note that $ \Phi_2(t,z)-  \lambda^{3/4}(t-t_4)\geq 5R-R=4R$ for any $t\in [t_4, t_3]$. We deduce from the above inequality that 
$$\frac{d}{dt} (\Phi_2(t,z)-\lambda^{3/4}(t-t_4)) \leq  \frac{80 (A H/4)^{1/2}+400}{\Phi_2(t,x)-\lambda^{3/4}(t-t_4)}\lambda^{1/2}.$$
This gives 
$$
\left(\Phi_2(t_3,z)-\lambda^{3/4}(t_3-t_4)\right)^2
\leq (5R)^2+4C_0 [80 (A H/4)^{1/2}+400],
$$	which implies
$$\Phi_2(t_3,z)\leq \left(
(5R)^2+4C_0 [80 (A H/4)^{1/2}+400]
\right)^{1/2}+R.
$$
Then, by taking a larger constant $C_2$ in \eqref{large_H} if necessary, we can guarantee that $H$ is large enough so that the right-hand side of the above is smaller than $H/8$. This contradicts $\Phi_2(t_3,z)=H/4$.

\medskip \noindent \textbf{Claim 4}. Lastly, we have 
$$
\Phi_2(t,z)< H/2\quad\mbox{for any}\quad t\in[t_0,T] \quad \mbox{satisfying }\quad \Phi(t,z)\in Q_{II}.
$$		

\medskip\noindent 
Consider any $t_5\in (t_0, T]$ such that $\Phi(t_5,z)\in Q_{II}$. Let $t_6\in (t_0, t_5)$ be the largest time before $t_{5}$ such that $ \Phi(t_6,z)\in \{x_2=2(x_1+D)\}$. That is, $\Phi_2(t_6,z)=2(\Phi_1(t_6,z)+D)$. Then, \textbf{Claims 2 and 3} give $\Phi_2(t_6,z)\leq H/4$ and hence $$0<\Phi_1(t_6,z)-\Phi_1(t_5,z)\leq\Phi_1(t_6,z)+D=\frac12 \Phi_2(t_6,z)\leq \frac{H}{8}.$$
This implies $t_5-t_6\leq H/4$. Then, using the first inequality of \eqref{eq:vel-Q-decomp}, we obtain \begin{equation*}
	\begin{split}
		\Phi_2(t_5,z)\leq \Phi_2(t_6,z)+ \lambda^{1/6}(t_5-t_6) < \frac{H}{2}, 
	\end{split}
\end{equation*} which establishes \textbf{Claim 4}, which also finishes the proof.
\end{proof}

As a consequence, we get a lower bound of the impulse $\mu_Q(t)$:
\begin{lemma} \label{lem:impulse-L1star-Q}  We have \begin{equation}\label{eq:impulse-L1star-Q}
	\begin{split}
		0 \le 
		|\mu|_{Q}(t)		
		- \mu_{Q}(t) \leq  \frac{\lambda}{4}\quad\mbox{and}\quad \mu_Q(t) \ge \mu[\omg_{Lamb}] - 12\lambda.
	\end{split}
\end{equation}
\end{lemma}
\begin{proof}
We first note 
\begin{equation*}
	\begin{split}
		|\mu|_{Q}(t)		
		- \mu_{Q}(t) = \int_{Q} 2x_{2} \rho_{Q}^{(-)}(t,x)dx\ge 0. 
	\end{split}
\end{equation*} Then, from \eqref{eq:L2-neg-small} and \eqref{eq:B2} together with Lemma \ref{lem:S_R_to_H}, we bound 
\begin{equation*}
	\begin{split}
		\int_{Q} 2x_{2} \rho_{Q}^{(-)}(t,x)dx 
		& \leq
		2\left(\int_{Q\cap \Phi(t, A^{loc,t})\}}x_2\rho^{(-)}(t,x) dx +\int_{Q\cap\Phi(t, A^{far,t})}x_2|\rho(t,x)| dx	\right)	\\
		& \leq 
		2(H \times (3\delta)+\frac{\lambda}{10})\leq \frac{1}{4}\lambda.
	\end{split}
\end{equation*}	Here we assume $\dlt$ small enough to get the last inequality.	 
Then, by using \eqref{eq:kpp-mu-Q-lower-bounds}, we obtain
\begin{equation*}
	\begin{split}
		\mu_Q(t) \ge \mu[\omg_{Lamb}] - 12\lambda . \qedhere 
	\end{split}
\end{equation*} 
\end{proof}

\subsection{Vertical speed and height control II}\label{subsec:height} 
\subsubsection{Vertical speed in far region}
\begin{lemma}\label{lem:v_2_on_S}
For any $t\in[0,T]$,
\begin{equation}\label{eq:v_2_on_S}
	\sup_{-D\leq x_1\leq -R,\, x_2\geq 4H}v_2(t,x)\leq \lambda^{1/2}.
\end{equation}
\end{lemma}

\begin{proof}
Consider any $x$ satisfying $-D\leq x_1\leq -R$ and $x_2\geq 4H$.
As in the proof of Lemma \ref{lem:vel-Q-decomp}, we decompose \begin{equation*}
	\begin{split}
		v_2(t, x) & = \frac{1}{2\pi}\int \frac{x_1-y_1}{|x-y|^2}\frac{4x_2y_2}{|x-\bar y|^2} \left(\rho_{Q}^{main,+}-\rho_{Q}^{main,-}+\rho_{Q}^{err}+\tilde \rho \right)(t,y) dy\\
		& =: v_{Q,2}^{main,+}+v_{Q,2}^{main,-}+v_{Q,2}^{err}+\tilde v_{2}.
	\end{split}
\end{equation*} Then as before,  we get $v_{Q,2}^{main,+} \le 0$ and  
$\nrm{v_{Q,2}^{main,-}}_{L^\infty}\leq \lambda^{3/4}.$\ \\

To estimate $v_{Q,2}^{err}$, we decompose
\begin{equation*}
	\begin{split}
		v_{Q,2}^{err}(t,x)  = \frac{1}{2\pi}\left(\int_{\{|x-y|<x_2/2\}}+\int_{\{|x-y|\geq x_2/2\}}\right) \frac{x_1-y_1}{|x-y|^2}\frac{4x_2y_2}{|x-\bar y|^2} \rho_{Q}^{err} (t,y) dy. 
	\end{split}
\end{equation*}
Then, by the exactly same way of the proof of Lemma \ref{lem:vel-Q-decomp}, we get
$$
|v_{Q,2}^{err}(t,x) |\leq \frac{80 (Ax_2)^{1/2}+400}{ x_2}	{\lambda^{1/2}}
\leq \left(\frac{40 A^{1/2} }{\sqrt{H}}+\frac{100}{H}\right)\cdot	{\lambda^{1/2}}\leq \frac{\lambda^{1/2}}{10},
$$ where we assume $H$ large enough in \eqref{large_H}.\\ 

It remains to control the contribution from $\tld{\rho}$.
First we decompose $\mathbb{R}^2_+\setminus Q$ into
$$
\mathbb{R}^2_+\setminus Q=U_{lower}\cup U_{upper}\cup U_{left},
$$ where
\begin{equation*}\begin{split}
		&U_{lower}:=\{y\in\mathbb{R}^2_+\,:\,-D- {\lambda^{-1/6}}\leq y_1\leq -D\quad\mbox{and}\quad 
		y_2\leq 3H
		\},\\ &U_{upper}:=\{y\in\mathbb{R}^2_+\,:\,-D-{\lambda^{-1/6}}\leq y_1\leq -D\quad\mbox{and}\quad 
		y_2> 3H
		\},\quad \mbox{and}\\ &   U_{left}:=\{y\in\mathbb{R}^2_+\,:\,y_1<-D-{\lambda^{-1/6}}\},
	\end{split}
\end{equation*} see Figure \ref{fig:far}.  
\begin{figure}
	\centering
	\includegraphics[width=0.5\linewidth]{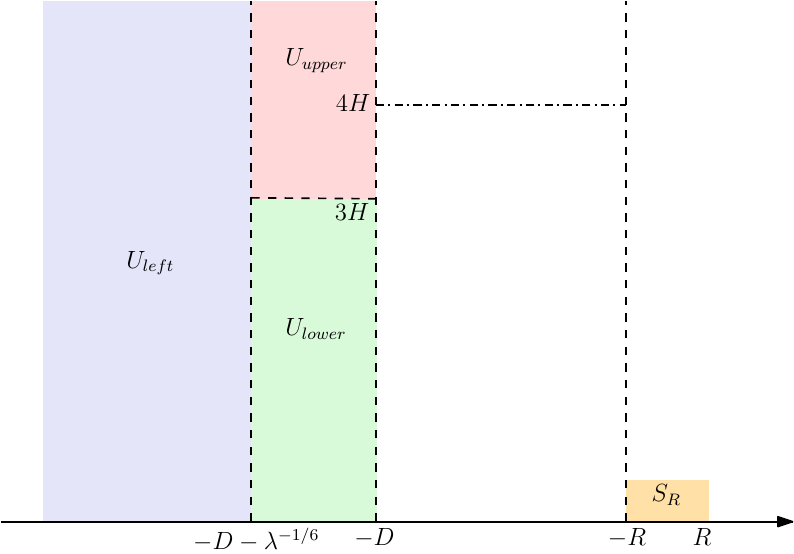}
	\caption{Decomposing $\mathbb{R}^2_+\setminus Q$ (Figure not drawn to scale.)}
	\label{fig:far}
\end{figure}
We decompose the contribution from $\tilde\rho$ by \begin{equation*}
	\begin{split}
		|\tld{v}_{2}(t,x)| \le \left[\int_{U_{lower}}  +\int_{U_{upper}}  +\int_{U_{left}} \right]  \frac{1}{|x-y|} |\tld{\rho}(t,y)| dy.
	\end{split}
\end{equation*}
For the  integral on $U_{lower}$, note that $|x-y|\geq |x_2-y_2|\geq H$ from 
$y_2\leq 3H$ and $x_2\geq 4H$: 
\begin{equation*}
	\begin{split}
		\int_{U_{lower}}\frac{1}{|x-y|} |\tld{\rho}(t,y)| dy\leq \frac{1}{H} \|\tilde\rho\|_{L^1}\leq
		\frac{60}{H}\lambda^{1/2}\leq\frac{1}{10}\lambda^{1/2}.
	\end{split}
\end{equation*}For the  integral on $U_{upper}$, we have the following

\medskip \noindent \textbf{Claim}. For any $y\in U_{upper}$,   we have  $z:=\Phi^{-1}_{t}(y) \notin S_R$. 

\medskip \noindent 
For a contradiction, suppose 
$z\in S_R$. 
Then, when the trajectory enters the strip $U_{upper}\cup U_{lower}=\{-D- {\lambda^{-1/6}}\leq y_1\leq -D\}$ at some $t'> 0$, we have
$\Phi_2(t',z)\leq H$ thanks to   Lemma \ref{lem:S_R_to_H}. Then using the first estimate \eqref{eq:vel-Q-decomp}, the $x_2$-height is smaller than $2H$ as long as the particle lives in the strip $U_{upper}\cup U_{lower}$ since the width of the strip is $\lambda^{-1/6}$ so that the staying time for that particle is at most $2\lambda^{-1/6}$ while the vertical speed is bounded above by $\lambda^{1/6}$. It is a contradiction to $y_2>3H$. 

\medskip 
The above \textbf{Claim} implies
$$ \||\tilde{\rho}|\cdot\mathbf{1}_{U_{upper}}\|_{L^1}\leq \int_{\mathbb{R}^2_+\setminus S_R}|\rho_0(z)|dz\leq \int_{\mathbb{R}^2_+\setminus B_1^+}|\omega_0(z)|dz\leq \delta.$$ Therefore, we finish the estimate for the integral on $U_{upper}$ by
\begin{equation*}
	\begin{split}
		\int_{U_{upper}}\frac{1}{|x-y|} |\tld{\rho}(t,y)| dy\leq 
		6\| |\tilde{\rho}|\cdot\mathbf{1}_{U_{upper}} \|^{1/2}_{L^1}\| |\tilde{\rho}|\cdot\mathbf{1}_{U_{upper}} \|^{1/2}_{L^\infty}
		\leq 
		6\sqrt{A\delta}
		\leq\frac{\lambda^{1/2}}{10}.
	\end{split}
\end{equation*}

For the  integral on $U_{left}$,
we note $|x-y|\geq |x_1-y_1|\geq \lambda^{-1/6}$ 
we compute
\begin{equation*}
	\begin{split}
		\int_{U_{lower}}\frac{1}{|x-y|} |\tld{\rho}(t,y)| dy\leq \lambda^{1/6} \|\tilde\rho\|_{L^1}\leq
		60 \lambda^{(1/2)+(1/6)}\leq\frac{\lambda^{1/2}}{10}.
	\end{split}
\end{equation*} This finishes the proof of \eqref{eq:v_2_on_S}. \end{proof}



\subsubsection{Height control for far particles} We now present the key lemma regarding far field particles.

\begin{lemma}\label{lem:far_cont}
Take any $z\in Q $ and $T_{0} \in[0,T]$. Then,  
\begin{equation*}
	\Phi(T_{0},z)\in Q\quad \mbox{implies} \quad \Phi_2(T_{0},z)\leq 2|z|+6H.
\end{equation*} 
That is,  as long as the particle stays in $Q$, its $x_{2}$-coordinate is bounded by $2|z|+6H.$
\end{lemma}

\begin{proof} For $z\in Q$, suppose  $\Phi(T_0,z)\in Q$ for some $T_0\in (0,T]$ (the conclusion of the lemma is trivial for $T_0=0$). Then we know
$$\Phi(t,z)\in Q\quad \mbox{for all}\quad t\in[0,T_0]$$ because $v_1\leq -1/2$ on the vertical line $\{x_1=-D\}$. We consider two cases: 

\medskip \noindent \textbf{Case I:} there exists $T_1\in [0,T_0]$ satisfying $
\Phi_1(t,z)\geq -R$ {for all}  $t\in[0,T_1].$ We then claim that 
\begin{equation}\label{eq:far_control_1}
	\Phi_2(t,z)\leq 2|z|+R +H \quad\mbox{for all}\quad t\in[0,T_1].
\end{equation} Indeed, consider
any $t\in[0,T_1]$. 
If there exists $s\in[0,t]$ satisfying 
$\Phi(s,z)\in S_R$, then Lemma \ref{lem:S_R_to_H} implies 
$\Phi_2(t,z)\leq H$ which gives \eqref{eq:far_control_1}. 
So we may assume  $\Phi(s,z)\notin S_R$ for every $s\in[0,t]$.
Then, from the velocity estimate \eqref{eq:vel_out_S_R} outside $S_R$, we know for $s\in[0,t]$,
$$
\frac{d}{ds}\Phi_2(s,z)\leq \frac{1}{2}\quad\mbox{and}\quad
\frac{d}{ds}\Phi_1(s,z) \leq -\frac{1}{2}.
$$
Thus $-R\leq \Phi_1(t,z)\leq z_1-0.5 t$ and
$\Phi_2(t,z)\leq z_2+0.5t$, which imply
\eqref{eq:far_control_1}. Note that we can take $T_1$ to be the maximal (within $[0,T_0]$) with the stated property. If the maximal $T_1$ is not equal to $T_0$, we move on to \textbf{Case II} below.

\medskip \noindent \textbf{Case II}: Suppose
$\Phi_1(T_2,z)\leq -R$  for some $T_2\in[0,T_0]$.
Note that for any $t\in[T_2,T]$, we know $\Phi_1(t,z)\leq -R$ 
due to \eqref{eq:vel_out_S_R}.
Thus
$$-D\leq \Phi_1(t,z)\leq -R\quad\mbox{for all}\quad t\in[T_2,T_0].$$ 
Recall $T_0-T_2\leq 2D=2C_0 \lambda^{-1/2}.$ Thus, by Lemma \ref{lem:v_2_on_S}, we get
$$
\Phi_2(t,z)\leq \Phi_2(T_2,z)+4H +2C_0
\leq 2|z|+R+5H +2C_0\leq 2|z|+6H
\quad\mbox{for all}\quad t\in[T_2,T_0],
$$ by taking $H$ large if necessary. This finishes the proof. \end{proof} 

\section{Closing bootstrap hypotheses}\label{sec:closing}

In this section, we close all the bootstrap hypotheses. First, on the time interval $[0,T]$ where \eqref{eq:B1}--\eqref{eq:B5} hold, we prove that the inequality in \eqref{eq:B2} holds with $\lambda/20$ instead of $\lambda/10$.
Then we prove that the inequalities \eqref{eq:B3}--\eqref{eq:B5} hold with $\lambda/2$ instead of $\lambda$.  Then, we show that \eqref{eq:B1} holds for a slightly larger time interval $[0,T^*]$, $T^*>T$. Then, by continuity of the involved quantities in time, we immediately conclude \eqref{eq:B2}--\eqref{eq:B5} in a time interval slightly larger than $[0,T]$. Then, we deduce that \eqref{eq:B1}--\eqref{eq:B5} hold for $[0,\infty)$, which in particular implies the conclusion of the main theorem.

This section is organized as follows. To begin with,
in \S \ref{subsec:2}, we close \eqref{eq:B2}.
Next, in \S \ref{subsec:34}, we close \eqref{eq:B3} and \eqref{eq:B4}. Next, in \S \ref{subsec:5}, we close \eqref{eq:B5}. Lastly, \eqref{eq:B1} is closed in  \S \ref{subsec:1}. We formally complete the proof of the main theorem in \S \ref{subsec:proof-main}. 

\subsection{Closing (B2): Far field control}\label{subsec:2}

From Lemma \ref{lem:far_cont}, we compute
\begin{equation}\label{closing-B2}
\begin{split}
	&	\int_{Q}x_2|\rho(t,x)|\mathbf{1}_{\Phi(t,A^{far,t})}(x) dx=\int_{A^{far,t}\cap\left( \Phi^{-1}_{t}(Q)\right)} \Phi_2(t,z)|\rho_0(z)|\,dz\\
	&\quad\quad\quad  \leq \int_{\mathbb{R}^2_+\setminus S_R} (2|z|+6H)|\rho_0(z)|\,dx\leq (2+6H)		\int_{|z|>10}|z||\rho_0(z)|\,dz		\leq 2\times  6H\delta\leq \frac{\lambda}{20},
\end{split}
\end{equation} where we take $\delta$ small enough.
Here we used, from $|p_1(0)|\leq 1$,
\begin{equation}\label{mom_rho}
\begin{split}
	\int_{|z|>10}|z||\rho_0(z)|\,dz	
	\leq \int_{|x|>5}(|x|+1)|\omega_0(x)|\,dx	\leq 2\delta.
\end{split}
\end{equation}	
\subsection{Closing (B3) and (B4): Estimating enstrophy and impulse }\label{subsec:34}

\subsubsection{Enstrophy estimate} 
We compute \begin{equation*}
\begin{split}
	\frac{d}{dt} (\kpp_{Q}(t))^2 = - \int_{Q} v\cdot\nb (\rho^2) dx = \int_{ \{ x_{1} = -D \} }  v_{1} \rho^2 dx_{2}. 
\end{split}
\end{equation*} From the global velocity estimate \eqref{eq:global-v}, $v_{1} \le -1/2$ on $\{ x_{1} = -D \}$. Hence \begin{equation*}
\begin{split}
	\frac{d}{dt} (\kpp_{Q}(t))^2 \le 0, \quad \mbox{which gives}\quad \kpp_{Q}(t) \le \kpp_{Q}(0). 
\end{split}
\end{equation*} In particular, for all $0\le t \le T$,  we have  \begin{equation}\label{closing-B3}
\begin{split}
	\kpp_{Q}(T) \le \kpp_{Q}(t) < \kpp[\omg_{Lamb}] + \dlt <  \kpp[\omg_{Lamb}] + \frac{1}{2}\lambda. 
\end{split}
\end{equation}  
\subsubsection{Impulse estimate}
We compute \begin{equation*}
\begin{split}
	\frac{d}{dt} \mu_Q(t) & = - \int_{Q} x_{2} v\cdot\nb\rho dx  = \int_{Q} \tld{v}_{2} \rho dx + \int_{ \{ x_{1} = - D \} } ( v_{Q,1} + \tld{v}_{1} - \dot{p}_{1} ) x_{2} ( \rho^{(+)} - \rho^{(-)} ) dx_{2} \\
	& =: F_{\mu}(t) - I_{\mu}^{(+)}(t) + I_{\mu}^{(-)}(t). 
\end{split}
\end{equation*} Here, $I_{\mu}^{(\pm)}(t)$ denotes the ``flux'' of impulse through the borderline: 
\begin{equation}\label{defn_I_mu}
\begin{split}
	I_{\mu}^{(\pm)}(t) := \int_{ \{ x_{1} = - D \} } (  \dot{p}_{1}  -( v_{Q,1} + \tld{v}_{1}) ) x_{2} \rho^{(\pm)} dx_{2}. 
\end{split}
\end{equation}
As a consequence of the velocity estimate \eqref{eq:borderline-v} in Lemma \ref{lem:borderline-v}, we have $I_{\mu}^{(\pm)}(t) \ge 0$. 
We now integrate the above in time to get \begin{equation}\label{eq:mu}
\begin{split}
	\mu_{Q}(T) - \mu_{Q}(0) + \int_{0}^{T} I_{\mu}^{(+)}(t) dt  = \int_{0}^{T} F_{\mu}(t) dt  + \int_{0}^{T} I_{\mu}^{(-)}(t) dt . 
\end{split}
\end{equation}
\begin{lemma} \label{lem:I-mu-minus} We have 
\begin{equation}\label{eq:I-mu-minus}
	\begin{split}
		\int_{0}^{T} I_{\mu}^{(-)}(t) dt 
		\leq  \frac{\lambda}{10}.			
	\end{split}
\end{equation} 
\end{lemma}
\begin{proof} 
From the definition of $I_{\mu}^{(-)}$ and the velocity estimate \eqref{eq:borderline-v}, we get \begin{equation*}
	\begin{split}
		I_{\mu}^{(-)}(t) &\le 2 \int_0^\infty
		x_{2} 
		\rho^{(-)}(t,-D,x_{2}) dx_{2}
		=2\lim_n\left( n \int_{-D-(1/n)}^{-D} \int_{0}^\infty
		x_{2} 
		\rho^{(-)}(t,x_1,x_{2}) dx_{2}dx_{1}\right).\\
	\end{split}
\end{equation*}
Let $n\geq 1.$ We take time integral to get
\begin{equation}\label{eq:n_time}
	\begin{split}
		&\int_0^T\left( n \int_{-D-(1/n)}^{-D} \int_{0}^\infty
		x_{2} 
		\rho^{(-)}(t,x_1,x_{2}) dx_{2}dx_{1}\right)dt\\
		&=
		n	\int_0^T\left( \int_{\bbR^2_+} \Phi_2(t,z)\rho_0^{(-)}(z)\cdot \mathbf{1}_{\{y\,:\,\Phi_1(t,y)\in[-D-(1/n),-D]\}}(z)
		dz\right)dt.\\
	\end{split}
\end{equation} 
Due to the horizontal  velocity estimate \eqref{eq:vel-Q-decomp}, 
we observe for any $z\in\bbR^2_+$,
$$
|\{t\in[0,T]\,:\,\Phi_1(t,z)\in[-D-(1/n), -D]\}|\leq \frac{2}{n}.
$$
We claim that for any $z\in\bbR^2_+$, 
\begin{equation}\label{eq:claim:im}
	\Phi_1(t,z)\in[-D-1,-D]\quad\Rightarrow \quad
	\Phi_2(t,z)\leq 2|z|+6H+1.
\end{equation} 

Indeed, we first consider the case when
$z\in Q$. Consider $s\in[0,t]$ satisfying $\Phi_1(s,z)=-D$. Then, Lemma \ref{lem:far_cont} implies
$
\Phi_2(s,z)\leq 2|z|+6H.
$
On the other hand, due to the horizontal velocity estimate \eqref{eq:borderline-v}, we know $|t-s|\leq 2$ so that the vertical velocity estimate \eqref{eq:vel-Q-decomp} implies $\Phi_2(t,z)\leq \Phi_2(s,z)+2\lambda^{1/6}\leq 2|z|+6H+1$.
Now consider any $z\notin Q$. Then the vertical velocity estimate \eqref{eq:vel-Q-decomp} gives
$\Phi_2(t,z)\leq z_2+t\lambda^{1/6}
\leq z_2+2\lambda^{1/6}
\leq 2|z|+6H+1$. It finishes the proof of the above claim \eqref{eq:claim:im}. 

Then, by using Fubini, we continue the computation in \eqref{eq:n_time}:
\begin{equation*}
	\begin{split}
		\dots
		&\leq   n \int_{\bbR^2_+} (2|z|+6H+1)\rho_0^{(-)}(z)\cdot\left(\frac{2}{n}\right)
		dz
		\\
		&=2\int_{\bbR^2_+} (2|z|+6H+1)\rho_0^{(-)}(z)
		dz\leq 4\int_{\bbR^2_+} |z|\rho_0^{(-)}(z)
		dz+100H\int_{Q}\rho_0^{(-)}dz\\
		&\leq 100\delta+100H\|\omg_0^{(-)}\|_{L^1}
		\leq 500H\dlt
		\leq  \frac{\lambda}{10},
	\end{split}
\end{equation*} where we assumed $\dlt$ small enough. This estimate holds uniformly in $n\geq1$ so that  it finishes the proof. \end{proof}

To this end, we define the set of ``gained particles'' (from the perspective of $\bbR^2_+\backslash Q = \{ x_{1} < -D \}$) by the time $T$: \begin{equation*}
\begin{split}
	A^{gain, T} := \left\{ z\in \bbR^2_+
	\, : \, \mbox{there exists } t \in [0,T] \, \mbox{ such that } \, \Phi_{1}(t,z) < -D \right\}. 
\end{split}
\end{equation*} 
Thanks to the velocity estimate \eqref{eq:borderline-v}, for any $z \in A^{gain, T}$, we define $\tau(z)=0$ when $z_1<-D$. Otherwise, we can \textit{uniquely} define $\tau(z) \in [0,T]$ to be the time such that $\Phi_{1}(\tau(z),z) = -D$.
We then set \begin{equation*}
\begin{split}
	\rho_{0}^{gain, T} := \rho_{0} \mathbf{1}_{A^{gain, T}}.
\end{split}
\end{equation*}
Then we know from \eqref{eq:rem-small} that \begin{equation}\label{eq:gain-small}
\begin{split}
	&	\nrm{ \rho_{0}^{gain, T} }_{L^{2}} \le \sup_{ t \in [0,T] } \nrm{ \tld{\rho}(t,\cdot) }_{L^2} \leq 
	\nrm{ \tld{\rho}(T,\cdot) }_{L^2}		
	\le 30\lambda^{1/2}, \quad \nrm{ \rho_{0}^{gain, T} }_{L^{1}}
	\leq 
	\nrm{ \tld{\rho}(T,\cdot) }_{L^1}\leq 60\lambda^{1/2}.
\end{split}
\end{equation}  

We now estimate time integral of the non-local interaction term $F_{\mu}$:
\begin{lemma}\label{lem:F-mu}
We have  \begin{equation}\label{eq:I-mu-plus}
	\begin{split}
		\int_0^T F_{\mu}(t) dt < \frac{\lambda}{10}\quad\mbox{and} \quad	\int_0^T I_\mu^{(+)}(t) dt < 13\lambda.
	\end{split}
\end{equation}
\end{lemma}
\begin{proof}

We now rewrite and decompose \begin{equation*}
	\begin{split}
		F_{\mu}(t) = \int_{Q} \tld{v}_{2}\rho dx = \int \tld{\rho} ( - v^{main}_{Q} - v^{err}_{Q}  )_{2} dx =: F^{main}_{\mu}(t) + F^{err}_{\mu}(t). 
	\end{split}
\end{equation*} 
In what follows, we are going to prove \begin{equation*}
	\begin{split}
		\int_0^T F^{main}_{\mu}(t) dt < \frac{\lambda}{20} \qquad \mbox{and}\qquad \int_0^T F^{err}_{\mu}(t) dt < \frac{\lambda}{20},
	\end{split}
\end{equation*} which gives the lemma.

\medskip

\noindent \textbf{Estimate of $F_\mu^{err}$}. We rewrite the integral by changing variables from $x$ to $z$ related by $x = \Phi(t,z)$ and then using Fubini's theorem to exchange $z$ and $t$ integrals:
\begin{equation*}
	\begin{split}
		\int_0^{T} F_{\mu}^{err}(t) dt & = \int_0^{T} \int_{ \bbR^2_+\backslash Q } \tld{\rho}(t,x) (- 	v^{err}_{Q,2})(t,x)  dx \,dt \\&= \int_0^{T} \int_{\bbR^2_+} {\rho}^{gain, T}_{0}(z) \mathbf{1}_{\{\tau(z)<t\}}(z)	(-v^{err}_{Q,2})(t,\Phi(t,z))  dz\,dt\\
		& = \int_{\bbR^2_+} \rho_{0}^{gain, T}(z) \, \left[ \int_{ \tau(z)}^{T} - v^{err}_{Q,2}(t,\Phi(t,z))  dt \right] dz , 
	\end{split}
\end{equation*} 
where	\begin{equation}\label{eq:v-err-2}
	\begin{split}
		-v^{err}_{Q,2}(t,\Phi(t,z))   =  -\frac1{2\pi} \int_{ \{ -D < y_{1} \le -R \} } \frac{\Phi_1(t,z)-y_1}{|\Phi(t,z)-y|^{2}} \frac{4\Phi_2(t,z)}{|\Phi(t,z)-\bar{y}|^2} y_{2}\rho^{err}_{Q}(t,y) dy. 
	\end{split}
\end{equation}
We are going to estimate $v^{err}_{Q,2}$ in two different ways: \begin{itemize}
	\item First, from Lemma \ref{lem:error-small}, we have uniform in time-space bound \begin{equation*}
		\begin{split}
			\nrm{v^{err}_{Q,2}}_{L^\infty([0,T]\times\bbR^2_+)} \le  CA^{1/4}\lambda^{{3/8}}. 
		\end{split}
	\end{equation*}
	\item Next, if $\Phi_{1}(t,z) < -D$ i.e. when $t > \tau(z)$, we obtain from \eqref{eq:v-err-2} that \begin{equation*}
		\begin{split}
			| v^{err}_{Q,2}(t,\Phi(t,z)) | & \le \int_{ \{ -D < y_{1} \le -R \} } \frac{Cy_{2}|\rho^{err}_{Q}(t,y)| }{|\Phi(t,z)-y||\Phi(t,z)-\bar{y}|} dy \\
			& \le C|\Phi_{1}(t,z) + D|^{-2} \nrm{\rho^{err}_{Q}(t,\cdot)}_{L^1_*} \le C\lambda |\Phi_{1}(t,z) + D|^{-2}. 
		\end{split}
	\end{equation*} Moreover, we infer from \eqref{eq:borderline-v} 
	that $\dot{\Phi}_1(t,z)<-\frac12$ when $t > \tau(z)$. This gives $\Phi_1(t,z)<-\frac{t-\tau(z)}{2}-D$ for any $t > \tau(z)$. 
\end{itemize}
Therefore, from these two bounds 		we conclude that 
\begin{equation*}
	\begin{split}
		\int_0^{T} F_{\mu}^{err}(t) dt &\le C\int \left|\rho_{0}^{gain, T}(z)\right| \left[\int_{\tau(z)}^{T} \min\left\{  A^{1/4}\lambda^{3/8} , \lambda |\Phi_{1}(t,z) + D|^{-2} \right\} dt \right] dz \\
		&\le CA^{1/4} \lambda^{11/16} \int \left|\rho_{0}^{gain, T}(z)\right| dz \le CA^{1/4} \lambda^{11/16} 
		\lambda^{1/2}
		\leq CA^{1/4}  \lambda^{19/16}\leq
		\frac{\lambda}{20},\\			 
	\end{split}
\end{equation*} by taking $\lambda$ smaller if necessary
where we have used \eqref{eq:gain-small}.

\medskip

\noindent \textbf{Estimate of $F_\mu^{main}$}.
Integrating in time and applying Fubini's theorem as in the above, 
\begin{equation*}
	\begin{split}
		\int_0^{T} F_{\mu}^{main}(t) dt  
		& = \int \rho_{0}^{gain, T}(z) \, \left[ \int_{ \tau(z)}^{T} -v^{main}_{Q,2}(t,\Phi(t,z))  dt \right] dz
	\end{split}
\end{equation*} with 	\begin{equation*}
	\begin{split}
		-v^{main}_{Q,2}(t,\Phi(t,z))   =  -\frac1{2\pi} \int_{ \{ -R \le y_{1} \} } \frac{\Phi_1(t,z)-y_1}{|\Phi(t,z)-y|^{2}} \frac{4\Phi_2(t,z)}{|\Phi(t,z)-\bar{y}|^2} y_{2}\rho^{main}_{Q}(t,y) dy.
	\end{split}
\end{equation*} Using that for $t \ge \tau(z)$, we have $|\Phi_{1}(t,x) - y_1| \ge D/2 + (t-\tau(z))/2,$ we compute
\begin{equation*}
	\begin{split}
		\int_0^{T} F_{\mu}^{main}(t) dt  
		& \leq  \int \left|\rho_{0}^{gain, T}(z)\right| \, \left[ \int_{ \tau(z)}^{T} \frac{32}{\pi} \nrm{\rho_Q^{main}}_{L_*^1} \frac{1}{(D+(t-\tau(z)))^2}  dt \right] dx \\
		&\leq 
		{	(60\lambda^{1/2})}
		\times 64 D^{-1}\leq \frac{5000}{C_0}  \lambda \leq \frac{\lambda}{20},
		\\
	\end{split}
\end{equation*}
where we used $$ \nrm{\rho_Q^{main}}_{L_*^1}\leq \|\omega_{Lamb}\|_{L^1_*}+\lambda\leq \pi+\lambda\leq 4.$$ and the choice of $C_0$ in \eqref{defn_C_0}.
This finishes the proof.  
\end{proof}

Returning to \eqref{eq:mu} and applying Lemmas \ref{lem:I-mu-minus} and \ref{lem:F-mu}, we obtain \begin{equation*}
\begin{split}
	\mu_Q(T) + \int_0^T I_\mu^{(+)}(t) dt \leq \mu_Q(0) + \frac15 \lambda \leq \mu[\omg_{Lamb}] +  \frac15\lambda+\delta \leq \mu[\omg_{Lamb}] +  \frac14\lambda . 
\end{split}
\end{equation*}
Since $I_\mu^{(+)}(t)\ge0$, we conclude  by \eqref{eq:impulse-L1star-Q} that \begin{equation}\label{eq:B4-close}
\begin{split}
	\mu_Q(T) <   \mu[\omg_{Lamb}] +  \frac14\lambda,\quad |\mu|_Q(T) <   \mu[\omg_{Lamb}] +  \frac12\lambda, \quad \mbox{and} \quad \int_0^T I_\mu^{(+)}(t) dt <  13\lambda.
\end{split}
\end{equation}  

\subsection{Closing (B5): Estimating energy}\label{subsec:5}

We decompose the energy \begin{equation*}
\begin{split}
	E[\omg_{0}] = E[\omg(T,\cdot)] = E_{Q}(T) + \tld{E}(T) + 2E_{inter}[ \rho_{Q}(T,\cdot), \tld{\rho}(T,\cdot) ]
\end{split}
\end{equation*} where $E_{inter}$ is defined in \eqref{eq:E-int-def}, $E_{Q}(T) = E[\rho_{Q}(T,\cdot)]$, and $\tld{E}(T,\cdot) = E[\tld{\rho}(T,\cdot)]$. To obtain a lower bound for $E_{Q}$, in the following two lemmas, we prove that each of   $\tld{E}(T)$ and $ 2E_{inter}[ \rho_{Q}(T,\cdot), \tld{\rho}(T,\cdot) ]$  are smaller than 
$\lambda/8$	
for all times. Then, using \eqref{eq:ini_en}, we get \begin{equation}\label{closing-B5}
\begin{split}
	E_{Q}(T) \ge E[\omg_{0}] - \frac{\lambda}{4} > E[\omg_{Lamb}] - \frac{\lambda}{2}, 
\end{split}
\end{equation} which closes \eqref{eq:B5}.

\subsubsection{Estimating the interaction energy} \begin{lemma}
For all $t \in [0,T]$, we have $2E_{inter}[ \rho_{Q}, \tld{\rho} ](t) \leq \frac{1}{8} \lambda$. 
\end{lemma} \begin{proof}
Let $t\in[0,T]$. We decompose \begin{equation*}
	\begin{split}
		E_{inter}[ \rho_{Q}, \tld{\rho} ](t) = E_{inter}[ \rho_{Q}^{main}, \tld{\rho} ](t)  + E_{inter}[ \rho_{Q}^{err}, \tld{\rho} ](t) 
	\end{split}
\end{equation*} and for the last term, we simply use \eqref{eq:E-inter-unsymm} from Lemma \ref{prop:Energy-X}, Lemma \ref{lem:rem-small}, and Lemma \ref{lem:error-small} to bound \begin{equation*}
	\begin{split}
		\left| E_{inter}[ \rho_{Q}^{err}, \tld{\rho} ](t) \right| \le \frac12 \nrm{ x_{2}^{-1}\tld{\psi} }_{L^{\infty}}  \nrm{  \rho_{Q}^{err} }_{L^{1}_{*}} \le \frac12 \nrm{ \tld{v} }_{L^{\infty}}  \nrm{  \rho_{Q}^{err} }_{L^{1}_{*}} \leq  \frac{1}{2}{C}A^{1/2}  \lambda^{1/4}\cdot 30\lambda\leq \frac{1}{16}\lambda.
	\end{split}
\end{equation*} 
For the interaction term between $\rho_{Q}^{main}$ and $\tld{\rho}$,
we first claim that for any  $x,y\in \ohp$ satisfying $x_1<-D<-R<y_1,$  we have the following $x_2$-estimate:
\begin{equation}\label{eq:claim_en}
	x_2\leq 2|z|+ 4\lambda^{1/6}|x-y|,
\end{equation}
where $z=\Phi^{-1}_t(x)$. To prove the claim, 
we first recall
$Q= A^{loc, t}\cup  A^{far, t}$ (see \eqref{defn_A_loc_far}). Thus there are three possible cases for $z\in\ohp$:
$$z\in A^{loc,t},\quad z\in A^{far,t},\quad\mbox{or}\quad z\in\ohp\setminus  Q.$$  	
First consider the case when
$z\in A^{loc,t}$.  	
Then from $x_1<-D$ and from the estimate 
\eqref{eq:borderline-v} for $v_1$ and the first estimate \eqref{eq:vel-Q-decomp} for $v_2$ together with Lemma \ref{lem:S_R_to_H}, we have 
$x_2\leq H+2\lambda^{1/6}(-D-x_1)$ and $|x-y|\geq |x_1-y_1|\geq \frac{D}{2}+(-D-x_1)$. Thus we get
$$
x_2\leq H+2\lambda^{1/6}(|x-y|-\frac{D}{2})=
(H-2\lambda^{1/6}\frac{D}{2})+2\lambda^{1/6}|x-y|\leq 2\lambda^{1/6}|x-y|\leq 2|z|+ 4\lambda^{1/6}|x-y|.
$$ 
where we used $D=C_0/\sqrt\lambda$ and assumed $\lambda$ small enough.
Next,  we  consider $z\in A^{far,t}$.
Then,  Lemma \ref{lem:far_cont} implies 
$x_2\leq (2|z|+6H)+2\lambda^{1/6}(-D-x_1)$, 
which gives
$$
x_2\leq (2|z|+6H)+2\lambda^{1/6}(|x-y|-\frac{D}{2})\leq 2|z|+4 \lambda^{1/6}|x-y|.
$$ 	Lastly, we  consider $z\in \ohp\setminus Q$.
From $x_1\leq z_1\leq -D$, we know
$x_2\leq z_2+2\lambda^{1/6}t$ and $|x-y|\geq |x_1-z_1|	
\geq \frac{t}{2}$, which implies
$$
x_2\leq z_2+4\lambda^{1/6}|x-y|\leq 2|z|+4\lambda^{1/6}|x-y|.
$$ We proved the above claim \eqref{eq:claim_en}. 

Then, by the claim \eqref{eq:claim_en} and by $|x-y|\geq D/2$,  we compute	
\begin{equation*}
	\begin{split}
		\int_{ \bbR^2_+} \frac{x_2}{|x-y|^{2}}| \tld{\rho}(t,x)| \, dx&\leq 
		\left(\frac{2}{D}\right)^2\int_{ \bbR^2_+} 2|{\Phi^{-1}_{t}(x)}|\cdot | \tld{\rho}(t,x)| \, dx
		+4\lambda^{1/6}\cdot \frac{2}{D}				  \int_{ \bbR^2_+}  | \tld{\rho}(t,x)| \, dx
		\\
		&\leq  8\frac{\lambda}{C_0^2}\int_{ \bbR^2_+} |{z}|\cdot| {\rho}(0,z)| \, dz
		+8\lambda^{1/6}\cdot \frac{\lambda^{1/2}}{C_0}				  \int_{ \bbR^2_+}  | \tld{\rho}(T,x)| \, dx\\
        &\leq  8\frac{\lambda}{C_0^2}\int_{ \bbR^2_+} (|x|+1)\cdot| {\omg_0}(x)| \, dx
		+8\lambda^{1/6}\cdot \frac{\lambda^{1/2}}{C_0}\cdot 60{\lambda}^{1/2}\\
        &\leq  800\frac{\lambda}{C_0^2}
		+800\lambda^{1/6}\cdot \frac{\lambda}{C_0}\leq \frac{\lambda}{100}. 
	\end{split}
\end{equation*}

Therefore, we  can estimate \begin{equation}\label{en_inter_1}
	\begin{split}
		\left|E_{inter}[ \rho_{Q}^{main}, \tld{\rho} ](t)\right|  & = \left|\frac1{8\pi} \iint_{\bbR^2_+ \times \bbR^2_+} \rho_{Q}^{main}(t,y) \log\left( 1 + \frac{4x_2y_2}{|x-y|^{2}} \right) \tld{\rho}(t,x) \, dxdy \right| \\
		& \le \frac1{2\pi} \iint_{\bbR^2_+ \times \bbR^2_+} \frac{x_2y_2}{|x-y|^{2}} |\rho_{Q}^{main}(t,y) \tld{\rho}(t,x)| \, dxdy\\
		& \le \frac1{2\pi} \int_{\bbR^2_+} y_2 |\rho_{Q}^{main}(t,y)|\int_{ \bbR^2_+} \frac{x_2}{|x-y|^{2}}| \tld{\rho}(t,x)| \, dxdy\\
		& \le \frac{\lambda}{100}\frac1{2\pi}\|\rho_Q\|_{L^1_*} \le \frac{\lambda}{100}. \qedhere 
	\end{split}
\end{equation} 
\end{proof}

\subsubsection{Estimating perturbation energy}

\begin{lemma}\label{lem:pert_en}
We have $\int_{0}^{T}\left| \frac{d}{dt} \tld{E}(t) \right|dt\leq \frac{1}{30} \lambda$. 
\end{lemma} 
Once we prove the above lemma, it closes \eqref{eq:B5} since we have, from \eqref{eq: SEI} and \eqref{mom_rho},
\begin{equation*}
\begin{split}
	|\tld{E}(0)|\leq C_L\|\tilde\rho(0)\|_{L^2}\|\tilde\rho(0)\|_{L^1_*}\leq C_L\cdot \delta\cdot(2\delta)\leq \frac{\lambda}{100}
\end{split}
\end{equation*} which implies $$	\tld{E}(T) \le \tld{E}(0) +\int_{0}^{T}\left| \frac{d}{dt} \tld{E}(t) \right|dt\leq \frac{\lambda}{8}.$$
\begin{proof}[Proof of Lemma \ref{lem:pert_en}]
We compute 
\begin{equation*}
	\begin{split}
		\frac{d}{dt} \tld{E}(t)  & = \int_{ \bbR^{2}_{+} \backslash Q }(\partial_t\tilde\rho)\tilde\psi\,dx=- \int_{ \bbR^{2}_{+} \backslash Q } \nb\cdot( v \tilde\rho )  \tilde\psi dx \\
		& = \int_{ \bbR^{2}_{+} \backslash Q } ( \nb\tld{\psi} \cdot v ) \tld{\rho} \, dx - \int_{ \{ x_{1} = - D \} } ( v_{Q,1} + \tld{v}_{1} - \dot{p}_{1} ) \tld{\psi} \rho \,  dx_{2}.  
	\end{split}
\end{equation*}

Using \eqref{eq:rem-small} to estimate $\nrm{x_{2}^{-1}\tld\psi}_{L^\infty} \le \nrm{\tld{v}}_{L^\infty}\leq {C}A^{1/2} \lambda^{1/4}$, we easily obtain \begin{equation*}
	\begin{split}
		&\int_{0}^{T} 	\left| \int_{ \{ x_{1} = - D \} } ( v_{Q,1} + \tld{v}_{1} - \dot{p}_{1} ) \tld{\psi} \rho \, dx_{2} \right| dt \leq {C}A^{1/2} \lambda^{1/4} \int_{0}^{T} 	\left| \int_{ \{ x_{1} = - D \} } \underbrace{(\dot{p}_{1} -( v_{Q,1} + \tld{v}_{1})   )}_{>0}x_2 |\rho| \, dx_{2} \right| dt \\ &\le {C}A^{1/2} \lambda^{1/4}\left( \int_{0}^{T} I_{\mu}^{(+)}  + I_{\mu}^{(-)} dt \right) \leq  {C}A^{1/2} \lambda^{1/4}\cdot(20\lambda) \leq \frac{\lambda}{100},
	\end{split}
\end{equation*}
where we recall the definition 
\eqref{defn_I_mu}		
and used \eqref{eq:I-mu-minus} and \eqref{eq:I-mu-plus}.\\ 

To estimate the other term, we begin by noting that 
\begin{equation*}
	\begin{split}
		\tld{F}_{E}(t) &:= \int_{\bbR^2_+ } ( \nb\tld{\psi} \cdot v ) \tld{\rho} dx =  \int_{\bbR^2_+ } ( \nb\tld{\psi} \cdot v_{Q} ) \tld{\rho} dx, 
	\end{split}
\end{equation*} 
using the cancellations  
\begin{equation*}
	\begin{split}
		\int_{\bbR^2_+ } ( \nb\tld{\psi} \cdot \tld{v} ) \tld{\rho} \,  dx =-\int_{\bbR^2_+ } ( \nb\tld{\psi} \cdot \nb^\perp\tld{\psi} ) \tld{\rho} \,  dx = 0  
	\end{split}
\end{equation*} and  
\begin{equation*}
	\begin{split}
		\int_{\bbR^2_+ } ( \nb\tld{\psi} \cdot ( \dot{p}_{1} \bfe_{1} ) ) \tld{\rho}  \, dx= -\dot{p}_{1}(t)\int_{\bbR^2_+ } \tld{v}_2  \tld{\rho}  \, dx =0.
	\end{split}
\end{equation*} 
We decompose $\tld{F}_{E}(t) =\tld{F}_{E, 1}(t) + \tld{F}_{E,2}(t)$
where  \begin{equation*}
	\begin{split}
		\tld{F}_{E, j}(t) :=  \int_{\bbR^2_+ }  \rd_{j}\tld{\psi} \,  v_{Q, j} \, \tld{\rho} \, dx. 
	\end{split}
\end{equation*}

\medskip

\noindent \textbf{Estimate of $\tld{F}_{E, 2}(t)$}. This is parallel to estimating $F_\mu$ in Lemma \ref{lem:F-mu}: after expanding $v_{Q,2}$ as an integral against $\rho_{Q}$, we have \begin{equation*}
	\begin{split}
		\tld{F}_{E, 2}(t)  =  \frac1{2\pi}\int_{}  \int_{ \{ -D \le y_{1}  \} } \rd_{2}\tilde{\psi}(t,x)  \frac{x_1-y_1}{|x-y|^{2}} \frac{4x_2y_2}{|x-\bar{y}|^2} \rho_{Q}(t,y) \tld{\rho}(t,x) dy dx .
	\end{split}
\end{equation*} Simply using that $\nrm{  \rd_{2}\tilde{\psi}  }_{L^\infty} \le \nrm{\tld{v}}_{L^\infty}\leq
{C}A^{1/2} \lambda^{1/4}		
$, we are in the same situation as in
the proof for the estimate for 	for $F_{\mu}$	in Lemma \ref{lem:F-mu}. Thus we decompose
$$\tld{F}_{E, 2}(t) =\tld{F}_{E, 2}^{main}(t) +\tld{F}_{E, 2}^{err}(t)\quad\mbox{from}\quad \rho_{Q}(t)=\rho^{main}_{Q}(t)+\rho^{err}_{Q}(t) $$ and
follow the exactly same way to get
\begin{equation*}
	\begin{split}
		\int_{0}^{T} \left| \tld{F}_{E, 2}(t)  \right| dt \leq
		{C}A^{1/2} \lambda^{1/4}\cdot\left(\frac{\lambda}{20}+\frac{\lambda}{20}\right)				\leq
		\frac{ \lambda}{100}. 
	\end{split}
\end{equation*}

\medskip

\noindent \textbf{Estimate of $\tld{F}_{E, 1}(t)$}. Recalling the definition \eqref{eq:psi-G}, 
we have 
\begin{equation*}
	\begin{split}
		\tld{F}_{E, 1}(t) = -\iint_{\bbR^2_+ \times \bbR^2_+}
		\bfK_1(x,y)
		(\rd_{1}\tld{\psi} \, \tld{\rho})(t,x) \, \rho_{Q}(t,y) \, dxdy. 
	\end{split}
\end{equation*} Using
$ \nrm{ \rd_{1}\tld{\psi} }\leq\|\tilde{v}\|_{L^\infty}\leq {C}A^{1/2} \lambda^{1/4}$ and \eqref{eq:K-decay},		
we   estimate \begin{equation*}
	\begin{split}
		&  \left| \tld{F}_{E, 1}(t)  \right|   \le {C}A^{1/2} \lambda^{1/4} \iint_{\bbR^2_+ \times \bbR^2_+} \frac{y_{2}}{ |x-y||x-\bar{y}|}  |\tld{\rho}(t,x)| \, |\rho_{Q}(t,y)| \, dxdy   \\
		& = {C}A^{1/2} \lambda^{1/4} \int_{\bbR^2_+ } |\tld{\rho}(t,x)|\left( \int_{\bbR^2_+ }\frac{y_{2}}{ |x-y||x-\bar{y}|}  \, |\rho_{Q}(t,y)| \, dy \right) dx\\
		& \le  {C}A^{1/2} \lambda^{1/4}\left(G^{main}(t)+G^{err}(t)\right),
	\end{split}
\end{equation*}  where 
$$
G^{\star}(t):=\int_{\bbR^2_+ } |\tld{\rho}(t,x)|\left( \int_{\bbR^2_+ }\frac{y_{2}}{ |x-y||x-\bar{y}|}  \, |\rho^{\star}_{Q}(t,y)| \, dy \right) dx,\quad \star\in\{main, err\}.
$$
Again, we follow the same idea we used to estimate $F_\mu^{main}$ and $F_\mu^{err}$ in the proof of Lemma \ref{lem:F-mu}		
to get 
\begin{equation*}
	\begin{split}
		\int_{0}^{T} \left| \tld{F}_{E, 1}(t)  \right|  dt \le {C}A^{1/2} \lambda^{1/4}\cdot\left(\frac{\lambda}{20}+\frac{\lambda}{20}\right)\leq \frac{\lambda}{100}.
	\end{split}
\end{equation*}
The proof is complete.  \end{proof}


\subsection{Extending and closing (B1): Finding time-dependent center}\label{subsec:1}
We define \begin{equation*}
\begin{split}
	\hat Q(t):=\begin{cases}
		\{x\in\mathbb R^2_+\ :\ x_1>p_1(t)-D\},\quad &\ \mbox{ for }\ t\in[0,T],\\
		\{x\in\mathbb R^2_+\ :\ x_1>p_1(T)-D\}, &\ \mbox{ for } \ t>T,
	\end{cases}
\end{split}
\end{equation*} and write $\hat \omega(t,x) := \omega(t,x) \mathbf{1}_{\hat Q(t)}(x)$. We already have, for $t\in [0,T]$,
\begin{equation*}
\begin{split}
	\nrm{\hat \omega}_{L^2}\leq \kpp[\omega_{Lamb}]+\frac{\lambda}{2},\quad 
	{\nrm{\hat \omega}_{L^1_*}\leq \mu[\omega_{Lamb}]+\frac{\lambda}{2},\quad }
	E[\hat \omega]\geq E[\omega_{Lamb}]-\frac{\lambda}{2},\quad 
	\nrm{\omega-\hat \omega}_{L^2}\leq 30 \lambda^{1/2},
\end{split}
\end{equation*} where the last inequality comes from 
\eqref{eq:rem-small}.
By continuity in time of these quantities, we can find a sufficiently small $0<\eta_1\ll 1$ such that for $t\in [0,T+\eta_1]$, 
\begin{equation*}
\begin{split}
	\nrm{\hat \omega}_{L^2}\leq \kpp[\omega_{Lamb}]+ \lambda,\quad 
	\nrm{\hat \omega}_{L^1_*}\leq \mu[\omega_{Lamb}]+\lambda,\quad 
	E[\hat \omega]\geq E[\omega_{Lamb}]-\lambda,\quad 
	\nrm{\omega-\hat \omega}_{L^2}\leq 60 \lambda^{1/2}.
\end{split}
\end{equation*}
Recall the constant $M_0$ in \eqref{defn_M_0}.
Then, we take $\lambda$ smaller (if necessary) to guarantee $\lambda<\lambda_1$ where $\lambda_1>0$ is the constant from Proposition \ref{thm:AC2019} with the choice
$$
\kappa=\kappa[\omega_{Lamb}], \quad \mu=\mu[\omega_{Lamb}],\quad M=10A, \quad \varep_1=\varep/M_0.
$$ Then applying the proposition
to $\hat\omega$,
we get a function $\tau:[0,T+\eta_1]\to\mathbb{R}$ 
such that 
\begin{equation}
\nrm{\hat \omega(t,\cdot)-\omega_{Lamb}(\cdot-\tau(t)\bfe_1)}_{L^2\cap L^1_*(\mathbb R_+^2)}\leq \frac{\eps}{M_0},\quad \mbox{ for all }\ t\in [0,T+\eta_1].
\end{equation} 
From now on, we consider $t \in [0,T+\eta_1]$ unless stated otherwise.
We define $\omega^{\tau,(+)}_{rem}(t,x) := \omega^{(+)}(t,x)-\omega_{Lamb}(x-\tau(t)\bfe_{1})$ and estimate 
\begin{equation*}
\begin{split}
	\nrm{ \omega^{\tau,(+)}_{rem}(t,\cdot)  }_{L^2(\mathbb R_+^2)} &
	\leq 	\nrm{  \omega(t,\cdot)-\omega_{Lamb}(\cdot-\tau(t)\bfe_1)}_{L^2(\mathbb R_+^2)}\\&\leq \nrm{  \hat \omega(t,\cdot)-\omega_{Lamb}(\cdot-\tau(t)\bfe_1)}_{L^2(\mathbb R_+^2)}+ \nrm{  \hat \omega(t,\cdot)-\omega(t,\cdot)}_{L^2(\mathbb R_+^2)}\\
	&	\leq \frac{\eps}{M_0}+60 \lambda^{1/2}.
\end{split}
\end{equation*}

To estimate $\nrm{ \omega^{\tau,(+)}_{rem}(t,\cdot)  }_{L^1(\mathbb R_+^2)}$, we first
note 
$
|\tau(0)|<10
$
from the initial condition, which implies
that $\omega_{Lamb}(\cdot-\tau(0)\bfe_1)$ is supported in $S_R$. Thus, by H\"older, we compute
\begin{equation}\label{eq:tau0}
\begin{split}&
	\nrm{  \omega^{(+)}(0,\cdot)-\omega_{Lamb}(\cdot-\tau(0)\bfe_1)}_{L^1(\mathbb R_+^2)}\\&\quad\quad\leq \sqrt{2R^2}\cdot \nrm{  \omega^{(+)}(0,\cdot)-\omega_{Lamb}(\cdot-\tau(0)\bfe_1)}_{L^2(S_R)}+\nrm{  \omega^{(+)}(0,\cdot)}_{L^1(\mathbb R_+^2\setminus S_R)}\\&\quad\quad\leq 2R  \left(\frac{\varep}{M_0}+60\lambda^{1/2}\right)+\delta,
\end{split}
\end{equation} which gives, by Lemma \ref{L2_to_L1},
\begin{equation}\label{eq:postive-lamb_L1_fin}
\begin{split}
	&\nrm{ \omega^{\tau,(+)}_{rem}(t,\cdot)  }_{L^1(\mathbb R_+^2)} 
	\\& \quad \leq
	4\nrm{  \omega^{(+)}(t,\cdot)-\omega_{Lamb}(\cdot-\tau(t)\bfe_1)}_{L^2(\mathbb R_+^2)} +\nrm{  \omega^{(+)}(0,\cdot)-\omega_{Lamb}(\cdot-\tau(0)\bfe_1)}_{L^1(\mathbb R_+^2)}\\
	& \quad 	\leq 4\left(\frac{\eps}{M_0}+60 \lambda^{1/2}\right)+2R \left(\frac{\varep}{M_0}+60\lambda^{1/2}\right)+\delta\leq \frac{3R\varep}{M_0}+\delta+200R\lambda^{1/2}.
\end{split}
\end{equation} 
Now, recalling the definition \eqref{eq:h}, we compute 
\begin{equation*}
\begin{split}
	h(t, \tau(t))&=\int_{\mathbb R_+^2} \omega^{(+)}(t,x+\tau(t)\bfe_1)g(x_1) dx =\int_{\mathbb R_+^2} \omega_{Lamb} (x)g(x_1) dx+\int_{\mathbb R_+^2} \omega^{\tau,(+)}_{rem}(t,x+\tau(t)\bfe_1)g(x_1) dx\\
	&=\int_{\mathbb R_+^2} \omega^{\tau,(+)}_{rem}(t,x+\tau(t)\bfe_1)g(x_1) dx,
\end{split}
\end{equation*}
where we have used the odd symmetry of $g$ to deduce $\int_{\mathbb R_+^2} \omega_{Lamb} (x)g(x_1) dx=0$. Then it follows from  \eqref{eq:postive-lamb_L1_fin} that 
\begin{equation}\label{eq:H-tau}
|h(t, \tau(t))|
\leq \left( \frac{3R\varep}{M_0}+\delta+200R\lambda^{1/2}\right) \|g\|_{L^\infty}	
\leq \left( \frac{12R \varep}{M_0}+4\delta+800R\lambda^{1/2}\right).
\end{equation}
For points $(t,p)$ with $p\in (\tau(t)-2, \tau(t)+2)$, we compute 
\begin{equation*}
\begin{split}
	\partial_p h(t, p)&=-\int_{\mathbb R_+^2} \omega^{(+)}(t,x+p\bfe_1)g'(x_1) dx \\&=-\int_{\mathbb R_+^2} \omega_{Lamb} (x+(p-\tau(t))\bfe_1)g'(x_1)  dx-\int_{\mathbb R_+^2} \omega^{\tau,(+)}_{rem}(t,x+p\bfe_1)g'(x_1) dx\\& =-\int_{\mathbb R_+^2} \omega_{Lamb} (x)  dx-\int_{\mathbb R_+^2} \omega^{\tau,(+)}_{rem}(t,x+p\bfe_1)g'(x_1) dx.
\end{split}
\end{equation*}
Using the estimate \eqref{eq:postive-lamb_L1_fin} again, we infer from the above identity that
\begin{equation}\label{eq:H-deriv}
|\partial_p h(t, p)+m_L|\leq \left( \frac{3R\varep}{M_0}+\delta+200R\lambda^{1/2}\right) \|g'\|_{L^\infty}
\leq \left( \frac{12R\varep}{M_0}+4\delta+800R\lambda^{1/2}\right). 
\end{equation}
By choosing $\eps, \delta,\lambda$  smaller  if necessary to have
$$ \frac{12R\varep}{M_0}\leq \frac{m_L}{10},\quad 4\delta\leq\frac{m_L}{10},\quad\mbox{and}\quad 800R\lambda^{1/2}\leq \frac{m_L}{10},\quad \mbox{respectively},
$$
\eqref{eq:H-deriv} yields
\begin{equation}\label{eq:H-deriv-upbound}
-\frac{3m_L}{2}\leq\partial_p h(t, p)<-\frac{m_L}{2}\quad \text{for all}\quad p\in (\tau(t)-2, \tau(t)+2).
\end{equation}
Combining \eqref{eq:H-tau} and \eqref{eq:H-deriv-upbound}, we use the Implicit Function Theorem to find a unique $\bar{p}_1\in C^1([0,T+\eta_1])$ (by making $\eta_1>0$ smaller if necessary) such that
$$\bar{p}_1(t)\in \left[\tau(t)-\frac{ 2}{m_L}\cdot \left( \frac{12R\varep}{M_0}+4\delta+800R\lambda^{1/2}\right)	, \tau(t)+\frac{2}{m_L}\cdot  \left( \frac{12R\varep}{M_0}+4\delta+800R\lambda^{1/2}\right) \right]$$
and $h(t,\bar{p}_1(t))=0$ for all $t\in [0,T+\eta_1].$ We observe $$\bar{p}_1(t)=p_1(t)\quad \mbox{for all}\quad t\in[0,T]$$ where $p_1(\cdot)$ is the function from \eqref{eq:B1}. From now on, we simply denote $p_1(\cdot)$ instead of $\bar{p}_1(\cdot)$. Note that it is differentiable with its derivative is given by \eqref{eq:deriv-p}. \\

The Lipschitz continuity of the Lamb dipole with the Lipschitz constant $C_{Lip}$ gives
\begin{equation}\label{diff-positve-lamb}
\begin{split}
	&\nrm{\omega^{(+)}(t,\cdot)-\omega_{Lamb}(\cdot-p_1(t)\bfe_{1})}_{L^2}\leq \nrm{\omega(t,\cdot)-\omega_{Lamb}(\cdot-p_1(t)\bfe_{1})}_{L^2}\\&\leq \nrm{\omega(t,\cdot)-\omega_{Lamb}(\cdot-\tau(t)\bfe_{1})}_{L^2}+\nrm{\omega_{Lamb}(\cdot)-\omega_{Lamb}(\cdot-(\tau(t)-p_1(t))\bfe_{1})}_{L^2}\\&\leq \big(\frac{\eps}{M_0}+60 \lambda^{1/2}\big)+C_{Lip} \sqrt{\pi} |\tau(t)-p_1(t)|\\
	&\leq \big(\frac{\eps}{M_0}+60 \lambda^{1/2}\big)+C_{Lip} \sqrt{\pi}\cdot \frac{ 2}{m_L}\cdot \big( \frac{12R\varep}{M_0}+4\delta+800R\lambda^{1/2}\big)	
	\\
	&\leq \big(\frac{10C_{Lip}}{m_L}+1\big) \big( \frac{12R\varep}{M_0}+4\delta+800R\lambda^{1/2}\big)	
\end{split}
\end{equation}
and 
\begin{equation}\label{diff-positve-lamb-impluse}
\begin{split}
	&\nrm{\hat\omega(t,\cdot)-\omega_{Lamb}(\cdot-p_1(t)\bfe_{1})}_{L^1_*} \\ &\leq \nrm{\hat \omega (t,\cdot)-\omega_{Lamb}(\cdot-\tau(t)\bfe_{1})}_{L^1_*}+\nrm{\omega_{Lamb}(\cdot)-\omega_{Lamb}(\cdot-(\tau(t)-p_1(t))\bfe_{1})}_{L^1_*}\\&\leq \frac{\eps}{M_0}+C_{Lip} \pi |\tau(t)-p_1(t)| \leq \frac{\eps}{M_0}+C_{Lip} \pi\cdot \frac{ 2}{m_L}\cdot \big( \frac{12R\varep}{M_0}+4\delta+800R\lambda^{1/2}\big)\\
	&\leq \big(\frac{10C_{Lip}}{m_L}+1\big)   \big( \frac{12R\varep}{M_0}+4\delta+800R\lambda^{1/2}\big)	.
\end{split}
\end{equation}
Next, we recall  \eqref{eq:deriv-p}: $ 	\dot{p}_{1}(t) = \frac{(\partial_t  h)(t, p_1(t))}{(-\partial_p h)(t, p_1(t))}.$  Since we have already computed $(\partial_p h)(t, p_1(t))$ in \eqref{eq:H-deriv}, it remains to estimate its numerator $(\partial_t  h)(t, p_1(t))$. As in \eqref{eq:deriv-t-H_}, we compute 
\begin{equation}\label{eq:deriv-t-H}
\begin{split}
	(\partial_t  h)(t, p_1(t))&=\int_{\mathbb{R}^2_{+}} u_1(t, x+p_1(t)\bfe_1)\omega^{(+)} (t, x+p_1(t)\bfe_1) g'(x_1) dx\\
	&=\int_{\mathbb{R}^2_{+}} u_{Lamb, 1}(  x )\omega_{Lamb} (  x ) g'(x_1) dx\\
	&\quad+\int_{\mathbb{R}^2_{+}} [u_1(t, x+p_1(t)\bfe_1)-u_{Lamb, 1}(  x )] \omega^{(+)} (t, x+p_1(t)\bfe_1) g'(x_1) dx\\
	&\quad+\int_{\mathbb{R}^2_{+}} u_{Lamb, 1}(  x ) [\omega^{(+)}(t, x+p_1(t)\bfe_1) -\omega_{Lamb} ( x ) ]g'(x_1) dx. 
\end{split}
\end{equation}
As in \eqref{ml_integral}, the first term on the right hand side of \eqref{eq:deriv-t-H} is equal to $m_L$.
The last two terms on the right hand side of \eqref{eq:deriv-t-H} can be   estimated using \eqref{diff-positve-lamb}.
Indeed, as in \eqref{eq:tau0}, we first estimate,  by \eqref{diff-positve-lamb} and by noting $|p_1(0)|\leq 10$,
\begin{equation*}
\begin{split}
	\nrm{\omega_0(\cdot)-\omega_{Lamb}(\cdot-p_1(0)\bfe_{1})}_{L^1}&\leq \sqrt{2R^2}\nrm{\omega_0(\cdot)-\omega_{Lamb}(\cdot-p_1(0)\bfe_{1})}_{L^2(S_R)}+\nrm{\omega_0(\cdot)}_{L^1(\mathbb{R}^2_+\setminus S_R)}\\
	&{\leq \sqrt{2R^2}\cdot  \left(\frac{10C_{Lip}}{m_L}+1\right)\cdot \left( \frac{12R\varep}{M_0}+4\delta+800R\lambda^{1/2}\right)	+\delta}\\	
	&{\leq 2R\cdot  \left(\frac{10C_{Lip}}{m_L}+1\right)\cdot \left( \frac{12R\varep}{M_0}+4\delta+800R\lambda^{1/2}\right).}\\	
\end{split}
\end{equation*}
Then, by Lemma \ref{L2_to_L1} and by \eqref{diff-positve-lamb} again, we get
\begin{equation}\label{eq:postive-lamb_l1}
\begin{split}
	& \nrm{\omega(t,\cdot+p_1(t)\bfe_{1})-\omega_{Lamb}(\cdot)}_{L^1}\leq 4 \nrm{\omega(t,\cdot+p_1(t)\bfe_{1})-\omega_{Lamb}(\cdot)}_{L^2}+ \nrm{\omega_0(\cdot+p_1(0)\bfe_{1})-\omega_{Lamb}(\cdot)}_{L^1}\\
	&{\leq  3R\left(\frac{10C_{Lip}}{m_L}+1\right)\cdot \left( \frac{12R\varep}{M_0}+4\delta+800R\lambda^{1/2}\right).}\\
\end{split}
\end{equation}
From \eqref{eq:vel-L-infty-alpha2} with $\alpha=1/2$,
we estimate the corresponding velocity by \eqref{diff-positve-lamb} and by \eqref{eq:postive-lamb_l1}:
\begin{equation*}
\begin{split}
	& \nrm{u(t,\cdot+p_1(t)\bfe_{1})-u_{Lamb}(\cdot)}_{L^\infty}\\&\leq 	 C\nrm{\omega(t,\cdot+p_1(t)\bfe_{1})-\omega_{Lamb}(\cdot)}_{L^2}^{1/2} \cdot(2A)^{1/4}\cdot 
	\nrm{\omega(t,\cdot+p_1(t)\bfe_{1})-\omega_{Lamb}(\cdot)}_{L^1}^{1/4}
	\\
	&\leq  CA^{1/4}\left[{R^{1/3}}\left(\frac{10C_{Lip}}{m_L}+1\right)\cdot \left( \frac{12R\varep}{M_0}+4\delta+800R\lambda^{1/2}\right)	\right]^{3/4}.
\end{split}
\end{equation*}
Thus we  estimate 
the second integral on the right-hand side of \eqref{eq:deriv-t-H}:
\begin{equation*}
\begin{split}
	&\left| \int_{\mathbb{R}^2_{+}} [u_1(t, x+p_1(t)\bfe_1)-u_{Lamb, 1}(  x )] \omega^{(+)} (t, x+p_1(t)\bfe_1) g'(x_1) dx\right|\\&\leq	\nrm{u(t,\cdot)-u_{Lamb}(\cdot-p_1(t)\bfe_{1})}_{L^\infty}\cdot\|\omega(t,\cdot)\|_{L^1}\cdot\|g'\|_{L^\infty}\\
	&\leq  CA^{1/4}\left[{R^{1/3}}\left(\frac{10C_{Lip}}{m_L}+1\right)\cdot \left( \frac{12R\varep}{M_0}+4\delta+800R\lambda^{1/2}\right)	\right]^{3/4}.
\end{split}
\end{equation*}
For the last integral on the right-hand side of \eqref{eq:deriv-t-H}, we simply get, by \eqref{diff-positve-lamb}, 
\begin{equation*}
\begin{split}
	& \left|\int_{\mathbb{R}^2_{+}} u_{Lamb, 1}(  x ) [\omega^{(+)}(t, x+p_1(t)\bfe_1) -\omega_{Lamb} ( x ) ]g'(x_1) dx \right|\\&\leq \|u_{Lamb}\|_{L^2}\cdot
	\nrm{\omega^{(+)}(t,\cdot)-\omega_{Lamb}(\cdot-p_1(t)\bfe_{1})}_{L^2}
	\cdot\|g'\|_{L\infty}\\
	&\leq 100\cdot \left(\frac{10C_{Lip}}{m_L}+1\right)\cdot \left( \frac{12R\varep}{M_0}+4\delta+800R\lambda^{1/2}\right).	
\end{split}
\end{equation*}
In sum,  we 
get the following result:
$$|(\partial_t  h)(t, p_1(t))-m|\leq   CA^{1/4}\left[{R^{1/3}}\left(\frac{10C_{Lip}}{m_L}+1\right)\cdot \left( \frac{12R\varep}{M_0}+4\delta+800R\lambda^{1/2}\right)	\right]^{3/4}\leq C A^{1/4}\varep^{3/4}$$ by making $\lambda, \dlt$ small. 
Recalling 
\eqref{eq:deriv-p} and
combining this bound and \eqref{eq:H-deriv} gives 
\begin{equation}\label{eq:deriv-t-p}
\left|\frac{d}{dt} p_1(t)-1 \right| \leq  CA^{1/4}\eps^{3/4}\leq \eps^{1/2}\quad \text{for all}\quad t\in [0, T+\eta_1].
\end{equation} 
by taking $\eps$ smaller if necessary.
A simple corollary of \eqref{eq:deriv-t-p} is that $p_1(t)$ is a strictly increasing function and therefore $p_1(t)>p_1(T)$ provided $t\in (T, T+\eta_1]$. This means $$\{x\in\mathbb R^2_+\ :\ x_1>p_1(t)-D\}\subset \{x\in\mathbb R^2_+\ :\ x_1>p_1(T)-D \}.$$ Therefore, for any $t\in [0, T+\eta_1]$, we can deduce from \eqref{diff-positve-lamb} and \eqref{diff-positve-lamb-impluse} that 
\begin{equation}\label{eq:closing-B2-diff-Lamb}
\begin{split}
	&\nrm{\rho_Q-\omega_{Lamb}}_{L^2\cap L^1_*}
	\\
	&\leq \nrm{\omega(t,\cdot)-\omega_{Lamb}(\cdot-p_1(t)\bfe_{1})}_{L^2}+\nrm{\hat\omega(t,\cdot)-\omega_{Lamb}(\cdot-p_1(t)\bfe_{1})}_{L^1_*}\\
	&\leq  2 \big(\frac{10C_{Lip}}{m_L}+1\big) \big( \frac{12R\varep}{M_0}+4\delta+800R\lambda^{1/2}\big)	\leq \frac{\eps}{240}+\frac{\eps}{240}+\frac{\eps}{240}=\frac{\eps}{80}<\frac{\varep}{20}.
\end{split}
\end{equation}
Here, we used the definition of $M_0$ in \eqref{defn_M_0} and assumed $\lambda, \dlt$ small. 
The estimate \eqref{eq:closing-B2-diff-Lamb} holds for any $t\in[0,T+\eta_1]$ which closes \eqref{eq:B1}.

\subsection{Proof of the main theorem}\label{subsec:proof-main}

We are ready to complete the proof of the main theorem. 

\begin{proof}[Proof of Theorem \ref{thm:main}]
Under the bootstrap assumptions \eqref{eq:B1}--\eqref{eq:B5} for $t\in [0,T]$,  we already have proved  \eqref{eq:B1}  for all $t\in [0, T+\eta_1]$. Using the continuity of $p_1$ in $t$, \eqref{closing-B2}, \eqref{closing-B3}, \eqref{eq:B4-close} and  \eqref{closing-B5}, we can find $  \eta_2\in (0,   \eta_1)$ such that  \eqref{eq:B2}--\eqref{eq:B5} still hold for all $t\in [0, T+\eta_2]$. Therefore, by this extension property and the continuity of \eqref{eq:B1}--\eqref{eq:B5} in $t$,  we conclude that  \eqref{eq:B1}--\eqref{eq:B5} hold for all $t\in [0,+\infty)$. Indeed, let $$T_{\max} := \sup\{\,T \ge T_0 : \eqref{eq:B1}\mbox{--}\eqref{eq:B5} \ \text{hold  on}\ [0,T]\,\}.$$ Here $T_0>0$ is defined in Proposition \ref{pro:starting time}. 
If $T_{\max} < \infty$, the continuity yields \eqref{eq:B1}--\eqref{eq:B5} on $[0,T_{\max}]$. 
Then extension property gives an $\eta > 0$ such that \eqref{eq:B1}--\eqref{eq:B5} hold  on $[0,T_{\max}+\eta]$, contradicting the definition of $T_{\max}$. Hence $T_{\max} = \infty$.
Then, we use  Lemma \ref{L2_to_L1} with \eqref{eq:closing-B2-diff-Lamb} and \eqref{est_p_1(0)} to get
$$
\nrm{\rho_Q-\omega_{Lamb}}_{L^1}\leq\frac{4\eps}{80}+C\delta\leq \frac{\varep}{10}.
$$
Finally, by \eqref{eq:rem-small}, we get for any $t\geq 0$,
$$	\nrm{\omg(t,\cdot + p_{1}(t) \bfe_{1}) -  {\omg}_{Lamb}( \cdot  )}_{(L^{2}\cap L^1)(\ohp) \cap L^{1}_{*}{( \{ x_{1} \geq - C_0\lambda^{-1/2} \}  )}} < \varep,$$ which  finish the proof of Theorem \ref{thm:main} by recalling \S \ref{subsec:redu}.  \end{proof}

\bibliographystyle{abbrv}

\end{document}